\documentclass{article}

\usepackage{makeidx}
\usepackage{mathrsfs}
\usepackage{setspace}
\usepackage[totalwidth=14cm,totalheight=24cm]{geometry}
\usepackage[T1]{fontenc}
\usepackage[dvips]{graphicx}
\usepackage{epsfig}
\usepackage{amsmath,amssymb,amscd,amsthm}
\usepackage{subfigure}
\usepackage[latin1]{inputenc}
\usepackage{latexsym}
\usepackage{graphics}
\usepackage{colortbl}
\usepackage{color}
\usepackage{ae}
\usepackage{enumitem}
\usepackage{textcomp}
\begin{document}

\setcounter{secnumdepth}{3}
\setcounter{tocdepth}{2}

\newtheorem{definition}{Definition}[section]
\newtheorem{lemma}[definition]{Lemma}
\newtheorem{sublemma}[definition]{Sublemma}
\newtheorem{corollary}[definition]{Corollary}
\newtheorem{proposition}[definition]{Proposition}
\newtheorem{theorem}[definition]{Theorem}

\newtheorem{remark}[definition]{Remark}
\newtheorem{example}[definition]{Example}

\newcommand{\cov}{\mathrm{covol}}
\def \tr{{\mathrm{tr}}}
\def \det{{\mathrm{det}\;}}
\def\co{\colon\tanhinspace}
\def\I{{\mathcal I}}
\def\N{{\mathbb N}}
\def\R{{\mathbb R}}
\def\Z{{\mathbb Z}}
\def\Sph{{\mathbb S}}
\def\Tor{{\mathbb T}}
\def\Disk{{\mathbb D}}
\def\Hess{\mathrm{Hess}}
\def\rad{\mathbf{v}}

\def\H{{\mathbb H}}
\def\RP{{\mathbb R}{\mathrm{P}}}
\def\dS{{\mathrm d}{\mathbb{S}}}
\def\Isom{\mathrm{Isom}}

\def\sh{\mathrm{sinh}\,}
\def\ch{\mathrm{cosh}\,}
\newcommand{\arccosh}{\mathop{\mathrm{arccosh}}\nolimits}
\newcommand{\oh}{\overline{h}}

\newcommand{\mf}{\mathfrak}
\newcommand{\mb}{\mathbb}
\newcommand{\ol}{\overline}
\newcommand{\la}{\langle}
\newcommand{\ra}{\rangle}
\newcommand{\hess}{\mathrm{Hess}\;}
\newcommand{\grad}{\mathrm{grad}}
\newcommand{\M}{\mathrm{MA}}
\newcommand{\II}{\textsc{I\hspace{-0.05 cm}I}}
\renewcommand{\d}{\mathrm{d}}
\newcommand{\A}{\mathrm{A}}
\renewcommand{\L}{\mathcal{L}}

%%%%%%%%%%%%%%%%%%%%%%%%%%%%%%%%%%%%

%\newcommand{\newnew}[1]{{\color{green}#1}}
\newcommand{\newnew}[1]{#1}
\newcommand{\oldold}[1]{}
\newcommand{\old}[1]{}
\newcommand{\new}[1]{#1}
\newcommand{\note}[1]{{\color{green}{\small #1}}}
%\newcommand{\note}[1]{}

%%%%%%%%%%%%%%%%%%%%%%%%%%%%%%%%%%%%%%%%
% NEW MACROS                                                                                                %
%%%%%%%%%%%%%%%%%%%%%%%%%%%%%%%%%%%%%%%%

\newcommand{\Area}{\mathrm{Area}}
\newcommand{\mink}{\R^{d,1}}     
\newcommand{\gauss}{\mathcal{G}}    % Gauss map
\newcommand{\fconv}{\mathbf{K}}       % F-convex set 
\newcommand{\jplus}{\mathrm{J}^+}    % causal future (J^+(H))
\newcommand{\ball}{\mathbb B}           % unit ball 
\newcommand{\hor}{P}                          % horizontal plane
\newcommand{\fut}{\operatorname{I}^+}   % future
\newcommand{\proj}{\pi}                        %projection R^{d,1}-->R^d  
\newcommand{\plane}[2]{P_#1(#2)}      % P_f(v) hyperplane of equation <Y,v>=f(v)
\newcommand{\fund}{\mathcal P}          % fundamental domain
\newcommand{\poly}{\mathscr P}          % polyhedron in the equivariant Pogorelov example
\newcommand{\const}{\zeta}                  % constant in the appendix
\newcommand{\qd}{\hfill{\scriptsize $\square$}}   % end of the proof of a fact
%%%%%%%%%%%%%%%%%%%%%%%%%%%%%%%%%%%
%%%%%%%% gere espace texte formule
\setlength{\abovedisplayshortskip}{1pt}
\setlength{\belowdisplayshortskip}{3pt}
\setlength{\abovedisplayskip}{3pt}
\setlength{\belowdisplayskip}{3pt}

%%%%%%%%%%%%%%%%%%%%%%%%%%%%%%%%%%

\title{The equivariant Minkowski problem in Minkowski space}

\author{Francesco Bonsante and Fran\c{c}ois Fillastre}

\date{ \today}

\maketitle

Universit\`a degli Studi di Pavia, Via Ferrata, 1, 27100 Pavia, Italy

Universit\'e de Cergy-Pontoise, UMR CNRS 8088, F-95000 Cergy-Pontoise, France

francesco.bonsante@unipv.it, francois.fillastre@u-cergy.fr

\begin{abstract}

The classical Minkowski problem in Minkowski space asks, for a positive function
$\phi$ on $\H^d$, for a convex set  $\fconv$ 
in Minkowski space  with $C^2$ space-like boundary $S$,
such that $\phi(\eta)^{-1}$ is the Gauss--Kronecker curvature at the point with normal $\eta$.
Analogously to the Euclidean case, it is possible to formulate a weak version of this problem:
 given a Radon measure $\mu$ on $\H^d$  the generalized Minkowski problem in Minkowski space 
asks for a convex subset $\fconv$ such that the area measure of $\fconv$ is $\mu$.

In the present paper we look at an equivariant version of the problem: given 
a uniform lattice $\Gamma$ of isometries of $\H^d$, given a $\Gamma$ invariant Radon measure $\mu$, given 
a isometry group $\Gamma_{\tau}$ of Minkowski space, with $\Gamma$ as linear part, there exists a unique 
convex set with area measure $\mu$, invariant under the action of $\Gamma_{\tau}$.
 The proof uses a functional which is the covolume associated to every invariant convex set.

This result translates as a solution of the Minkowski problem in flat space times with compact hyperbolic Cauchy surface.
The uniqueness part, as well as regularity results,  follow from properties of the Monge--Amp\`ere equation. The existence part can be translated as 
an existence  result for Monge--Amp\`ere equation.

The regular version  was proved by
T.~Barbot, F.~B\'eguin and A.~Zeghib for $d=2$ and by V.~Oliker and U.~Simon for $\Gamma_{\tau}=\Gamma$. 
Our method is totally different. Moreover, we show that those cases are  very specific:
in general, 
there is no  smooth $\Gamma_\tau$-invariant hypersurface of constant Gauss-Kronecker curvature equal to $1$.

\end{abstract}



\setcounter{tocdepth}{3}
\tableofcontents

\section{Introduction}

In this paper we study an adapted version of 
classical Minkowski problem in Minkowski space and more generally
in globally hyperbolic flat Lorentzian manifolds.

Let us recall that Minkowski space $\mathbb R^{d,1}$ is the  manifold $\mathbb R^{d+1}$ equipped with the flat
Lorentzian structure given by Minkowski product: 
$$\langle X,Y\rangle_-=\sum_{i=1}^d X_iY_i-X_{d+1}Y_{d+1}~.$$
\index{$\langle,\cdot,\cdot\rangle_-$ Minkowski product}

The pseudo-sphere is the subset of $\mathbb R^{d,1}$ of unit future oriented time-like vectors 
$$\H^d=\{X\in\R^{d+1}| \langle X,X\rangle_-=-1, X_{d+1}>0 \}~.$$\index{$\H^d$ Hyperbolic space}
It is well-known that it is a space-like (i.e., the restriction of $\langle\cdot,\cdot\rangle_-$ to $T\H^d$ is positive) 
convex hypersurface in $\R^{d+1}$ isometric to the hyperbolic space.

Let $S$ be the boundary of a strictly convex set $\fconv$ of $\mink$.  Suppose that $S$ is $C^2$, and space-like. 
This assumption forces $\fconv$ to be unbounded: in fact either
$\fconv$ coincides with the future of $S$ or with its past. We will always assume that $\fconv$ is the future of $S$.

The normal vectors of $S$ are time-like, so a natural Gauss map is defined
$\gauss:S\rightarrow \H^d$ sending $p$ to the unique future-oriented unit time-like vector orthogonal to $T_pS$.
The map $\gauss$ is injective but in general is not surjective.
We say that $\fconv$ is an \emph{F-convex set} if $\gauss$ is surjective. 

Let $\phi$ be a continuous positive function on $\H^d$.
The classical Minkowski problem in Minkowski space asks for an F-convex set  $\fconv$ 
in Minkowski space  with $C^2$ space-like boundary $S$,
such that $\frac{1}{\phi(\eta)}$ is the Gaussian curvature at $\gauss^{-1}(\eta)$.
It turns out that $\phi(x)$ corresponds to the density of the direct image of the intrinsic volume measure of $S$ 
through the map $\gauss$ with respect to the volume measure of $\H^d$.

Analogously to the Euclidean case, it is possible to formulate a weak version of this problem.
Indeed if $\fconv$ is any convex set whose support planes are \old{not time-like}
\new{space-like}, it is always possible to define a set-valued
function $\gauss$ from $\partial \fconv=S$ to $\H^d$. Namely $\gauss$ sends a point $p$ to the set of unit time-like vectors
orthogonal to support planes at $p$.
As in the previous case, the condition that support planes are not time-like implies that $\fconv$ coincides either with the future
or the past of $S$. We will always assume that $\fconv$ is the future of $S$.
We say that $\fconv$ is an F-convex set if moreover $\gauss$ is surjective in the sense that $\bigcup_{x\in S} \gauss(x)=\H^d$.

In \cite{fv} a Radon measure is defined on $\H^d$ 
as the first-order variation of the volume of an $\epsilon$-neighborhood of  $S$. 
If $S$ is $C^1$ and space-like, it 
coincides with the direct image of  the intrinsic volume measure on $S$ through the map $\gauss$. 
We denote by $\A(\fconv)$ this measure, called the \emph{area measure}.

So given a Radon measure $\mu$ on $\H^d$  the generalized Minkowski problem in Minkowski space 
asks for a convex subset $\fconv$ such that $\A(\fconv)=\mu$.

As $\H^d$ is not compact,  a boundary condition must be added  for the well-possessedness of this  problem.
Geometrically this boundary condition corresponds to the choice of  the domain of dependence of the hypersurface $S$, or
equivalently the light-like support planes of $\fconv$.

Results on this direction are for instance pointed out in \cite{li95}, where an  assumption on the regularity of the boundary
condition has been considered. 

In a different direction one could look at an equivariant version of the problem. 
Isometries of Minkowski space are affine transformations of $\R^{d+1}$ whose linear part 
preserves the Minkowski  product.  We will restrict to isometries in the connected component of the identity
of the isometry group, that is $\sigma\in\Isom(\mink)$ whose linear part $\gamma\in \operatorname{SO}^+(d,1)$ preserves the orientation
and leaves the future cone from $0$ invariant. There is a natural isometric action of $\operatorname{SO}^+(d,1)$ on $\H^d$ which identifies
$\operatorname{SO}^+(d,1)$ with the group of orientation preserving isometries of $\H^d$.

Let us fix a group $\Gamma_\tau$  of isometries of Minkowski space.
If  $\fconv$ is an F-convex set invariant under the action of $\Gamma_\tau$, the area measure $\A(\fconv)$ is a measure on $\H^d$ invariant by
the group $\Gamma<\operatorname{SO}^+(d,1)$ obtained by taking the linear parts of elements of $\Gamma_\tau$.

So if a $\Gamma$-invariant measure $\mu$ in $\H^d$ is fixed, the equivariant version of Minkowski problem asks for
a $\Gamma_\tau$-invariant F-convex set $\fconv$ such that $\A(\fconv)=\mu$.
The geometric interest of this equivariant version  relies on the fact that Minkowski problem can be
intrinsically formulated in any flat globally hyperbolic space-time (see Section \ref{sec:mgh} of the Introduction).
By using the theory developed by Mess in \cite{Mes07} and generalized in \cite{Bar05, Bon05}, 
this intrinsic Minkowski problem  on a flat Lorentzian manifold 
turns out to be equivalent to an equivariant Minkowski problem in Minkowski space in many geometrically interesting cases.

From \cite{Bon05} it is known that if $\Gamma$ is a uniform lattice in $\operatorname{SO}^+(d,1)$ (that means that $\Gamma$ is discrete and
$\H^d/\Gamma$ is compact), there is a maximal $\Gamma_\tau$-invariant F-convex set, say $\Omega_\tau$, which
contains all the $\Gamma_\tau$-invariant F-convex sets. Moreover, if $\fconv$ is any $\Gamma_\tau$-invariant F-convex set, then the domain
of dependence of its boundary coincides with $\Omega_\tau$.
So  requiring the invariance under $\Gamma_\tau$ of $\fconv$ automatically fixes the boundary conditions. 

On the other hand, the boundary conditions arising in this setting
 do not satisfy the regularity conditions required in \cite{li95}, so the existence result
of that paper is not helpful in this context.

In the present paper we solve the equivariant generalized Minkowski problem in any dimension, with the only assumption that the group
$\Gamma$ is a uniform lattice of $\operatorname{SO}^+(d,1)$. 

The study of the Minkowski problem in Minkowski space, and its relations with Monge--Amp\`ere equation is not new. 
But, as far as the authors know,  all the results concern the regular Minkowski problem 
(excepted a polyhedral Minkowski problem in the Fuchsian case $\Gamma_\tau=\Gamma$ in \cite{Fil12}).
Two previously known results are particular cases of Theorem~\ref{thm:main1}.

The first one is a theorem of V.~Oliker and U.~Simon \cite{OS83}, solving the Minkowski problem in the 
$C^{\infty}$ setting and in the 
Fuchsian case, in all dimensions.
Actually they consider the underlying PDE problem intrinsically on the compact hyperbolic manifold $\H^d/\Gamma$.

The second one is a theorem of T.~Barbot, F.~B\'eguin and A.~Zeghib \cite{BBZ}, 
solving the equivariant Minkowski problem in the $C^{\infty}$ setting in the $d=2$ case.
It is interesting to note that their proof is totally different from the one presented here, and use a geometric feature of the dimension $2$. 
They show that a sequence of Cauchy surfaces with constant curvature $<-\infty$ and bounded diameter cannot have small systole as a simple consequence of Margulis Lemma.  
By contrast, in higher dimensions hypersurfaces with constant Gauss--Kronecker curvature
does not have any bound on the sectional curvature.

We also discuss the equivariant regular Minkowski problem. Opposite to the $d=2$ case, we will show that in general 
there is no  smooth $\Gamma_\tau$-invariant hypersurface of constant Gauss--Kronecker curvature equal to $1$.
On the other hand, we  prove that there is a constant $c$ depending on the group $\Gamma_\tau$ such that if $\phi$ is
a  smooth $\Gamma$-invariant function with $\phi\geq c$, then the $\Gamma_\tau$-invariant F-convex set $\fconv$ such 
that $\A(\fconv)=\phi \d\H^d$ is  smooth  and strictly convex.

We  prove that the constant $c$ can be taken equal to $0$ if the dual stratification in the sense of \cite{Bon05}
does not contain strata of dimension $k<d/2$, that roughly means that the dimension of the initial singularity
of $\Omega_\tau$ is $\leq d/2$. 
If $d=2$, then this condition is always true so this gives a new proof of the result in \cite{BBZ}.
At least for $d=3$ this condition seems optimal in the sense that the construction of the example  in
 Section~\ref{sub: pogo} works every time the singularity of $\Gamma_\tau$ is  simplicial (in the sense of \cite{Bon05})   
 with at least a stratum of dimension bigger than $1$.
In fact that construction works in any dimension 
under the weaker assumption that the initial singularity of $\Omega_\tau$ contains a $(d-1)$-dimensional polyhedron.

Finally, let us mention related results in  \cite{GJS06}, and also a ``dual'' problem to the Minkowski problem, known as  
the Alexandrov problem. Here one prescribes a ``curvature measure'' rather than an area measure.
Curvature measures are defined using a radial function, and the Alexandrov problem
was solved in the Fuchsian case  in \cite{Ber14}.
It is not clear if this problem have an analog in the general non-Fuchsian case.

The trivial  $d=1$ case was considered in \cite{fv}, so in all the paper, $d\geq 2$.

\subsection{Main results}

Let us formulate in a more precise way the results of the paper.
Let $\Gamma$ be a uniform lattice in  $\operatorname{SO}^+(d,1)$.
An affine deformation of $\Gamma$ is a subgroup of $\operatorname{Isom}(\mathbb R^{d,1})$ obtained by adding translations
to elements of $\Gamma$:
\[
\Gamma_\tau=\{\gamma+\tau_\gamma|\gamma\in\Gamma\}~.
\]
Notice that $\Gamma_\tau$ is determined by the function
$\tau :\Gamma \rightarrow \R^{d+1}$ which assigns to each $\gamma\in\Gamma$ the translation part of the
corresponding element of $\Gamma_\tau$. The condition that $\Gamma_\tau$ is a subgroup forces $\tau$ to verify
a \emph{cocycle} relation: $$\tau_{\alpha\beta}=\tau_\alpha+\alpha\tau_\beta~.$$

A \emph{$\tau$-convex set} is a proper convex set (globally) invariant under the action of $\Gamma_{\tau}$. Here we also require the $\tau$-convex sets to be future:  
\old{they are intersections of half-spaces bounded by space-like planes, and we suppose that the half-spaces are the 
future sides of the hyperplanes} \new{they are disjoint from the past of their support planes}. When $\tau=0$, the term Fuchsian is sometimes used. 

\begin{theorem}\label{thm:main1}
 Let $\mu$ be a $\Gamma$-invariant Radon measure on $\H^d$. Then
 for any cocycle $\tau$ there exists a unique $\tau$-convex set with area measure $\mu$.
\end{theorem}

Moreover the correspondence between the $\tau$-convex sets and the measures is continuous for a natural Hausdorff topology on the space of $\tau$-convex sets, see  Section~\ref{sec:33}.

The proof of the existence part Theorem~\ref{thm:main1} is variational. It is similar to the proofs of the Minkowski problem for convex bodies, where the functional involves the volume of the convex bodies \cite{Ale37,car04}.
Here the group action allows to define a \emph{covolume} for any $\tau$-convex set. The set of $\tau$-convex sets is convex (for linear combinations given by homotheties and Minkowski addition of 
convex sets), and the main ingredient in the variational approach is the following.

\begin{theorem}\label{thm: cov con}
 The covolume is convex on the set of $\tau$-convex sets.
\end{theorem}

\subsubsection{Monge--Amp\`ere equation and regular Minkowski problem}

An important point  widely used  in the paper is that the Minkowski problem can be formulated in terms of a Monge--Amp\`ere equation 
on the unit ball $\ball$ in $\R^d$. Indeed the boundary of any F-convex set $\fconv$ is the graph of a $1$-Lipschitz function $\tilde u$  on $\R^d$.
The Legendre--Fenchel transform of $\tilde u$ is a function $h$ defined on $\ball$ (since $\tilde u$ is $1$-Lipschitz). 
This correspondence is in fact bijective, as any convex function $h$ on $\ball$ defines dually a $1$-Lipschitz convex function
over the ``horizontal'' plane  $\R^d$.

Recall that  for a convex function $h$ there is a well-defined Monge--Amp\`ere measure $\M(h)$  on $\ball$
such that $\M(h)(\omega)$ is the Lebesgue measure of the set of sub-differentials at points in $\omega$.
If $h$ is $C^2$, then $\M(h)=\det(\mathrm{Hess} h)\mathcal L$ where $\mathcal L$ is the Lebesgue measure on $\ball$.

There is a very direct relation between the area measure $\A(\fconv)$ and the Monge--Amp\`ere measure $\M(h)$ of the
corresponding function $h$ on $\ball$. Regarding $\R^d$ as an affine chart in $\RP^{d}$, $\ball$ is identified to the set of 
time-like directions. This gives an identification $\rad:\ball\to\H^d$ and we simply have  
$\rad^{-1}_*(\A(\fconv))=(1-\|x\|^2)^{1/2}\M(h)$.

The correspondence between  F-convex sets and convex functions on $\ball$
 allows to define an action of $\mathrm{Isom}(\mink)$ on the space of convex functions on $\ball$.
Namely, if $\sigma\in \mathrm{Isom}(\mink)$ and $h$ is a convex function on $\ball$ corresponding to the F-convex set $\fconv$,
then $\sigma\cdot h$ is the convex function on $\ball$ corresponding to the F-convex set $\sigma(\fconv)$. 

This action is made explicit in Lemma~\ref{lem: projective action}. 
It turns out that if $\sigma$ is a pure translation then $\sigma\cdot h$ differs by 
$h$ by an affine function,
whereas if $\sigma$ is linear, then $\sigma\cdot h$ is explicitly related (although not equal) to the pull-back of $h$   by the projective
transformation induced by $\sigma$ on $\ball$.

Using those correspondences, Theorem~\ref{thm:main1} can be also stated in terms of Monge--Amp\`ere equation:

\begin{theorem}\label{thm:main2}
Let $\ball$ be the open  unit ball of $\R^d$, and $\|\cdot \|$ be the usual norm. Let $\Gamma$ be a uniform
lattice in $\operatorname{SO}^+(d,1)$.

Let $\mu$ be a Radon measure on $\ball$ such that $\sqrt{1-\|x\|^2}\mu$ is invariant for the 
projective action of $\Gamma$ on $\ball$.

Then, for any  convex function $h_0$ on $\overline{\ball}$ such that,  for any $\gamma\in \Gamma$, on $\partial \ball$, the restriction of
$\gamma\cdot h_0 - h_0$ coincides with the restriction  of an affine map,  there exists a unique convex function $h$ on $\ball$ such that
\begin{itemize}[nolistsep]
 \item $\M(h)=\mu$, where $\M$ is the Monge--Amp\`ere measure,
 \item $h=h_0$ on $\partial \ball$.
\end{itemize}
\end{theorem}

The uniqueness part of Theorem~\ref{thm:main1} follows from classical uniqueness result for the Monge--Amp\`ere equation. Conversely,
the existence part of Theorem~\ref{thm:main1} can be regarded
as a new result about Monge--Amp\`ere equation.

The main result about Monge--Amp\`ere equation, close to the ones considered here, is 
the existence of a convex function $h$ on a strictly convex 
open bounded set $\Omega$, with prescribed continuous values on the boundary, satisfying $\M (h) = \mu$, but for $\mu$ with finite total mass, \cite{RT77}, \cite[1.4.6]{Gu01}.
Also result in \cite{li95} is formulated in terms of existence of a solution of a Monge--Amp\`ere equation.

Under this correspondence, in the Fuchsian case the covolume can be written 
as the functional introduced by I.~Bakelman for the Monge--Amp\`ere equation, see Remark~\ref{rem: bakelman}.

Using again classical result of Monge--Amp\`ere theory, we can get a ``regular'' version of Theorem~\ref{thm:main1}. 
There are two main points that will allow to use those regularity results. The first one is a result 
contained in \cite{li95} 
which will imply that convex functions $h$ on $\ball$ corresponding to 
$\tau$-convex sets  have a continuous extension $g$ on the boundary of the ball. 
%The continuity of this extension  is not obvious at all,  at a first sight.
It turns out that the convex envelope of $g$ corresponds to the maximal $\tau$-convex set $\Omega_\tau$.
So $g$ only depends on $\tau$.

The second ingredient is a positive interior  upper bound of the difference between
the convex envelope of $g$ in $\ball$ and the support function $h$ of a $\tau$-convex set with positive
area measure.
The example in Section~\ref{sub: pogo} shows that some assumption must be added to get this bound.
In particular we find a bound  if the prescribed measure has sufficiently large total mass. 
This assumption can be weakened if the dual stratification of the initial singularity of $\Omega_\tau$ contains only
strata of dimension $\geq d/2$. In this case $\Omega_{\tau}$ is called \emph{simple}. This condition comes out as we notice here that the strata of the  dual stratification pointed out
in \cite{Bon05} correspond to the maximal regions of $\ball$ where the convex envelope of $g$ is affine. 
Assuming that those regions have dimension $\geq d/2$ a multi-dimensional version of the Alexander Heinz Theorem  
ensures that $h$ is strictly convex, providing the bound we need. 
This is stated in details in Theorem~\ref{thm: alex heinz} and  Corollary~\ref{cor: MA regularite}. See also Theorem~\ref{thm:reg equi}.

\subsubsection{Minkowski problem in flat space times with compact hyperbolic Cauchy surfaces}\label{sec:mgh}

A Cauchy surface in a Lorentzian manifold $S$ is an embedded hypersurface such that any inextensible  time-like path meets
$S$ exactly at one point. A space-time is said globally hyperbolic (GH) if it contains a Cauchy surface. 
Topologically a  GH space-time $M$ is simply $S\times\R$, where $S$ any Cauchy surface in $M$.

If $S$ is any space-like hypersurface in a flat space-time $N$, then it is a Cauchy  surface of some neighborhood
 in $N$.  An extension of $S$ is a GH  flat
space-time $M$ where  a neighborhood of $S$ in $N$ can be isometrically embedded
so that the image of $S$ is a Cauchy surface in $M$. 
By a very general result in general relativity \cite{GH69}, there exists a unique maximal extension $M$, 
where here maximal means that all the other extensions isometrically embeds into $M$.

A maximal globally hyperbolic compact flat space-time (we will use the acronym MGHCF) is a maximal flat space-time
containing a compact Cauchy surface.
In \cite{Mes07} Mess studied MGHCF space-times in dimension $2+1$, and a generalization of Mess theory was done
in any dimension in \cite{Bar05, Bon05}.

By \cite{Bar05} it is known that if $S$ is the Cauchy surface of some MGHCF space-time then up to a finite covering
$S$ can be equipped with a metric locally isometric to $\H^k\times\R^{d-k}$ for some $0\leq k\leq d$. 
We will be mainly interested in the case where $S$ is of hyperbolic case (that is $k=d$), as the Minkowski problem
is not meaningful in the other cases (see Remark~\ref{rk:nhyp}).

In this hyperbolic case it turns out that any MGHCF space-time (up to reversing the time orientation)
 is obtained as the quotient under an affine deformation $\Gamma_\tau$
of a uniform lattice $\Gamma$ in $\operatorname{SO}^+(d,1)$ of the interior of the maximal $\Gamma_\tau$-invariant 
F-convex set $\Omega_\tau$.

A  subset $\fconv$ in $M$ is convex if whenever the end-points of a geodesic segment in $M$ lies
in $\fconv$, then the whole segment is in $\fconv$.
Any $\tau$-convex set is contained in $\Omega_\tau$ and projects to a convex subset of $M$.
The boundary of $\fconv$ projects to a Cauchy surface in $M$, provided that $\fconv$ does not meet the boundary of $\Omega_\tau$.
Conversely if $\fconv$ is a convex subset of $M$, its pre-image in $\Omega_\tau$ is a $\tau$-convex set.

By the flatness of the manifold $M$, the tangent bundle $TM$ is a $\R^{d+1}$ flat bundle with holonomy
$\Gamma$. This implies that if $S$ is any $C^1$-space-like hypersurface, its normal field define a  Gauss map
\[
   \gauss_S:S\rightarrow\H^d/\Gamma
\]
which lifts to the usual Gauss map $\gauss:\tilde S\to \H^d$.
 If $S$ is a $C^1$-strictly convex  Cauchy surface whose future is convex, then   $\gauss_S$ is a homeomorphism.

We say that a Cauchy surface in $M$ is convex, if its future is convex. 
For any convex Cauchy surface it is possible to define a Gauss map as the set-valued function that 
sends  a point $p\in S$ to the set of unit time-like direction orthogonal to local support planes of $S$ at $p$.
It turns out that $\gauss_S$ is always surjective.

The area measure can be defined for every convex Cauchy surface in the same way as in the local model.
If $S$ is $C^1$, the corresponding area measure is a measure on  $\H^d/\Gamma$ that 
coincides with the intrinsic volume form via the Gauss map. So
the Minkowski problem in $M$ asks  for a convex
Cauchy surface in $M$ whose area measure is $\mu$, where  $\mu$ is  a fixed measure of $\H^d/\Gamma$,.

If $S$ is smooth  with positive Gauss--Kronecker curvature,
 then $\mu=f^{-1}\d V$, where $\d V$ is the volume measure on $\H^d/\Gamma$ and
$f(x)$ is the Gauss--Kronecker curvature of $S$ at $\gauss^{-1}(x)$.

Lifting the problem to the universal covering, it can be rephrased as a $\tau$-equivariant Minkowski problem,
with the additional requirement that the solution of the equivariant problem   must be contained in the interior of
$\Omega_\tau$. Even assuming that $\mu$ is a strictly positive measure (in the sense that it is bigger than $c \d V$, where 
$\d V$ is the intrinsic measure on $\H^d/\Gamma_\tau$), Theorem~\ref{thm:main1} does not ensure that the equivariant 
solution does not meet the boundary of $\Omega_\tau$.  To this aim  we need to require that $\mu$ has sufficiently big mass.

\begin{theorem}\label{thm GH1}
For any MGHCF space-time $M$, there is a constant $c_0$ such that if $\mu$ is a measure
on $\H^d/\Gamma$ such that $\mu>c_0\d V$, then there is a unique Cauchy surface $S$ in $M$ with
area measure equal to $\mu$.
If $\Omega_\tau$ is simple we may assume $c_0=0$.
\end{theorem}

%The last statement follows by the fact that $0$ is contained in the $\Gamma$-orbit of any point of $\partial \operatorname{I}^+(0)$. So
%if a $\Gamma$-invariant F-convex set meets the boundary of $\operatorname{I}^+(0)$, then it coincides with the closure of $\operatorname{I}^+(0)$.

\begin{remark}\label{rk:nhyp}\emph{
%One could wonder why we consider only the case where $S$ is of hyperbolic type.
By \cite{Bar05} if $M$ is a MGHCF space-time whose Cauchy surfaces are
  not of hyperbolic-type, then for any  convex  Cauchy surface $S$ in $M$  the Gauss map
on the universal covering $\gauss:\tilde S\to\H^d$ is never a proper map, as its image has dimension $k<d$ and
the fibers are $n-k$ affine subspaces. Notice that in this case the Gauss--Kronecker curvature of $S$ is zero, 
so the Minkowski problem is not meaningful in this case.}
\end{remark}

 The results about the regular equivariant Minkowski problem allows to state the following theorem.

\begin{theorem}\label{thm GH2}
Let $M$ be a flat MGHCF spacetime, 
 let $f$ be a positive function on $\H^d/\Gamma$ of class $C^{k+1}$, $k\geq 2$.

 There exists $c(M)\geq 0$, such that 
 if $f<c(M)$, then there exists a unique smooth strictly convex Cauchy surface $S$ of class $C^{k+2}$ in $M$ 
 such that $f(x)$  is the Gauss--Kronecker curvature of $S$ at $\gauss_S^{-1}(x)$.

 If $\Omega_\tau$ is simple, then $c(M)=+\infty$.
\end{theorem}

\subsection{Outline of the paper}
In Section \ref{sec:2} we state some basic facts we need on F-convex sets in Minkowski space.
In particular, we describe the connection between Minkowski problem in Minkowski space and the Monge--Amp\`ere
equation on the disc. As already mentioned this is achieved by considering the Legendre transform $h$ of the function
$\tilde u:\R^d\to\R$ whose graph is identified to the boundary of the F-convex set $\fconv$. More intrinsically we study the support function
of $\fconv$, as defined in \cite{fv}. This is  a $1$-homogenous function on $\fut(0)$. 
The area measure of $\fconv$ can be directly computed by the restriction of the support function on $\H^d$, 
whereas its restriction on the ball coincides with $h$.

Although the material of this section is already known, we  give a detailed exposition of the theory as
we widely use this correspondence. 
Most of the results are only stated and some references are given for the proof. 
One of the results we  prove in details is the multidimensional version of the Alexandrov--Heinz Theorem, Theorem~\ref{thm: alex heinz}.
Although the proof follows closely the proof of the standard Alexandrov--Heinz Theorem given in \cite{TW08}, we did not 
find a reference so details are given for the convenience of the reader.

In Section~\ref{sec3} we fix an affine deformation $\Gamma_\tau$ of a uniform lattice $\Gamma$ and 
study $\tau$-convex sets, i.e. F-convex sets invariant by $\Gamma_\tau$.

After stating some  simple properties  on the action of the isometry group of Minkowski space on F-convex sets and 
dually on the space of support functions, we point out a structure of affine cone on the space of $\tau$-convex sets.
By using a result of \cite{li95} we show that the support function of 
a $\tau$-convex set extends to the boundary of
$\fut(0)$ and the value of the support function on the boundary $g_\tau$ only depends on the cocycle. 
Observing that the convex envelope of $g_\tau$ is the support function  of the maximal $\tau$-convex set 
$\Omega_\tau$, we will recover the strata of $\H^d$ studied in \cite{Bon05} as 
maximal regions of $\H^d$ where $g_\tau$ is an affine function.

We put a distance on the space of $\tau$-convex sets that basically corresponds to the $L^\infty$-distance on the
corresponding support functions. A compactness property of this distance is proved, Lemma~\ref{lem: conv lip}, 
that can  be regarded as the analog of Blaschke Selection Theorem for convex bodies. 
We then study area measures of  $\tau$-convex sets. They are $\Gamma$-invariant so can be regarded as measures on $\H^d/\Gamma$,
in particular for any $\tau$-convex set $\fconv$, its total area is the total mass of the induced measure on $\H^d/\Gamma$.
We prove a monotonicity result for the total area Lemma~\ref{lem:area monotonic}, 
which implies that a $\tau$-convex set with sufficiently big area
does not meet the boundary of the maximal $\tau$-convex set.

After this preliminary part we introduce the covolume and study its main properties.
Covolume of $\fconv$ is defined as the volume of the complement of $\fconv/\Gamma_\tau$ in $\Omega_\tau/\Gamma_\tau$. 
 In particular we prove that it is a convex functional on the set of $\tau$-convex set. 
 The proof of this result relies  in a concavity result for the volume of particular convex bodies Lemma~\ref{lem:vol concave}
and its Lorentzian version Corollary~\ref{cor:vol concave lor}. 
A technical point in the proof 
is that we need to use 
%on the existence of 
 a convex fundamental domain for the
action of $\Gamma_\tau$ on a $\tau$-convex set $\fconv$. Actually we are able to construct this fundamental domain
only if $\fconv$ lies in the interior of $\Omega_\tau$, but this is in fact sufficient for our aims.
Using this convex fundamental domain  
we prove that the covolume is strictly convex on the  convex subspace of $\tau$-convex sets
lying in the interior of $\Omega_\tau$. Then by  approximation  the convexity is proved in general.

Let us explain how convexity of the covolume is related to the Minkowski problem.
In the Fuchsian case ($\tau=0$),       
similarly to the Euclidean case for convex bodies \cite{car04},  
the derivative of the covolume along a family of support functions  $t\oh_1+(1-t)\oh_0$
 at $\oh_0$ is just $-\int (\oh_1-\oh_0)\d\A(\oh_0)$. Here $\oh$ denotes the \emph{hyperbolic support function}
 obtained by restricting on $\H^d$ the $1$-homogenous extension of the support function $h$ on $\ball$ over
 $\fut(0)$.

So if $\mu$ is a $\Gamma$-invariant measure and $\oh_0$ is a Fuchsian solution of the Minkowski problem $\A(\oh_0)=\mu$, 
$\oh_0$ turns to be a critical point of the functional $L_\mu(\oh)=\cov(\oh)-\int \oh\d\mu$, and by convexity of $L_\mu$ it must 
be the minimum of $\oh_0$. So to find out the solution one looks at minima  of the functional $L_\mu$. 

In the non Fuchsian case some difficulties arise. 
First  in the Fuchsian case 
one has a simple representation formula for the covolume
 \[
     \cov(\oh)=-\frac{1}{d+1}\int \oh \d\A(\oh)
 \]
 where $\oh$ is the support function \cite{Fil12}. 
 \old{In the affine case, 
the representation formula cannot be stated as in the Fuchsian case, as $\oh$ is not $\Gamma$-invariant.
One could wonder if $\cov(\oh)=-\frac{1}{d+1}\int (\oh-\oh_\tau)\d\A(\oh)$, but 
this is false as it an be seen taking $\oh-\oh_\tau$ constant for some simple examples of $\Omega_{\tau}$ .}
\new{In the affine case there is no similar explicit and simple formula.}

Actually this is not a major problem, as the main point is  a formula for the first order variation of the covolume, and
not for the covolume itself. In fact by using the convexity of the covolume
its derivative can be computed along a family of support functions $t\oh_1+(1-t)\oh_0$ with $\oh_1\leq \oh_0$ and this
turns out to be sufficient to show that the minimum point of $L_\mu$ is a solution of the Minkowski problem.

Another difference with respect to the Fuchsian case is  that the strategy to find the solution
is slightly different. In fact first it was proved that  the minimum $\oh$ of the covolume on some level set 
$\{\oh|\int \oh d\mu=1\}$ satisfies the equation $\A(\oh)=c\mu$ where $c$ is a some positive.
Then by scaling $\oh$ by $c^{-1/d}$ one finds the solution. 
A similar  strategy could be also followed in the case we consider in the present paper, but, due to the fact that the set of
$\tau$-convex sets is not stable under homotheties, the argument would become more intricate and the strategy we propose
seems in this case more direct and simple.

Indeed, using the convexity of the covolume one sees that the difference 
 between the covolume of  two $\tau$-convex sets $\fconv_0\subset \fconv_1$ can be estimated in terms of the integrals of the difference of the corresponding support
 functions with respect to respective area measures, see formula \eqref{eq: enc diff vol}. 
 Then refining the expected arguments in the 
 Fuchsian case one sees that $\Gamma_\tau$-equivariant solutions of the equations $\A(\oh_0)=\mu$ corresponds to minima of
 the functional $L_\mu(\oh)=\cov(\oh)-\int (\oh-\oh_\tau)\d\mu$. The existence of the minimum is then simply achieved
 using the compactness property of the support functions.
 
 Once the general Minkowski problem is proved, using the correspondence with Monge--Amp\`ere equation we see that if the measure
 $\mu$  is regular with respect to the volume form of $\H^d$ with smooth and strictly positive density, the corresponding solution
 is smooth provided that the $\tau$-convex set lies in the interior of $\Omega_\tau$. 
 A counterexample is then given to show that in general even if $\mu$ coincides with the Lebesgue measure up to a constant factor
 the corresponding domain can  meet the boundary of $\Omega_\tau$.
 
 Finally in the appendix we prove  that the space of $C^2_+$ $\tau$-convex sets is dense
 in the space of $\tau$-convex sets (where $C^2$ can be replaced by smooth). 
 Here  a $C^2_+$ $\tau$-convex set is a $C^2$ $\tau$-convex set such that the Gauss map is a $C^1$-diffeomorphism,
 or analogously with positive Gauss--Kronecker curvature is strictly positive at any point.
 This basically implies that any convex hypersurface
 in a maximal globally hyperbolic flat space-time can be approximated by smooth strictly convex hypersurfaces. Although the fact is not surprising,
 there was no proof at our knowledge, so we give some details.

\subsection{Acknowledgements}

The authors thanks the anonymous referee for his/her valuable comments which helped to improve the presentation of the paper.

The second author thanks Guillaume Carlier for useful conversations.

The first author was partially supported by FIRB 2010 project  RBFR10GHHH\_002 ``Geometria e topologia delle variet\`a
in bassa dimensione''.

The second author was partially supported by the ANR GR-Analysis-Geometry.
Most of this work was achieved during stays of the second author in the Department of Mathematics of the University of Pavia. He thanks the institution for its hospitality.

\subsection{Notations}

\begin{tabular}{| l | p{5cm} || l |p{5cm} |}
\hline
$\mink$ & Minkowski space & $\L$ & Lebesgue measure of $\R^d$ \\

$\R^d$ & Euclidean space identified with the \emph{horizontal} hyperplane in $\mink$  &  $\d\H^d$  & volume measure of $\H^d$ \\

$\ball$ & unit ball in $\R^d$ centred at $0$ & $\M(h)$ & Monge-Amp\`ere measure of $h$ \\ 

$\hat \ball$ & intersection of $\fut(0)$ with the affine plane $(0,\ldots,0,1)+\R^d$ & $\A(\fconv)=\A(\oh)$ & area measure on $\H^d$ \\

$\fut(0)$ & future cone of $0$ in $\mink$  & $\A(h)$ & area measure on $\ball$, $A(h)=(\rad^{-1})_*(A(\oh))$ \\

 $\H^d$ & hyperboloid model of hyperbolic space & $\phi$ & curvature function, density of $\A(\oh)$ for $\oh\in C^2$\\

$\jplus(\H)$ & convex region bounded by the hyperboloid in $\mink$ & $\operatorname{Area}(\fconv)$ & $\A(K)(\H^d/\Gamma)$\\

$ \langle\cdot,\cdot\rangle_- $& Minkowski product & $V_\epsilon(\fconv)$ & measure on $\H^d$ introduced in Definition~\ref{defVeps} \\ 

$\plane{H}{X}$ & hyperplane defined by the equation $\la z, X\ra=H(X)$ & $S_i(\fconv)$ & area measure of order $i$\\

 $\proj:\mink\to\R^d$ & orthogonal projection & $\overline{V}_\epsilon(\fconv)$ & measure on $\H^d/\Gamma$ induced by $V_\epsilon(\fconv)$\\

$\hat x$ & $(x,1)\in\hat\ball$ for $x\in\ball$ &  $\overline{S_i}(\fconv)$ & area measure of order $i$ on $\H^d/\Gamma$ induced by $S_i(\fconv)$ \\

$\lambda(x)$ & $\sqrt{1-||x||^2}$ for $x\in\ball$ & $\operatorname{covol}$ & covolume\\

$\rad:\ball\to\H^d$ & radial identification & $\tilde u:\R^d\to\R$ & function whose graph is $\partial\fconv$ \\ 

 $\grad^-$ & Lorentzian gradient in $\mink$ & $\tilde h:\R^d\to \R\cup\{+\infty\}$ & Legendre-Fenchel transform of $\tilde u$\\

$\grad^\H$ & hyperbolic gradient in $\H^d$ & $u$ & restriction of $\tilde u$ on $\proj(\partial_s\ball)=\nabla u(\ball)$ \\

$\grad$ & Euclidean gradient in $\R^d$&$g$ & continuous function on $\partial\ball$ \\ 

$\nabla$ & Levi-Civita connection on $\H^d$ &$h_g$ & convex envelope of $g$\\

$\nabla^2$ & covariant Hessian in $\H^d$&$\Omega_g$ & F-convex set with support function $h_g$ \\ 

$\hess$  &Euclidean Hessian in $\R^d$ & $F_g(x)$ & maximal convex set of $\ball$ containing $x$ on which $h_g$ is affine\\

$\fconv$ & F-convex set & $\Gamma$ & uniform lattice in $\operatorname{SO}_0(2,1)$ \\

$\partial_s\fconv$ & space-like boundary of $\fconv$ &  $\tau:\Gamma\to\R^{d+1}$ & cocycle\\

$\partial_{\mathrm{reg}}\fconv$ & regular part of the space-like boundary & $\Gamma_\tau$ & affine deformation of $\Gamma$ \\  

$\gauss:\partial_s\fconv\to\H^d$ & Gauss map & $\gamma$ & element of $\Gamma$ \\

$\chi:\H^d\to\partial_s\fconv$ & inverse of the Gauss map & $\gamma_\tau$ & affine transformation $p\to\gamma(p)+\tau_\gamma$ \\ 

  $C^2_+$ & F-convex set with $\partial_s \fconv$ $C^2$ and $\gauss$ a $C^1$ diffeomorphism  & $\bar\gamma$ & projective transformation on $\ball$ induced by $\gamma$\\

$H:\fut(0)\to\R$ & extended support function & $\Omega_\tau$ & maximal $\tau$-convex set\\ 

$\oh:\H^d\to\R$ & hyperbolic support function & $H_\tau, \oh_\tau, h_\tau$ & support functions of $\Omega_\tau$\\

$h:\ball\to\R$ & ball support function  & $g_\tau$ & boundary value of the support function of any $\tau$-convex set\\ 

$\partial h$ & subdifferential of $h$ & $T$ & cosmological time of $\fconv$ \\  

  $\fund(x)$ & Lorentzian Dirichlet polyhedron & $\Sigma_t$ & level set $T^{-1}(t)$\\

\hline
\end{tabular}

%The present paper mainly deals with F-convex sets $\fconv$ (p.~\pageref{F-convex}), which are such that their Gauss map $G$ (p.~\pageref{Gauss map}) is surjective. Their boundary has a space-like part 
%$\partial_s \fconv$  (p.~\pageref{partial_s \fconv}). 
%The convex set $\fconv$ is determined by its support function $H$ (p.~\pageref{sup H})  defined on the   future cone of the origin $\mathrm{I}^+(0)$, and which is itself determined by its restriction $\oh$ (p.~\label{oh}) to the hyperbolic space or its restriction $h$ (p.~\pageref{h}) to the open unit ball $B$.

%The set $\fconv$ will be $g$-convex if $h$ has continuous extension $g$ to the sphere (p.~\pageref{g}). The $\Omega_g$ will be the biggest future $g$ convex set.

%+ measures

%+ groups

\section{Area measure of F-convex sets}\label{sec:2}

In this section, we recall some facts about particular convex sets in Minkowski space and the Monge--Amp\`ere equation.

\subsection{Background on convex sets}\label{sub:back}

%We refer to \cite{fv} for details.
%Let us introduce the following notations on $\mink$.

\paragraph{Notations}
 For a set $S$ in $\mink$, 
 $\fut(S)$ is the set of the final points  of curves directed by a future time-like vector and starting from $S$. In particular
 $\fut(0)$ is the future cone of the origin.
 
We will consider $\R^d$ as a subset of $\mink$ through the isometric inclusion $x\mapsto (x,0)$.
We denote by $\hor$ the affine hyperplane  $(0,\ldots,0,1)+\R^d$: 
for $x\in\R^d$ we denote by $\hat x=(x,1)$ the corresponding point on $\hor$.
The intersection of $\hor$ with $\fut(0)$ is $(0,\ldots,0,1)+\ball$, where $\ball$
is the unit ball centered in $0$ in $\R^d$. We denote the translated ball by $\hat\ball=(0,\ldots,0,1)+\ball$.

%Let us denote by $\rad$  the inverse mapping $\rad:P\cap \fut(0)\to\H^d$.
%We fix the point $\eta=(0,\ldots,0,1)$ and identify the affine horizontal plane $P$
%passing through $\eta$ with $\R^d$ through the map $x\mapsto(x,1)$
Let us introduce the following notations
\begin{itemize}[nolistsep]
% \item $\ball$ is the open unit ball centred at the origin of $\R^d$, \index{$B$ the open unit ball centred at the origin of $\R^d$}
 %\item $\hat{x}=(x,1)$ for any $x\in\R^d$, \index{$\hat{x}=(x,1)$}
 \item for $X\in\mink$ we put $\|X\|_-=\sqrt{|\langle X,X\rangle|_-}$
 \item $\langle \cdot,\cdot\rangle $ is the usual scalar product on $\R^d$, \index{$\langle \cdot,\cdot\rangle $ the usual scalar product on $\R^d$}
 \item $\|\cdot \|$ is the associated norm,\index{$\|\cdot \|$ the usual norm on $\R^d$}
 \item $\lambda(x)=\sqrt{1-\|x\|^2}=\|\hat{x}\|_-$, $x\in \ball$,\index{$\lambda(x)=\sqrt{1-\|x\|^2}=\|\hat{x}\|_-$}
 \item $\pi:\mink\to\R^d$ is the orthogonal projection sending $(x,x_{d+1})$ to $x$
 \item $\mathcal{L} $ is the Lebesgue measure on $\R^d$. \index{Measure! $\mathcal{L} $  the Lebesgue measure} 
\end{itemize}

%Notice that $\ball$ corresponds exactly to the intersection $P\cap \fut(0)$. 
The radial map
$\rad$  from $\ball$ to $\H^d$ is  $\rad(x)=\frac{1}{\lambda}\hat{x}$, while \index{$\rad$ Radial map  $\rad(x)=\frac{1}{\lambda}\hat{x}$}
$\rad^{-1}$ maps 
$(x,x_{d+1})\in\H^d$ to $x/x_{d+1}=x/\sqrt{1+\|x\|^2}\in \ball$.
It can be useful to consider on $\ball$ the pull-back of the hyperbolic metric (we then get the Klein model of the hyperbolic space). If 
$\mathrm{d}\H^d$ is the volume element for the hyperbolic metric on $\ball$,
an explicit computation gives
\begin{equation}\label{eq:density vol ball}
 \d\H^d=\lambda^{-d-1}\L~.
\end{equation}

\paragraph{F-convex sets.}

Let $\fconv$ be a closed convex set which is defined as an intersection of
the future side of space-like hyperplanes of $\mink$. 
Note that it may have light-like support planes but no time-like support plane.
We will denote by $\partial_s\fconv $ \index{$\partial_s \fconv$}\label{partial_s K} the part of the boundary $\partial \fconv$ of $\fconv$ which meets space-like support planes.
The Gauss map of $\fconv$ is a set-valued map from $\partial_s \fconv$ to $\H^d$, which associates to a point of $\partial_s\fconv$ \index{Gauss map}\label{Gauss map}
all the unit future time-like support vectors of $\fconv$ at $x$. 
We say that $\fconv$ is an \emph{F-convex set}\index{Convex set! F-convex set}\label{F-convex} if its Gauss map is surjective. In other terms, 
any unit future time-like vector is a  support vector of $\fconv$.
An F-convex set $\fconv$ is called $C^2_+$\index{$C^2_+$} if its boundary is a $C^2$ hypersurface and its Gauss map is a $C^1$ diffeomorphism.
We  say that  $\fconv$ is $C^\infty_+$ if its boundary is smooth and its Gauss map is a diffeomorphism.

\paragraph{Support functions.}
\new{
Let $\fconv$ be an F-convex set.
The  (extended) support function $H$ of $\fconv$ is the map from $\fut(0)$ onto 
$\R$ given by 
$$H(X)=\operatorname{sup}\{\langle p,X\rangle_- | \forall p\in \fconv \}~. $$\index{Support function!  $H$}
\label{sup H}

By definition the support function is sublinear:  it is  homogeneous of degree $1$ and subadditive, 
or equivalently $1$-homogenous and convex.
Conversely any sublinear function on $\fut(0)$ is the support function of a unique F-convex set  \cite{fv}.

By homogeneity, $H$ is determined by its restriction on any hypersurface of $\fut(0)$  that meets all the radial lines,
 like for instance $\H^d$ or the translated ball $\hat\ball$.
 We will denote by $\bar h$ the restriction of $H$  on $\H^d$. It will be called the \emph{hyperbolic support function} of $\fconv$. 
 \index{Support function! $\bar h$ the hyperbolic support function}
Analogously we set $h:\ball\to\R$ by $h(x)=H(\hat x)$, and refer to $h$ as the \emph{ball support function} of $\fconv$.
By the homogeneity of $H$  we  get
$$h(x)=\lambda(x)\overline{h}(\rad(x))~.$$
 
Of course the support function $h$  is convex on $\ball$. 
The following lemma says that all convex functions on $\ball$ are obtained in
this way.  
 \begin{lemma}
A   convex function $h$ on $\ball$ extends  (in the sense above) to a unique
 sublinear  function on $\fut(0)$, namely
 $$H\left((x,z)\right)=zh(x/z)~,$$ $x\in \ball,z\in\R$.
 \end{lemma}
\begin{proof}
The function $H$ is $1$-homogeneous by definition. So we only need to prove that it is subadditive.
Take $X=(x,x_{d+1}), Y=(y,y_{d+1})\in \fut(0)$.
By definition, $$H(X+Y)=(x_{d+1}+y_{d+1})h\left(\frac{x+y}{x_{d+1}+y_{d+1}}\right).$$
Notice that the point $\frac{x+y}{x_{d+1}+y_{d+1}}$ can be expressed as the following convex combination
of $x/x_{d+1}$ and $y/y_{d+1}$:
$$     
\frac{x+y}{x_{d+1}+y_{d+1}}=\frac{x_{d+1}}{x_{d+1}+y_{d+1}}\frac{x}{x_{d+1}}\ +\ \frac{y_{d+1}}{x_{d+1}+y_{d+1}}\frac{y}{y_{d+1}}~.
$$
Since $h$ is convex we get that
\[
   H(X+Y)\leq(x_{d+1}+y_{d+1})
   \left(\frac{x_{d+1}}{x_{d+1}+y_{d+1}}h\left (\frac{x}{x_{d+1}}\right)\ +\ \frac{y_{d+1}}{x_{d+1}+y_{d+1}}
   h\left(\frac{y}{y_{d+1}}\right)\right)\]
$$   =H(X)+H(Y)~.$$

\end{proof}
\begin{corollary}
Any convex function on $\ball$ is the support function of a unique F-convex set of $\mink$.
\end{corollary}

\begin{example}\rm{ The support function of the convex 
 side of the upper-sheet of the hyperboloid with radius $t$ is $H(X)=-t\|X\|_-$, 
and its restriction to $\H^d$ is the constant function equal to $-t$. 
Its restriction to $\ball$ is $h(x)=H(\hat{x})=-t\lambda(x)$. Note that
 $\hess h=-t\hess \lambda$, so the convexity of $h$  implies that $\hess \lambda$ is semi-negative definite.
 
We will denote by  $ \jplus(\H)$ the convex 
 side of the upper-sheet of the hyperboloid with radius $1$. 
 }
\end{example}

\begin{example}\rm{The restriction to $\fut(0)$ of the linear functionals
of $\R^{d+1}$ are the support functions on $\fut(0)$ of the future cone of points.
Their restriction to $\ball$ are exactly the restriction to $\ball$ of the 
affine maps of $\R^d$.
}
\end{example}

The support function $H$ determines $\fconv$: indeed for any $X\in\fut(0)$ the affine
plane $\plane{H}{X}=\{p\in\mink|\langle p, X\rangle_-=H(X)\}$ is the support  plane for $\fconv$
orthogonal to $X$ so $$\fconv=\bigcap_{X\in \fut(0)}\overline{\fut(\plane{H}{X})}~.$$

\new{A less obvious remark is that $\plane{H}{X}$ meets $\partial_s\fconv$ at least in one point, see
\cite{fv}.}
Under some regularity assumption on $H$ it is possible to determine the intersection between
$\partial_s\fconv$ and $\plane{H}{X}$. 
Indeed, if $H$ is $C^1$, by convexity we have $H(Y)\geq H(X)+\langle \grad^-_X H, Y-X\rangle_-$, where
$$\grad^-(H)=(\partial_1H,\ldots,\partial_d H,-\partial_{d+1}H)$$ is the Minkowski gradient.
Moreover, since $H$ is $1$-homogeneous we have $$H(X)=\langle\grad^-_X H, X\rangle_-~,$$ so we deduce
that $\langle \grad^-_X H, Y\rangle_-\leq H(Y)$, that implies that $\grad^-_X H$ lies on $\fut(\plane{H}{Y})$.
Indeed, $\grad^-_X H$ is the only point in $\fconv$ with support plane orthogonal to $X$, that is
$\grad^-_X H=\gauss^{-1}(X)$.
In fact the following lemma holds:

\begin{lemma}[{\cite{fv}}]\label{lm:c1supp}
If the support function $H$ of a $F$-convex set $\fconv$ is $C^1$, then $\partial_s\fconv$ is
the range of the map $\chi:\fut(0)\to\R^{d+1}$ \index{$\chi$} sending $X $ to $\grad_X^-H$.
Moreover, $\partial_s\fconv$ is strictly convex, in the sense that each space-like support plane meets $\fconv$
at one point.

Conversely if $\partial_s\fconv$ is strictly convex, the  support function $H$ is $C^1$.
\end{lemma}

Assume that $H$ is $C^1$, or analogously that $\partial_s\fconv$ is strictly convex. Denote by $\chi:\fut(0)\to\mink$ the map
sending $X$ to $\grad^-_X H$.  
As $H$ is $1$-homogeneous, $\chi$ is constant along the radial 
rays. So we have   $\partial_s\fconv=\chi(\H^d)=\chi(\hat\ball)$.
The restriction of $\chi$ to $\H^d$ (respectively $\hat\ball$)  can be easily expressed in terms of the support functions $\oh$  (respectively $h$):

\begin{lemma}
 Assume that the support function $H$ is $C^1$. 
 Then, for $\eta\in\H^d$
  \begin{equation}
   \label{eq:chi on hyp} \chi(\eta)=\grad^\H_\eta\oh - \oh(\eta)\eta
  \end{equation}
 where $\grad^\H$ is the hyperbolic gradient  regarded  as a vector in $\R^{d+1}$ by the natural inclusion
  $T_\eta\H^d\subset \mink$.

  Analogously, for $x\in\ball$
  \begin{equation}
   \label{eq:chi on disc} \chi(\hat{x})=\grad_xh + (\langle x,\grad_xh\rangle -h(x))e_{d+1}
  \end{equation}
where $\grad$ is the Euclidean gradient of $\R^d$.
\end{lemma}
\begin{proof}
The tangential component of $\grad^-_\eta H$ is clearly the intrinsic gradient of $\oh$.
The normal component is obtained by imposing that $D_\eta H(\eta)=H(\eta)=\oh(\eta)$. 

Analogously the tangential component of  the gradient of $H$ to $\R^d$ is equal to  the gradient of $h$.
To compute the normal component, $$D_{\hat{x}}H(e_{d+1})=D_{\hat{x}}H(\hat{x})-D_{\hat{x}}H(\hat{x}-e_{d+1})=H(\hat{x})-
D_xh(x)$$ where $x\in \R^d$ is identified
with $(x,0)\in\R^{d+1}$. Formula~ \eqref{eq:chi on disc} follows.
\end{proof}

By Lemma~\ref{lm:c1supp} it turns out that 
the injectivity of the Gauss map of $\fconv$ is related to the regularity of $H$.
In particular the support function of a $C^2_+$  F-convex set must be
$C^1$. This condition is  not sufficient.
Indeed it implies that $\chi=\gauss^{-1}$ is a well-defined map,
but still $\gauss$ could be set valued: e.g., take $\fconv=\fut(0)$, its support function is $H=0$
and $\gauss$ maps $0$ to the whole $\H^d$.

The following  characterization  of $C^2_+$ F-convex sets
in terms of  the hyperbolic support functions is proved in \cite{fv}.

\begin{lemma}\label{lm:c2+}
A function  $\oh$ on $\H^d$ is the support function
of a $C^2_+$ F-convex set iff it is $C^2$ and satisfies

\begin{equation}\label{eq:twdff}
 \nabla^2\oh-\oh g_\H >0
 \end{equation}
 with $g_\H$ the hyperbolic metric and $\nabla^2$ the hyperbolic Hessian.
 \end{lemma}

Condition \eqref{eq:twdff} can be also expressed in term of the ball
support  function $h$, as the next lemma suggests.

 \begin{lemma}
 Let $H$ be the support function of an $F$-convex set $\fconv$.
  If $H$ is $C^2$, then
  \begin{equation}\label{eq:hessh}
 \mathrm{Hess}_x h (v,w)=\lambda(x) \left(\nabla^2 \oh-\oh g\right)(D_x\rad(v),D_x\rad(w)) 
%\ &&\label{eq: det hess plan}\det \mathrm{Hess}_x h = \lambda^{-d-2}(x) \phi(\rad(x)) \\
%\ &&\label{eq: tr hess plan}\varphi(\rad(x)) = \frac{\lambda(x)}{d} (\tr\mathrm{Hess}_x h-\mathrm{Hess}_xh(x,x)) \nonumber
    \end{equation}
%with $\phi$ the curvature function  and $\varphi$ the mean radius of
%curvature (the mean value of the radii of curvature).
 where $\mathrm{Hess}$ is the Euclidean Hessian on $\R^d$.
 
 In particular $\fconv$ is $C^2_+$ iff $h$ is $C^2$ with positive Hessian everywhere.
 \end{lemma}

\begin{proof} 
By $0$-homogeneity of $\chi$ and because $\lambda \rad$ is the identity
$$h(x)=\langle \chi(\rad(x)), (\lambda \rad)(x)\rangle_-~.$$ 
 We compute 
$$D_x h (v) = \langle D_x(\chi\circ \rad)(v),(\lambda \rad)(x)\rangle_-+\langle \chi(\rad(x)),D_x(\lambda \rad)(v)\rangle_-
=\langle \chi(\rad(x)),D_x(\lambda \rad)(v)\rangle_-$$ 
 because  $\rad(x)$ is orthogonal to the tangent plane at the point $\chi (\rad(x)) $.
 For the same reason, and also because $\mathrm{Hess}_x(\lambda \rad)=0$ (because $\lambda \rad$ is the identity) we compute that
$$\mathrm{Hess}_x h (v,w)=\lambda(x) \langle D_{\rad(x)}\chi(D_x\rad(v)), D_x\rad(w)\rangle_-$$ that together with 
\begin{equation*}\label{eq:rel H h}
 \mathrm{Hess}_{\eta} H(v,w)=\nabla^2 \oh(v,w)-\oh g_\H(v,w)=\langle D_{\eta}\chi(v),w\rangle_-
\end{equation*} gives \eqref{eq:hessh}.
  \end{proof}

We denote by $C^2_+(\H^d)$ the set of hyperbolic support functions  satisfying \eqref{eq:twdff}.

Since $\chi|_{T\H^d}$
 is the inverse of the Gauss map, it turns out that  if $\fconv$ is $C^2_+$ then $D\chi|_{T\H^d}$ is the inverse of
the shape operator.  
Now by Equation (\ref{eq:chi on hyp}) we have that if $v\in T_\eta\H^d$ then 
\[
      D\chi(v)=D(\grad^\H(\oh))(v)-D\oh(v)\eta-\oh v~.
 \]
 By Gauss--Weingarten formulas we deduce
\[
  D(\grad^\H(\oh))=\nabla_v(\grad^\H(\oh))+\II(\grad^\H(\oh), v)\eta
\]
where $\nabla$ is the Levi-Civita connection on $\H^d$ and $\II$ is the second fundamental form of the immersion
$\H^d\subset\mink$. Since the second fundamental form is the identity, putting together these formulas we have
\[
    D\chi|_{T\H^d}=\nabla(\grad^\H(\oh))-\oh \operatorname{Id}~.
\]
It follows that the \emph{radii of curvature} $r_i$ of $\fconv$ at some point $p$ are the  eigenvalues of $\nabla^2\oh-\oh g_{\H}$
 (computed with respect to the hyperbolic metric) at $\gauss(p)$. 
The \emph{curvature function} of $\fconv$ is \begin{equation}\label{eqref:det hess}
                                 \phi=\Pi_{i=1}^dr_i=\det(\nabla^2\oh-\oh g_{\H})~.
                                 \end{equation}
The \emph{mean radius of curvature} is defined as $\varphi=\frac{1}{d}\sum_{i=1}^dr_i=\frac{1}{d}\tr\left(\nabla^2\oh-\oh g_{\H}\right)$.

Both the curvature function and the mean radius of curvature are defined on $\H^d$. 
If $\gauss$ is the Gauss map of $\fconv$, $\frac{1}{\phi}\circ \gauss$
is the  Gauss--Kronecker curvature of $\partial_s\fconv$.

It is useful to express the curvature function and the mean radius of curvature in terms of the ball support function $h$.
\begin{lemma}
 Let $\fconv$ be an $F$-convex set, and $h$ be the corresponding support function on $\ball$.
 If $h$ is $C^2$ we have 
  \begin{eqnarray}
\ &&\label{eq: det hess plan}\det \mathrm{Hess}_x h = \lambda^{-d-2}(x) \phi(\rad(x)) \\
\ &&\label{eq: tr hess plan}\varphi(\rad(x)) = \frac{\lambda(x)}{d} (\tr\mathrm{Hess}_x h-\mathrm{Hess}_xh(x,x)) \nonumber
    \end{eqnarray}
with $\phi$ the curvature function  and $\varphi$ the mean radius of
curvature.
 \end{lemma}
 \begin{proof}
 We know $\det \nabla^2 \overline{h}-\overline{h}g_\H=\phi$  
 so, by \eqref{eq:hessh} $\det \hess h = \lambda^d \det(D\rad)^2 \phi$
and  we already know $\det(D\rad)$, see \eqref{eq:density vol ball}.

Let $\square$ be the wave operator on $\R^{d+1}$:  
$\square_{\hat{x}} H = \tr\mathrm{Hess}_x h - \mathrm{Hess}_{\hat{x}}H(e_{d+1},e_{d+1})$.
Writing $$e_{d+1}=e_{d+1}-\hat{x}+\hat{x},$$ as by homogeneity $\mathrm{Hess}_XH(X,\cdot)=0\;\forall X$, we have
$\mathrm{Hess}_{\hat{x}}H(e_{d+1},e_{d+1})=\mathrm{Hess}_x h(x,x)$.
\eqref{eq: tr hess plan} follows because
$\varphi(\rad(x))=\frac{1}{d}\square_{\hat{x}/\lambda} H$ and $\square_{\hat{x}/\lambda} H=\lambda \square_{\hat{x}} H$ by $(-1)$-homogeneity.
 \end{proof}
 }
\paragraph{Area measures.}

Let $\omega\subset\H^d$ be a Borel set. 
\newnew{\begin{definition}\label{defVeps}
We will denote by $V_{\epsilon}(\fconv)(\omega)$  the volume of the the set of points of $\fconv$ which are at Lorentzian distance at most $\epsilon$ from
 $\partial_s\fconv$ and such that the direction of their orthogonal projection onto $\partial_s \fconv$ is colinear to an element of $\omega$.
 \end{definition}}
The (Lorentzian) \emph{area measure} of $\fconv$ is a Radon measure on $\H^d$ defined by \index{Measure! $\A(\fconv)$ area measure of $\fconv$}
\begin{equation*}\label{eq: mink area}
\A(\fconv)(\omega)=\underset{\epsilon\downarrow 0}{\mathrm{lim}}
\frac{V_{\epsilon}(\fconv)(\omega)-V_{0}(\fconv)(\omega)}{\epsilon}=\underset{\epsilon\downarrow 0}{\mathrm{lim}}
\frac{V_{\epsilon}(\fconv)(\omega)}{\epsilon}~.
\end{equation*}
\begin{remark}{\rm
 In the convex body case, the area measure was introduced under the name area function by A.D.~Alexandrov. 
A difference with respect to the Euclidean case is that the Lorentzian projection on $\partial_s\fconv$ is well-defined
in the interior of $\fconv$ (see \cite{Bon05} for a discussion on this point). For this reason is convenient to  define $V_\epsilon(\fconv)(\omega)$ as
the volume of points on $\fconv$ at distance $\epsilon$ from the boundary instead of considering the exterior of $\fconv$.}

\end{remark}

If $\fconv$ has support function $\oh$ on $\H^d$, we will denote $\A(\fconv)$ by $\A(\oh)$.  If $\oh\in C^2_+(\H^d)$, then
 \begin{equation}\label{eq: area measure c2+}
  \A(\oh)=\det(\nabla^2 \oh-\oh g_\H)\d\H^d=\phi \d\H^d~,
 \end{equation}
where $\d\H^d$ is the volume form on $\H^d$ given by the hyperbolic metric and $\phi$ the curvature function.
Notice that in this case $\A(\oh)$ coincides with the direct image of the intrinsic volume form on 
$\partial_s\fconv$ through the Gauss map.

We need a result of weak convergence of the area measure. For completeness we will consider all kinds of area measures.
It is proved in \cite{fv} that there exist Radon measures $S_0(\fconv),\ldots,S_d(\fconv)$
on $\H^d$ such that, for any Borel set $\omega$ of $\H^d$ and any $\epsilon >0$,
\begin{equation}\label{eq: mu}V_{\epsilon}(\fconv)(\omega)=\frac{1}{d+1}\sum_{i=0}^d \epsilon^{d+1-i}\binom{d+1}{i} S_i(\fconv)(\omega)~.
\end{equation}
$S_i(\fconv)$ is called the \emph{area measure of order $i$} of $\fconv$. Note that $V_{\epsilon}(\fconv)$ is itself a Radon measure.
We have  that $S_0(\fconv)$ is given by the volume form of $\H^d$. 
The area measure $\A(\fconv)$ introduced above coincides with the area measure of order $d$, that is $S_d(\fconv)$.

\begin{lemma}\label{lem: weak conv are fconv}
 Let $(\oh_n)_n$ be support functions on $\H^d$ of F-convex sets converging uniformly on a compact set 
 $\omega$ to a support function $\oh$. 
 Then, for any continuous function $f:\H^d\rightarrow \R$ with compact support contained in the interior of $\omega$,
 for any $\epsilon >0$,
 $$\int_{\omega} f \d V_\epsilon(\fconv_n)\rightarrow \int_{\omega} f \d V_{\epsilon}(\fconv)  \qquad  \text{as $n\to\infty$ }~,$$
and, for $0\leq i \leq s$,
$$\int_{\omega} f \d S_i(\fconv_n)\rightarrow \int_{\omega} f \d S_i(\fconv) \qquad \text{as $n\to\infty$}~.$$
 \end{lemma}
\begin{proof}
The convergence for the measures $S_i(\fconv)$ follows from the convergence for $V_{\epsilon}(\fconv)$ and the linearity of the integral, using a polynomial interpolation
from  \eqref{eq: mu}. We have to prove the convergence  result for $V_{\epsilon}(\fconv)$.
This follows directly from the following facts, which are explicitly proved in \cite{fv}:
\begin{itemize}[nolistsep]
 \item the result is true if the F-convex sets are invariant under the action of a   same uniform lattice $\Gamma$,
 provided that $\omega$ is contained in a fundamental domain for the action of $\Gamma$;
 \item there exist a uniform lattice $\Gamma$ and $\Gamma$-invariant convex sets $\tilde{\fconv}_n$ and $\tilde{\fconv}$
 such that the set of points on $\fconv_n$ (resp. $\fconv$) with support normals in $\omega$
is, up to a translation, the set of points on $\tilde{\fconv}_n$ (resp. $\tilde{\fconv}$) with support normals in $\omega$.  As for any subset $b$, $V_{\epsilon}(\fconv)(b)$ depends only on the subset of $\partial \fconv$ of points
with support normals in $b$, and not on the whole $\fconv$, it follows that $V_{\epsilon}(\tilde{\fconv}_n)$
(resp. $V_{\epsilon}(\tilde{\fconv})$) coincide with $V_{\epsilon}(\fconv_n)$ (resp. $V_{\epsilon}(\fconv)$) on $\omega$;
 \item as $(\oh_n)$ uniformly converges to $\oh$ on $\omega$, the same is true for the support functions of $\tilde{\fconv}_n$ and $\tilde{\fconv}$.
\end{itemize}
\end{proof}

\old{

\subsection{Support functions restricted to the ball}

At a point $\eta$ of $\H^d\subset \mink$, consider the radial projection
from $\H^d$ to the affine plane $P=\eta+T_\eta\H^2$.
Let us denote by $\rad$  the inverse mapping $\rad:P\cap \fut(0)\to\H^d$.
We fix the point $\eta=(0,\ldots,0,1)$ and identify the affine horizontal plane $P$
passing through $\eta$ with $\R^d$ through the map $x\mapsto(x,1)$
Let us introduce the following notations
\begin{itemize}[nolistsep]
 \item $\ball$ is the open unit ball centred at the origin of $\R^d$, \index{$B$ the open unit ball centred at the origin of $\R^d$}
 \item $\hat{x}=(x,1)$ for any $x\in\R^d$, \index{$\hat{x}=(x,1)$}
 \item $\langle \cdot,\cdot\rangle $ the usual scalar product on $\R^d$, \index{$\langle \cdot,\cdot\rangle $ the usual scalar product on $\R^d$}
 \item $\|\cdot \|$ the associated norm,\index{$\|\cdot \|$ the usual norm on $\R^d$}
 \item $\lambda(x)=\sqrt{1-\|x\|^2}=\|\hat{x}\|_-$, $x\in \ball$,\index{$\lambda(x)=\sqrt{1-\|x\|^2}=\|\hat{x}\|_-$}
  \item $\mathcal{L} $ is the Lebesgue measure on $\R^d$. \index{Measure! $\mathcal{L} $  the Lebesgue measure}
\end{itemize}

Notice that $\ball$ corresponds exactly to the intersection $P\cap \fut(0)$. The radial map
$\rad$  from $\ball$ to $\H^d$ is  $\rad(x)=\frac{1}{\lambda}\hat{x}$, while \index{$\rad$ Radial map  $\rad(x)=\frac{1}{\lambda}\hat{x}$}
$\rad^{-1}$ maps 
$(x,x_{d+1})\in\H^d$ to $x/x_{d+1}=x/\sqrt{1+\|x\|^2}\in \ball$.
It can be useful to consider on $\ball$ the pull-back of the hyperbolic metric (we then get the Klein model of the hyperbolic space). If 
$\mathrm{d}\H^d$ is the volume element for the hyperbolic metric on $\ball$,
an explicit computation gives
\begin{equation}\label{eq:density vol ball}
 \d\H^d=\lambda^{-d-1}\L~.
\end{equation}

Now let $H$ be a convex $1$-homogeneous function on $\fut(0)$, and
consider its restriction to $\ball\times\{1\}$. By an abuse of terminology, we will call
the restriction of $H$ to $\ball$ the following map $h$ defined on $\ball$: \index{Support function! $h$ support function on $B$}\label{h}
 \begin{equation*}\label{eq: def oh}h(x)=H(\hat{x})~. \end{equation*}
 If $\overline{h}$ is the restriction of $H$ to $\H^d$, we then get
$$h(x)=\lambda(x)\overline{h}(\rad(x))~.$$
 
Of course a function $h$ obtained in this way is convex on $\ball$. 
The following lemma says that all convex functions on $\ball$ are obtained in
this way.  
 \begin{lemma}
A   convex function $h$ on $\ball$ is the restriction (in the sense above) of the
following sublinear 1-homogeneous function on $\fut(0)$:
 $$H\left((x,z)\right)=zh(x/z)~,$$ $x\in \ball,z\in\R$.
 \end{lemma}
\begin{proof}
The function $H$ is $1$-homogeneous by definition. So we only need to prove that it is sub-additive.
Take $X=(x,x_{d+1}), Y=(y,y_{d+1})\in \fut(0)$.
By definition, $$H(X+Y)=(x_{d+1}+y_{d+1})h\left(\frac{x+y}{x_{d+1}+y_{d+1}}\right).$$
Notice that the point $\frac{x+y}{x_{d+1}+y_{d+1}}$ can be expressed as the following convex combination
of $x/x_{d+1}$ and $y/y_{d+1}$:
$$     
\frac{x+y}{x_{d+1}+y_{d+1}}=\frac{x_{d+1}}{x_{d+1}+y_{d+1}}\frac{x}{x_{d+1}}\ +\ \frac{y_{d+1}}{x_{d+1}+y_{d+1}}\frac{y}{y_{d+1}}~.
$$
Since $h$ is convex we get that
\[
   H(X+Y)\leq(x_{d+1}+y_{d+1})
   \left(\frac{x_{d+1}}{x_{d+1}+y_{d+1}}h\left (\frac{x}{x_{d+1}}\right)\ +\ \frac{y_{d+1}}{x_{d+1}+y_{d+1}}
   h\left(\frac{y}{y_{d+1}}\right)\right)\]
$$   =H(X)+H(Y)~.$$

\end{proof}
\begin{corollary}
Any convex function on $\ball$ is the support function of a unique F-convex set of $\mink$.
\end{corollary}

By Lemma \ref{lm:c1supp}, if $H$  is $C^1$, then 
$\partial_s \fconv$ is the range of the function  $\chi:\fut(0)\to\R^{d+1}$ sending $X$ to $\grad^-_{X}H$.
We express both $\chi$ and curvature function in terms of the ball support function $h$.
 
 \begin{lemma}
 Let $H$ be a convex $1$-homogeneous function on $\fut(0)$. Let us denote by $h$ its restriction to $\ball$, and by $\oh$ its restriction to
 $\H^d$.
  If $H$ is $C^1$, then
  \begin{equation}
   \label{eq:chi on disc} \chi(\hat{x})=\grad_xh + (\langle x,\grad_xh\rangle -h(x))e_{d+1}
  \end{equation}
  If $H$ is $C^2$, then
  \begin{eqnarray}
\ &&\label{eq: hess plan}\mathrm{Hess}_x h (v,w)=\lambda(x) \left(\nabla^2 \oh-\oh g\right)(D_x\rad(v),D_x\rad(w)) \\
\ &&\label{eq: det hess plan}\det \mathrm{Hess}_x h = \lambda^{-d-2}(x) \phi(\rad(x)) \\
\ &&\label{eq: tr hess plan}\varphi(\rad(x)) = \frac{\lambda(x)}{d} (\tr\mathrm{Hess}_x h-\mathrm{Hess}_xh(x,x)) \nonumber
    \end{eqnarray}
with $\phi$ the curvature function  and $\varphi$ the mean radius of
curvature (the mean value of the radii of curvature).
 \end{lemma}

\begin{proof} 
It is clear that the projection of the gradient of $H$ to $\R^d$ is equal to  the gradient of $h$.
To compute the normal component, $$D_{\hat{x}}H(e_{d+1})=D_{\hat{x}}H(\hat{x})-D_{\hat{x}}H(\hat{x}-e_{d+1})=H(\hat{x})-
D_xh(x)$$ where $x\in \R^d$ is identified
with $(x,0)\in\R^{d+1}$. Formula~ \eqref{eq:chi on disc} follows.

By $0$-homogeneity of $\chi$ and because $\lambda \rad$ is the identity
$$h(x)=\langle \chi(\rad(x)), (\lambda \rad)(x)\rangle_-~.$$ 
 We compute 
$$D_x h (v) = \langle D_x(\chi\circ \rad)(v),(\lambda \rad)(x)\rangle_-+\langle \chi(\rad(x)),D_x(\lambda \rad)(v)\rangle_-
=\langle \chi(\rad(x)),D_x(\lambda \rad)(v)\rangle_-$$ 
 because  $\rad(x)$ is orthogonal to the tangent plane at the point $\chi (\rad(x)) $.
 For the same reason, and also because $\mathrm{Hess}_x(\lambda \rad)=0$ (because $\lambda \rad$ is the identity) we compute that
$$\mathrm{Hess}_x h (v,w)=\lambda(x) \langle D_{\rad(x)}\chi(D_x\rad(v)), D_x\rad(w)\rangle_-$$ that together with 
\begin{equation*}\label{eq:rel H h}
 \mathrm{Hess}_{\eta} H(v,w)=\nabla^2 \oh(v,w)-\oh g_\H(v,w)=\langle D_{\eta}\chi(v),w\rangle_-
\end{equation*} gives \eqref{eq: hess plan}.

We know $\det \nabla^2 \overline{h}-\overline{h}g_\H=\phi$  
 so $\det \hess h = \lambda^d \det(D\rad)^2 \phi$
and  we already know $\det(D\rad)$, see \eqref{eq:density vol ball}.

Let $\square$ be the wave operator on $\R^{d+1}$:  
$\square_{\hat{x}} H = \tr\mathrm{Hess}_x h - \mathrm{Hess}_{\hat{x}}H(e_{d+1},e_{d+1})$.
Writing $$e_{d+1}=e_{d+1}-\hat{x}+\hat{x},$$ as by homogeneity $\mathrm{Hess}_XH(X,\cdot)=0\;\forall X$, we have
$\mathrm{Hess}_{\hat{x}}H(e_{d+1},e_{d+1})=\mathrm{Hess}_x h(x,x)$.
\eqref{eq: tr hess plan} follows because
$\varphi(\rad(x))=\frac{1}{d}\square_{\hat{x}/\lambda} H$ and $\square_{\hat{x}/\lambda} H=\lambda \square_{\hat{x}} H$ by $(-1)$-homogeneity.
  \end{proof}

 \begin{example}\rm{ The support function of the convex 
 side of the upper-sheet of the hyperboloid with radius $t$ is $H(X)=-t\|X\|_-$, 
and its restriction to $\H^d$ is the constant function equal to $-t$. 
Its restriction to $\ball$ is $h(x)=H(\hat{x})=-t\lambda(x)$. Note that
 $\hess h=-t\hess \lambda$, so the convexity of $h$  implies that $\hess \lambda$ is semi-negative definite.
 
We will denote by  $ \jplus(\H)$ the convex 
 side of the upper-sheet of the hyperboloid with radius $1$. 
 }
\end{example}

\begin{example}\rm{The restriction to $\fut(0)$ of the linear functionals
of $\R^{d+1}$ are the support functions on $\fut(0)$ of the future cone of points.
Their restriction to $\ball$ are exactly the restriction to $\ball$ of the 
affine maps of $\R^d$.
}
\end{example}

From \eqref{eq: hess plan}, support functions of $C^2_+$ F-convex sets are exactly 
the $C^2$  maps on $\ball$ with positive definite Hessian.

}
 \subsection{F-convex sets as graphs}
 Let $\fconv$ be an F-convex set. 
 Since $\fconv$ has no time-like support plane, each time-like line 
 must meet $\fconv$. Moreover since $\fconv$ is future complete ($\fconv=\fut(\fconv)$)
the intersection must be a future-complete half-line.
So  each time-like line meets $\partial \fconv$ exactly at one point.
This shows that $\partial \fconv$ is a graph
of a function $\tilde u$ defined on the horizontal plane $\R^d\subset\R^{d+1}$.
Convexity of $\fconv$ implies that $\tilde u$ is a convex function.
 
Let $h$ be the support function of $\fconv$ on $\ball$. In this section we investigate the relation between 
$h$ and $\tilde u$. To this aim it is convenient to consider the extension of $h$ to $\R^d$ defined as 
$$\tilde{h}:\R^{d}\rightarrow \overline{\R}:=\R\cup \{+\infty\}~,\qquad \tilde{h}(x)=\sup\left\{\left\langle k,\hat{x}\right\rangle_-\vert k \in \fconv\right\}~.$$
It is a proper convex lower semi-continuous function (as a sup of continuous functions). It takes values 
$+\infty$ out of $\overline{\ball}$ and is equal to $h$ on $\ball$. Its value on 
$\partial \ball$ may be finite or infinite. In other terms, $\tilde{h}$ is the lower semi-continuous hull of $h$. If $\tilde{h}$ has a finite value at $\ell\in \partial \ball$, then
$\fconv$ has a light-like support plane ``at infinity'' directed by $\ell$. This means that any parallel displacement
of the hyperplane in the future direction will meet the interior of $\fconv$. The hyperplane may or may not meet $\fconv$.

By properties of lower semi-continuous convex functions, for any $\ell\in\partial \ball$ and any $x\in \ball$, we have 
\begin{equation}\tilde{h}(\ell)=\underset{t\downarrow 0}{\mathrm{lim}}\: h(\ell+t(x-\ell))~. \end{equation}
Let us write
$$\tilde{h}(x)=\sup_{(p,p_{d+1})\in \fconv}\left\langle \binom x 1,\binom{p}{p_{d+1}} \right\rangle_- $$
$$=\sup_{(p,p_{d+1})\in \fconv}\left\{ \langle x,p\rangle - p_{d+1}\right\} =\sup_{p\in \R^d}\left\{ \langle x,p\rangle - \tilde{u}(p)\right\} $$
where $\tilde{u}$ is the graph of $\partial \fconv$. From this simple relation we immediately get

\begin{definition}
The function  $\tilde{h}$  is the \emph{dual} of $\tilde u$. The correspondence $\tilde u\mapsto \tilde h$
is the \emph{Legendre--Fenchel transform}.
\end{definition}

Conversely by the involution property of the Legendre--Fenchel transform we have that $\tilde{u}$ is the dual of $\tilde h$, that is
\[
     \tilde{u}(p)=\sup_{x\in\R^d}\left(\langle p,x\rangle-\tilde{h}(x)\right)
\]
for all $p\in\R^d$.

Since the dual of $\tilde u$ (which is $\tilde{h}$) takes finite values only on $\ball$, it turns out that
$\tilde u$ is $1$-Lipschitz. Indeed a general fact is that proper future convex sets are exactly the graphs of $1$-Lipschitz
convex functions
\cite[Lemma~3.11]{Bon05}. 
 As $\tilde{u}$ is convex and $1$-Lipschitz, by Rademacher theorem, 
 $\grad \tilde{u}$ exists and is continuous almost everywhere, and
 its norm is less than $1$.

\begin{remark}
Given a convex $1$-Lipschitz function $\tilde u:\R^d\to\R$, its graph bound a convex set $\fconv$ that coincides with its future.
For $x\in\ball$ the point $\rad(x)$ lies in the image of the Gauss map of $\fconv$ iff $\tilde h$ is finite at $x$.
So $\fconv$ is an F-convex set iff the function $\tilde h$, dual of $\tilde u$, takes finite values on $\ball$.
\end{remark}

\old{ 
Suppose now that $h$ is $C^1$. 
Putting  $y=\grad_xh\in\R^d$, \eqref{eq:chi on disc}  precisely says that 
$\partial_s\fconv$   is the graph over $\grad h(\ball)\subset \R^d$  of the function 
$u(y)=\langle x,y\rangle -h(x)$.
So $(u,\grad h(\ball))$ is the \emph{Legendre transform} of $(h,\ball)$. 
Actually $u$ is nothing but the
 the restriction of $\tilde u$ on $\grad h(\ball)$ \cite[26.4]{Roc97}. 
Note that  $\grad h(\ball)$ --- the definition domain of $u$ --- is not necessarily the whole $\R^d$. 
It can be for example reduced to a point (e.g. taking $h=0$).
}

\new{
%   Let $h$ be a convex function on $\ball$ and let $\tilde h$ its extension to $\R^d$.
    For a convex function $f$ defined on some convex domain $U$ of $\R^d$, the \emph{normal mapping}, or \emph{subdifferential} of $f$ at $x\in U$, 
    denoted by $\partial f(x)$ \index{ $\partial h(x)$ the normal mapping}, is a subset of $ \R^d$ 
 defined as follow: it is the horizontal projection onto $\R^d$ of the set of inward (Lorentzian) normals with
 last coordinate equal to one of the support hyperplanes 
 of the graph of $f$ at the point $(x,f(x))$. In other terms,
 $p\in \partial f(x)$ if and only if, for all $y\in U$, $\langle \binom p 1, \binom{y-x}{f(y)-f(x)}\rangle_- \leq 0$, i.e.
$f(y)\geq f(x)+\langle p, y-x\rangle$. It turns out that $\partial f(x)$ is a convex set in $\R^d$, \cite{Gu01}.
If  $\omega$ is a Borel set  of $U$,  we  set $\partial f (\omega):=\cup_{x\in\omega} \partial f(x)$.

Let $h$ be a convex function on $\ball$ and $\tilde h$ be its extension to $\R^d$.
Note that for $x\in \ball$, $\partial \tilde{h}(x)=\partial h(x)$.

If $\tilde{u}$ is the conjugate of $\tilde{h}$, we have \cite[11.3]{RW98}
\begin{equation}\label{eq:inv partial}\partial \tilde{u} = \left( \partial \tilde{h} \right)^{-1} \end{equation}
that has to be understood as: 
 $p\in \partial \tilde{h}(x)$ if and only if $x\in \partial\tilde{u}(p)$

From \eqref{eq:inv partial} we simply deduce the following relation between $\partial h$ and the Gauss map
$\gauss$ of the domain supported by $h$.

\begin{lemma}
Let $\gauss$ be the Gauss map of the F-convex set supported by $h$.
Let $\proj$ be the orthogonal projection from $\R^{d+1}$ onto  the horizontal plane. Then for any Borel set $\omega$ of $\H^d$
 \begin{equation}\label{eq: proj subdif}
 \partial h(\rad^{-1}(\omega))=\proj(\gauss^{-1}(\omega))~.
\end{equation}
\end{lemma}
\begin{proof}
By \eqref{eq:inv partial}  $p\in \partial h(x)$ if and only if
  $p\in \partial \tilde{h}(x)$ if and only if $x\in \partial\tilde{u}(p)$ if and only if 
 $\hat{x}$ is an inward unit normal to $\fconv$ at $(p,\tilde{u}(p))$ if and only if $(p,\tilde{u}(p))\in \gauss^{-1}(\rad(x))$, 
if and only if $p\in \proj(\gauss^{-1}(\rad(x)))$. 
\end{proof}

As a corollary the projection to $\R^d$ of $\partial_s\fconv$ is $\partial h(\ball)=\bigcup_{x\in\ball}\partial h(x)$.
We will denote by $u$ the restriction of $\tilde u$ on $\partial h(\ball)$, so that $\partial_s\fconv$ is the graph of $u$.

If $h$ is $C^1$, then $\partial h$ is single-valued and coincides with the gradient map, so it is a  continuous mapping from $\ball$ to
$\partial\fconv$. If $h$ is also strictly convex, then the gardient map is injective so $\partial_s\fconv$ is open in $\partial\fconv$.
However there are examples of smooth strictly convex functions $h$, so that $\partial_s\fconv\neq\partial\fconv$. 
  }
A criterion to have  $\grad h(\ball)=\R^d$ is that $\partial_s \fconv$ is $C^1$ and the induced Riemannian metric is complete \cite[Lemma~3.1]{Bon05}.  
Space-like hypersurfaces which are graph over $\R^d$ are usually called \emph{entire}.
Entire hypersurfaces are not necessarily graphs of F-convex sets. A condition for a $C^2$ entire space-like hypersurface to bound an F-convex set can 
be found in  \cite{li95}.

If $h$ is $C^1$, strictly convex, and $\grad h(\ball)=\R^d$, then $h$ is of \emph{Legendre type}, and the 
gradient of $u$ is the inverse of the gradient of $h$ at the corresponding points \cite[26.5]{Roc97}. 
In particular, if  moreover the Hessian of $h$ is positive definite, then, looking the Hessian as the Jacobian of the gradient, 
the Hessian matrix of $h$ is the inverse of the Hessian matrix of $u$ at the corresponding points.
So  from \eqref{eq: det hess plan} we recover the well-known formula for the Gauss--Kronecker
curvature of $\partial_s\fconv$:
\newnew{$$\phi^{-1}=\left(1-\|\grad u\|^2\right)^{-1-d/2}\det \hess u~. $$}
\oldold{$$\phi^{-1}=\left(\sqrt{1-\|\grad u\|^2}\right)^{-1-d/2}\det \hess u~. $$}

\begin{lemma}\label{lem: conv unif sur compt}
 Let $H$ be a support function of an F-convex $\fconv$ set on $\fut(0)$. 
 Then there exists a sequence of $C^{\infty}_+$ support functions converging to $H$, uniformly on 
 any compact set.
\end{lemma}
\begin{proof}
 Let $u:\R^d\to\R$ be the convex function whose graph is $\partial \fconv$.
 Let $\phi_n$ be smooth functions on $\R^d$, with compact support in a ball of radius $1/n$, 
 and with $\int_{\R^d }\phi_n \d x=1$. By the convolution properties,
 $$u_n(x)=(u\ast \phi_n)(x)=\int_{\R^d}u(x-y)\phi_n(y)\d y~, $$
are smooth  functions converging uniformly to $u$. 
More precisely since $u$ is $1$-Lipschitz, it is straightforward to check that
$|u(x)-u_n(x)|<1/n$.

\emph{Fact 1:  The graph of $u_n$ bounds an F-convex set.}
 Notice that $u_n$ is convex.
 Indeed, by convexity of $u$
 \begin{align*}
  u_n((1-t)x_1+tx_2)=\int_{\R^d}u((1-t)(x_1-y)+t(x_2-y))\phi_n(y)\d y\leq\\ \leq \int_{\R^d}(1-t)u(x_1-y)\phi_n(y)\d y +\int_{\R^d}t u(x_2-y)\phi_n(y)\d y=(1-t)u_n(x_1)+tu_n(x_2)~.
 \end{align*}
 
 Let $y\in \ball$.  
 As the graph of $u$ is an F-convex set, there exists $\alpha\in \R$ such that
for any $x\in \R^d$, 
$\langle \binom{x}{u(x)},\binom{y}{1}\rangle_-=\langle x,y\rangle - u(x)\leq \alpha$.
Since $|u(x)-u_n(x)|<1/n$ we deduce
$\langle \binom{x}{u_n(x)},\binom{y}{1}\rangle_-\leq \alpha+1/n$,
i.e. there is a support plane orthogonal to $(y,1)$.
\qd 

The functiont $u_n$ is smooth but its Hessian is not strictly positive, so in principle the inverse of the 
 Gauss map of the corresponding graph
 is not differentiable.
To fix this problem we consider the functions
  $$ v_n=\left(1-\frac{1}{n}\right)u_n+\frac{1}{n}f(x)$$
with $f(x)=\sqrt{1+\|x\|^2}$ the graph function of the upper hyperboloid. 
Clearly the Hessian of $v_n$ is strictly positive. On the other hand 
$v_n$ converges  to $u$ uniformly only on compact sets, so we cannot use the simple argument
above to show that the graph of $v_n$ bounds an F-convex set.
This problem is balanced by the following observation.

\emph{Fact 2: $(u_n-f)$ is uniformly bounded from above.}
The F-convex set bounded by the graph of $u_n$ contains the future cone of any of its points,
 so there exists a constant $M$ with $u_n(x)<M+\|x\|$, and $f(x)>\|x\|$,  and the result follows. \qd

 \emph{Fact 3:  $v_n$ is the graph of the boundary of an F-convex set.}
We want to prove that, given $y \in \ball$,  there exists $\alpha$ such that for any $x\in \R^d$,
$$\left\langle \binom{x}{v_n(x)},\binom{y}{1}\right\rangle_-\leq \alpha~. $$
The left-hand side of the above equation
is
$$\langle x,y\rangle-u_n(x)+\frac{1}{n}(u_n-f)(x)~. $$
 As $u_n$ bounds an F-convex set by Fact~1, there exists $\alpha'$ such that
 $\langle x,y\rangle-u_n(x)<\alpha'$. The result follows by Fact~2. \qd

 \emph{Fact 4: The support functions $h_n$ of the graph of $u_n$ converge pointwise to the support function 
 $h$ of the graph of $u$.}
Let $y\in \ball$, and let $\binom{x}{u(x)}$ be a point on the intersection of $\fconv$ and its support plane  orthogonal to $(y,1)$. 
Let  $\epsilon >0$.  First we show that  for $n$ sufficiently large, $h_n(y)\geq h(y)-\epsilon$.
Notice that
$$h_n(y)\geq\left\langle \binom{x}{v_n(x)},\binom{y}{1}\right\rangle_- =\left\langle \binom{x}{u(x)},\binom{y}{1}\right\rangle_-+(u(x)-v_n(x))~. $$
As  
$$h(y)=\left\langle \binom{x}{u(x)},\binom{y}{1}\right\rangle_-$$
we need
$$v_n(x)-u(x)<\epsilon~. $$
This is  true  for $n$ big as $v_n$ converges to $u$.

Let us prove that for $\epsilon >0$ and  $\forall x\in \R^d$,
$$\left\langle \binom{x}{v_n(x)},\binom{y}{1}\right\rangle_- < h(y)+\epsilon~.  $$
We have
$$\left\langle \binom{x}{u(x)},\binom{y}{1}\right\rangle_- < h(y)+\epsilon~,  $$
that gives, adding $v_n(x)-v_n(x)$ on the left-hand side and reordering the terms
$$\left\langle x,y\right\rangle_--v_n(x)<h(y)+u(x)-v_n(x)~. $$
But $u(x)-v_n(x)=u(x)-u_n(x)+\frac{1}{n}(u_n-f)(x)$, and the result follows by Fact~2 and because $u_n$ uniformly
converge to $u$.
\qd

By homogeneity, we obtain the 
convergence on the support functions on $\fut(0)$.
A classical property of convex functions says that the convergence is actually uniform on each compact set  \cite{HUL93,Roc97}.
 \end{proof}

\subsection{Area measure and Monge--Amp\`ere measure}

\old{
    Let $h$ be a convex function on $\ball$.
    For $x\in \ball$, the \emph{normal mapping}, or \emph{subdifferential} of $h$ at $x$, 
    denoted by $\partial h(x)$ \index{ $\partial h(x)$ the normal mapping}, is a subset of $ \R^d$ 
 defined as follow: it is the horizontal projection onto $\R^d$ of the set of inward (Lorentzian) normals with
 last coordinate equal to one of the support hyperplanes 
 of the graph of $h$ at the point $(x,h(x))$. In other terms,
 $p\in \partial h(x)$ if and only if, for all $y\in \ball$, $\langle \binom p 1, \binom{y-x}{h(y)-h(x)}\rangle_- \leq 0$, i.e.
$h(y)\geq h(x)+\langle p, y-x\rangle$.

 If  $\omega$ is a Borel set  of $\ball$, it turns out that the set $\partial h (\omega):=\cup_{x\in\omega} \partial h(x)$.
is still a Borel set of $\R^d$ \cite{Gu01}.

\begin{lemma}
Let $\gauss$ be the Gauss map of the F-convex set supported by $h$.
Let $\proj$ be the orthogonal projection from $\R^{d+1}$ onto  the horizontal plane. Then
 \begin{equation}\label{eq: proj subdif}
 \partial h(\rad^{-1}(\omega))=\proj(\gauss^{-1}(\omega))~.
\end{equation}
\end{lemma}
\begin{proof}
Let $\tilde{h}$ be the extension of $h$ to $\R^d$. The definition of $\partial \tilde{h}$ is clear.
Note that for $x\in \ball$, $\partial \tilde{h}(x)=\partial h(x)$.
If $\tilde{u}$ is the conjugate of $\tilde{h}$, we have \cite[11.3]{RW98}
\begin{equation}\label{eq:inv partial}\partial \tilde{u} = \left( \partial \tilde{h} \right)^{-1} \end{equation}
that has to be understood as: 
 $p\in \partial \tilde{h}(x)$ if and only if $x\in \partial\tilde{u}(p)$. So  $p\in \partial h(x)$ if and only if
  $p\in \partial \tilde{h}(x)$ if and only if $x\in \partial\tilde{u}(p)$ if and only if 
 $\hat{x}$ is an inward unit normal to $\fconv$ at $(p,\tilde{u}(p))$ if and only if $(p,\tilde{u}(p))\in \gauss^{-1}(\rad(x))$, 
if and only if $p\in \proj(\gauss^{-1}(\rad(x)))$. 
\end{proof}
}
Let $h$ be a convex function on $\ball$.
 The \emph{Monge--Amp\`ere measure} $\M(h)$ of $h$ is a Borel measure over $\ball$ defined by \index{Measure! $\M$ Monge--Amp\`ere measure}
$$\M(h)(\omega)=\L(\partial h(\omega))~, $$
where $\omega$ is any Borel subset of $\ball$.
The Monge--Amp\`ere measure of $h$
 is  finite
 on compact sets \cite[1.1.13]{Gu01}. As $\ball$ is
$\sigma$-compact (i.e., a countable
union of compact sets),  the Monge--Amp\`ere measure is a Radon measure on $\ball$.

Let us list some immediate properties.
\begin{itemize}
\item If $h_1$ and $h_2$ are convex functions, and coincide over $\omega$, then $\M(h_1)(\omega)=\M(h_2)(\omega)$. 
\item For $c> 0$, $\partial (c h)(x)=c\partial h(x)$, so
\begin{equation}\label{eq:MA homothetie}
 \M(c h)=c^{d} \M(h)~.
\end{equation}
\item If $A$ is an affine function, then $\partial (h+A)(\omega)$ is a translate
of $\partial (h)(\omega)$, so
\begin{equation}\label{eq:MA affine}
 \M( h+A)= \M(h)~.
\end{equation}
\item Let $h=\operatorname{max}\{h_1,h_2\}$ with $h_i$ a convex function on $\ball$.
Denote by $D_1$ the set where $h_1> h_2$ and $D_2$ the set where $h_2>h_1$.
Notice that $D_i$ are open and for each $x\in D_i$ the set of subdifferentials of  $h$ at $x$ coincides
 with the set of subdifferentials of $h_i$.
 Now let $E$ be where $h_1$ coincides with $h_2$. If $x\in E$ each subdifferential of $h_i$ at
 $x$ is also a subdifferential of $h$ at $x$.  It follows that
 \begin{equation}\label{eq: MA max}
  \M(\operatorname{max}\{h_1,h_2\})(\omega)\geq \operatorname{min}\{\M(h_1)(\omega),\M(h_2)(\omega)\}~.
 \end{equation}
\item If $h$ is $C^1$, then $\partial h(x)=\grad_xh$, and then
$$\M(h)(\omega)=\L(\grad h(\omega))~, $$
for which it can be deduced that  \cite[1.1.14]{Gu01}, if $h$ is $C^2$
\begin{equation}\label{eq:ma c2}\M(h)(\omega)= \int_{\omega}\det \mathrm{Hess}_x h \d x~.
\end{equation}
\item If $h$ is a convex function on $\ball$, by the Alexandrov Theorem, it is twice differentiable almost everywhere
and then the right-hand side of \eqref{eq:ma c2} is still meaningful. It is actually 
 the regular part of the Lebesgue decomposition of $\M(h)$ \cite[Lemma~2.3]{TW08}.
So, for a general convex function $h$ we get
\begin{equation}\label{eq:ma c3}\M(h)(\omega)\geq \int_{\omega}\det \mathrm{Hess}_x h \d x~.
\end{equation}
 \end{itemize}

Let $h$ be a convex function on $\ball$ and $\fconv$ be the corresponding F-convex set.
There is a precise relation between the Monge--Amp\`ere measure of $h$ and the area
measure $\A(\fconv)$.
To this aim it is convenient to  push forward the area measure $\A(\fconv)$  to $\ball$ by the radial
map $\rad^{-1}:\H^d\to \ball$. We denote by $\A(h)$ the corresponding area measure on $\ball$:
$$\A(h)(\omega)=\A(\fconv)(\rad(\omega))~. $$

 \begin{proposition}\label{prop: area=ma}
 Let $h$ be a convex function on $\ball$. Then $\A(h)$ and $\M(h)$
 are equivalent. More precisely,
  $$\A(h)=\lambda \M(h)~,\;  \M(h)=\frac{1}{\lambda} \A(h)~.$$
 \end{proposition}

\begin{proof}
The second relation will follow from the first one, because $\lambda$ is positive on $\ball$.
 Let $h\in C^2(\ball)$ with positive definite Hessian.
From \eqref{eq: area measure c2+}, \eqref{eq: det hess plan} and \eqref{eq:density vol ball}
 we obtain $$\A(h)=\lambda \det (\hess h) \mathcal{L}$$ and the result follows from  \eqref{eq:ma c2}.
 
Now let $h$ be convex on $\ball$. From Lemma~\ref{lem: conv unif sur compt},
we can find  a sequence $(h_n)$ of elements of $C^2(\ball)$ with positive definite Hessian,
 converging uniformly to $h$ on compact sets. From \cite[1.2.3]{Gu01}, $\M(h_n)$ weakly converges  to $\M(h)$.
 From Lemma~\ref{lem: weak conv are fconv}, $\A(h_n)$ also weakly converge to $\A(h)$. 
 Since $\A(h_n)=\lambda \M(h_n)$, the result follows by the uniqueness part 
 of the Riesz Representation Theorem. 
\end{proof}

\begin{corollary}\label{cor: relation area ma}
 Let $c>0$. Then on $\ball$
 $$\A(h)\geq c (\rad^{-1})_*(\d\H^d) \Rightarrow \M(h)\geq c \L~. $$
 Conversely if $\M(h)\geq c\L$ then on any ball of radius $r<1$, 
 $\displaystyle\A(h)\geq c(1-r^2)^{(d+2)/2}\d\H^d$.
\end{corollary}

Let us denote by $\partial_{\mathrm{reg}}\fconv$ \index{$\partial_{\mathrm{reg}}\fconv$ }
the regular part of the  space-like boundary of $\fconv$, that is the set of 
points such that the Gauss map $\gauss$ of $\fconv$ is well defined. Equivalently $\partial_{\mathrm{reg}}\fconv$ 
is the differentiable part of $\partial_s\fconv$, i.e.~the set of points where $\fconv$ admits a  unique space-like support plane. 
 Notice that $\partial_{\mathrm{reg}}\fconv$ is the intersection of the region of differentiable points of $\partial\fconv$ and
 $\partial_s\fconv$. In general it can be strictly   smaller than the the region of differentiable points in 
the boundary of $\fconv$. For instance if $\fconv=\fut(0)$ then $\partial_{\mathrm{reg}}\fconv$ is empty.

  \begin{proposition}\label{prop: area zero not c1}
   Let $\fconv$ be an F-convex set. The Gauss image
   of  $\partial_s\fconv\setminus \partial_{\mathrm{reg}} \fconv$  has zero area measure. 
    \end{proposition}
  \begin{proof}
  Let $\hat E$ be the Gauss image in $\H^d$ of the non-regular points of $\partial_s\fconv$, and denote by
  $E$ its radial projection to $\ball$.
  By Proposition~\ref{prop: area=ma} and the definition of the Monge--Amp\`ere measure, it suffices to show that 
  the Lebesgue measure of $\partial h(E)$ is zero.
 From  \eqref{eq: proj subdif}, $\partial h(E)$ is the projection onto $\R^d$ of the non-differentiable points of $\partial_s \fconv$.
As $\partial \fconv$ is the graph of a convex function, the set of non-differentiable points has zero measure.
  \end{proof}

Notice that Proposition \ref{prop: area zero not c1}  implies that if $\partial_{\mathrm{reg}}\fconv$ is empty, the area measure is $0$. 
In Section~\ref{sec:ext sphere} we will give a precise characterization of F-convex sets with empty regular boundary.

Another simple remark implied by Proposition~\ref{prop: area zero not c1} is the following:

\begin{corollary}\label{cor: area grad}
 Let $\fconv$ be an F-convex set,  with Gauss map $\gauss$ and $\partial \fconv$ the graph of a function $\tilde{u}$. Then 
 $$\A(\fconv)(\omega)=\int_{\proj(\gauss^{-1}(\omega))} \sqrt{1-\|\grad \tilde{u}\|^2} \d x~,$$
 where $\proj:\mink\to\R^d$ is the orthogonal projection.
\end{corollary}
\begin{proof}
 Suppose that $\tilde{u}$ is $C^1$   and
 $h$ is the support function of $\fconv$ on $\ball$. Let $b=\rad^{-1}(\omega)$.
  Then $$\M(h)(b)=\mathcal{L}(\partial h (b))=\mathcal{L}(\partial \tilde{h} (b)) 
  \stackrel{\eqref{eq:inv partial} }{=}\mathcal{L}((\grad \tilde{u})^{-1}(b))=
  (\grad \tilde{u})_{*} \mathcal{L}(b)~,$$
 so 
 $$\A(\fconv)(\omega)=\int_b \lambda \d (\grad \tilde{u})_{*} \mathcal{L} = \int_{\partial h (b)} (\lambda\circ \grad \tilde{u}) \d x 
 = \int_{\proj(\gauss^{-1}(\omega))} \sqrt{1-\|\grad \tilde{u} \|^2} \d x~. $$
 
Now for any convex $h$, $\tilde{u}$ is convex and by Rademacher Theorem, $\| \grad \tilde{u} \|$ is continuous except on 
a zero measure set $N$. 
Let us define the following measure on $\H^d$:   $$\mu(\omega)=\int_{\proj(\gauss^{-1}(\omega)) } \sqrt{1-\|\grad \tilde{u} \|^2} \d x~.$$ 
By Proposition~\ref{prop: area zero not c1} $\A(\fconv)(N)=0$, so $\A(\fconv)$ and $\mu$ coincide on $N$.
Out of $N$ they also coincide by the argument above.  
\end{proof}

 \begin{corollary}\label{cor: area C1}
  Let $\fconv$ be an F-convex set with $\partial_s \fconv$ a $C^1$ hypersurface. Then for  any Borel set $\omega$ of $\H^d$, $\A(\fconv)(\omega)$ is the measure of the inverse image of $\omega$ by the Gauss map,
  for the measure given by the induced Riemannian metric on $\partial_s \fconv$.
 \end{corollary}

  \begin{proof}
  Let $\partial_s \fconv$ be the graph of $u$. For two tangent vectors $U,V$ in $\R^d$, the first fundamental form 
of $\partial_s\fconv$ is given by
$$\operatorname{I}(U,V)=\left\langle \binom{U}{\langle \grad_xu,U\rangle},\binom{V}{\langle \grad_xu,V\rangle} \right\rangle_-=\langle U,V\rangle - \langle \grad_x u,U\rangle\langle \grad_xu ,V\rangle$$
   so $\operatorname{I}=\mathrm{Id}_d-\d u\otimes \d u$. So the area element is $\sqrt{\det \operatorname{I}}\mathcal{L}=\sqrt{1-\|\grad u\|^2}\mathcal{L}$ and the result follows from Corollary~\ref{cor: area grad}. 
  \end{proof}

  \begin{remark}[\textbf{Euclidean area}]\rm{ \label{rem: euclidean area} To define the Monge--Amp\`ere measure,
 we used the Lorentzian normal vectors. Classically the definition is given by taking the Euclidean outward unit normal vectors with 
last coordinate equal to $-1$. 
 Actually, if $\hat{x}=\binom{x}{1}$ is a Lorentzian  inward normal vector  of some space-like hypersurface $S$ in $\R^{d+1}$,  
 then $\binom{x}{-1}$ is a Euclidean outward  normal vector.
 So the definition of the Monge--Amp\`ere measure is independent of this choice.
 
 Considering a convex function $h$ on $\ball$, it can be seen as the restriction to $\ball \times \{1\}$ of the Lorentzian extended support function 
 of an F-convex set $\fconv$, or the restriction to $\ball\times \{-1\}$ of the Euclidean extended support function 
 (defined on the past cone) of an unbounded convex set $\fconv'$.
 Actually $\fconv=\{p\in \R^{d+1}|\langle p, \binom x 1 \rangle_- \leq h(x), \forall x\in \ball\}$ and  $\fconv'=\{p\in \R^{d+1}|\langle p, \binom{x}{-1} \rangle_{d+1} \leq h(x), \forall x\in \ball\}$,
 with $\langle \cdot, \cdot \rangle_{d+1}$ the usual scalar product of $\R^{d+1}$, so $\fconv=\fconv'$.
 
  If we denote by $\A^e(h)$ the Euclidean area measure of $\fconv$, we have in the same manner as above (or see e.g.~\cite{CY77})
 $$\A^e(h)=\lambda^e \M (h)~,\qquad\text{where } \lambda^e(x)=\sqrt{1+\|x\|^2}~.$$ 
 In particular the Euclidean area measure and the Lorentzian area measure are equivalent. 
 Heuristically, this is not surprising: if $\omega$ is  a Borel set of $\ball$,  the pre-image
 by the Lorentzian or the Euclidean Gauss map of $\omega$ gives the same set on $\partial \fconv$.
 This because the hyperplane  defined by the equation  $\langle p,\hat{x}\rangle_-=h(x)$ coincides with  the 
 hyperplane  defined by the equation $\langle p,\hat{x}\rangle_{d+1}=h(x)$.  
 
 The fact that  area measure is null on $\omega$ means that for any small parallel displacement of the support planes defined by 
 $\omega$, the first order variation of the volume is zero. Parallel displacement and volume 
 determined on $\R^{d+1}$ by the Euclidean metric and by the Minkowski metric coincide
  (although the parallel  displacement used in the definition of the area measure
   is not in the same direction in each case).
 It also follows that the Lorentzian area measure and the Hausdorff measure are equivalent
 \cite[Theorem~4.2.5]{sch93}.
 }  \end{remark}

\subsection{Support functions with continuous extension on the sphere}\label{sec:ext sphere}
 
Let $g\in C^0(\partial \ball)$. A \emph{$g$-convex set}\index{Convex set! $g$-convex set}\label{g} is an F-convex set $\fconv$ with  support
function $\tilde{h}$  on $\R^d$ such that $\tilde{h}_{|\partial \ball}=g$.
  Not all F-convex sets are $g$-convex sets for some $g$ as $h=\lambda^{-1}$ shows. The following result gives 
 a sufficient condition. %For completness we reproduce the proof.
 \begin{proposition}[{\cite{li95}}]\label{prop: li}
  Let $\fconv$ be a $C^2_+$ F-convex set with  mean radius of curvature bounded from above.
  Then $\tilde{h}_{|\partial \ball}$  is continuous.
 \end{proposition}

For any $g\in C^0(\partial \ball)$, there is a distinguished $g$-convex set, that can be introduced in many equivalent ways. 
They are listed in the following proposition.

\begin{proposition}\label{prop:a=0}
Let $\fconv$ be an F-convex set with support function $h\in C^0(\overline \ball)$, Gauss map $\gauss$, and set $g=h|_{\partial \ball}$. 
The following are equivalent:
\begin{enumerate}[nolistsep]
\item $h$ is the \emph{convex envelope} of $g$, i.e.
$h(x)=\mathrm{sup}\{l(x) | l \mbox{ affine and } l\leq g \mbox{ on } \partial \ball \} ;$ \label{car1}
\item $\A(\fconv)=0$;\label{car2}
\item $\partial_{\mathrm{reg}}\fconv$ is empty; \label{car3}
\item $\forall p\in \partial_s \fconv$, $\rad^{-1}(\gauss(p))$ is the convex hull of points of $\partial \ball$;\label{car4}
\item $\fconv=\bigcap_{\ell \in \partial \ball} \overline{\fut(P_g(\ell))}$, 
with $P_g(\ell)$ are the null hyperplane defined by $\langle \cdot,\hat{\ell}\rangle_- = g(\ell)$;
 \label{car6}
\item any $g$-convex set is contained in $\fconv$; \label{car5}
\item the interior of $\fconv$ union $\partial_s \fconv$ is the \emph{domain of dependence} (or Cauchy development) of 
the boundary $S$ of any $g$-convex set: it is the set of points $p$ of $ \mink$ 
such that all inextensible causal curve through $p$ meet 
$S$. \label{car7}
\end{enumerate}
\end{proposition}

\begin{proof}

(\ref{car1}$\Leftrightarrow$\ref{car2}) By Proposition \ref{prop: area=ma}, $\A(\fconv)=0$ is equivalent to  $\M(h)=0$, which
happens exactly when  $h$  is the convex envelope of $h|_{\partial \ball}$ (see \cite[1.5.2]{Gu01}).

(\ref{car1}$\Rightarrow$\ref{car4}) 
Let $x\in \rad^{-1}(\gauss(p))=:C(p)$.
Up to composing  $h$ by an affine function we may suppose that $h=0$ on $C(p)$ and $g\geq 0$ (it suffices to translate $p$ to the origin). 
Suppose that $g\geq c >0$. Then by definition of $h$, $h\geq c$, that contradicts the assumption. 
So  $g^{-1}(0)=:\Lambda$ is not empty.
As $C(p)$ is convex, its closure contains the convex hull $C(\Lambda)$ of $\Lambda$. 
We want to prove that the closure of $C(p)$ is contained in $C(\Lambda)$. 
Let us suppose that this is false: let $x\in C(p)$ and $x\notin C(\Lambda)$. 
\new{
As $C(\Lambda)$ is convex, the Separating Hyperplane Theorem implies the existence of a hyperplane
separating $x$ form $C(\Lambda)$, that is} 
\old{As the support function of $C(\Lambda)$ at $v\in\R^d$ is the max of the 
$\langle v,\eta\rangle $ for all $\eta\in\Lambda$,  $x\notin C(\Lambda)$ means that }there exists a vector $v$ such
that $\langle v, \eta\rangle<a$
 and $\langle v, x\rangle >a$.
Take the affine function $l(y)=\langle v,y\rangle- a$. 
Note that $l(x)>0$, and that  $g$ is strictly positive on $l^{-1}[0,+\infty)\cap\partial \ball$, 
 so   there is $k>0$ such that $g\geq k$ whenever $l>0$. It follows that there is $\epsilon>0$ such that
 $\epsilon l<g$, so  $h\geq \epsilon l$ on $\ball$, but evaluating the inequality at $p$ we get a  contradiction.

(\ref{car4}$\Rightarrow$\ref{car3})
is clear.
 
(\ref{car3}$\Rightarrow$\ref{car2}) By Proposition \ref{prop: area zero not c1} that if  $\partial_{\mathrm{reg}}\fconv$  is empty
then $\A(\fconv)=0$. \\
Here have that (\ref{car1}$\Leftrightarrow$\ref{car2}$\Leftrightarrow$\ref{car3}$\Leftrightarrow$\ref{car4}).

(\ref{car7}$\Leftrightarrow$\ref{car6}) It is proved in  \cite{Bar10}.

(\ref{car6}$\Rightarrow$\ref{car5}) If $H$ is a $g$-convex set then $H$ is contained in $\overline{I^+(P_g(\ell))}$ for every $\ell\in\partial \ball$
so $H$ is contained in $\fconv$.

(\ref{car5}$\Rightarrow$\ref{car1}) Let $h'$ be the convex envelope of $g$. Then $h\leq h'$ so $\fconv\subset \fconv(h')$. On the other hand
by hypothesis $\fconv(h')\subset \fconv$ so $\fconv=\fconv(h')$, that is $h=h'$.

(\ref{car1}$\Rightarrow$\ref{car6}) 
Let $\fconv_0=\bigcap_{\ell\in\partial \ball} \overline{I^+(P_g(\ell))}$. Notice that if $p\in \fconv_0$ then the affine function 
on $\ball$, $l_p(x)=\langle \hat x, p\rangle_-$ satisfies $l_p\leq g$ on $\partial \ball$. Conversely if
$l(x)=\langle x,y\rangle+c$ is an affine function on $\ball$ with $f\leq g$ then the point $p_l=\binom{y}{-c}$ lies on $\fconv_0$.
That is there is a $1$-to-$1$ correspondence between points of $\fconv_0$ and affine functions $l$ with  $l\leq g$.
Now, as the support function of $\fconv_0$ is defined as $h_0(x)=\sup_{p\in \fconv_0}\langle x,p\rangle$, the correspondence
above shows that $h_0$ is the convex envelope of $g$.

\end{proof}

The unique $g$-convex set satisfying the properties above will be denoted by $\Omega_g$\index{Convex set! $\Omega_g$}.
Its support function on $\ball$ will be denoted by $h_g$.

\begin{remark}{\rm
If $g$ is the restriction to $\partial \ball$ of an affine map (in particular if $g$ is constant), then $\Omega_g$
is the future cone of a point.

Actually, any $g\in C^0(\partial \ball)$  has a maximal value $m$. 
 Then $\Omega_{g}$ is in the future side of light-like hyperplanes $\langle \cdot,\hat{\ell}\rangle_-=m$
 $\forall \ell\in\partial \ball$. They meet at the point $(0,\ldots,0,-m)$: $\Omega_g$ is always contained in the future cone of a point.
}
\end{remark}

\begin{remark}{\rm

 With the terminology of \cite{Bon05}, from Property \ref{car6},
 $\Omega_g$ is a \emph{regular domain}. Moreover,  the space-like boundary $\partial_s\Omega_g$ corresponds
to the \emph{initial singularity}. 
Given an element $p\in\partial_s\Omega_g$ then $\gauss(p)$ is a ideal convex subset of $\H^d$.
Since on $\gauss(p)$ the support function coincides with the scalar product by $p$, the restriction of the support
function to the radial projection of $\gauss(p)$ to $\ball$ is an affine function.
So the dual stratification pointed out in \cite{Bon05} gives a decomposition of $\H^d$ into ideal convex regions
where the support function is affine. 

}
\end{remark}

For $x\in \ball$, we will denote by $F_g(x)$ the maximal convex set of $\ball$ containing $x$ on which $h_g$ is affine. In other words,
$F_g(x)=\rad^{-1}(\gauss(\gauss^{-1}(\rad (x)))$.

\begin{definition}\label{def:simple}
We say that $\Omega_g$  is \emph{ simple}\index{Convex set! Simple} if 
for any $x\in\partial_s\Omega_g$, 
 $F_g(x)$ has codimension $k$, with $k\leq d/2$, $k\in \N$.
\end{definition}
Sometimes for brevity we will say that the support function $h_g$ is simple to mean that $\Omega_g$ is simple.

Note that when $d=2$, $\Omega_g$ is always simple. In higher dimensions $d>2$, $\Omega_g$ is not necessarily simple. However those simple domains
are of some interest in the study of MGHCF space-times (see Section \ref{sub: pogo}). 

\begin{lemma}\label{lem: cuahcy int}
 If $\fconv$ is a $g$-convex set contained in the interior of $\Omega_g$, then $\partial_s\fconv=\partial \fconv$. In particular, $\partial_s\fconv$ is an entire space-like hypersurface. 
   \end{lemma}
   \begin{proof}
Suppose that the statement is false.
Then there would be a point $p\in\partial \fconv$ such that
$\langle p, \hat{x}\rangle_-<h(x)$ for every $x\in \ball$.
Since $p$ is in the boundary of $\fconv$ and 
\[
   \fconv=\{q\in \R^{d+1}|\langle q, \hat{x}\rangle_-\leq h(x)\, \forall x\in\overline{\ball}\}~,
\]
 there is $\ell\in\partial \ball$ such that $\langle p, \hat{\ell}\rangle_-=h(\ell)$.
 
 On the other hand, on $\partial \ball$ the function $h$ coincides with $h_g$, so if the  point $p $
satisfies $\langle p, \hat{\ell}\rangle_-=h(\ell)=h_g(\ell)$, it cannot lie in the interior of $\Omega_g$,
contradicting the assumption on $\fconv$.
\end{proof}

We will also need the following definition, adapted from a base concept of general relativity. 

 \begin{definition}\label{def:cauchy}
 A   \emph{Cauchy} $g$-convex\index{Convex set! Cauchy $g$-convex} set is a $g$-convex set 
 with support function satisfying $h<h_g$.
 \end{definition}

Clearly, a $g$-convex set contained in the interior of $\Omega_g$ is Cauchy. The converse does not hold, as  the 
convex set with support function $h(x)=\frac{1}{2}\|x\|^2$ shows.  In this case $g=h|_{\partial \ball}$ 
is constant equal to $1/2$ so $h_g=1/2$ and $h<1/2$ on $\ball$. On the other hand $\grad h(\ball)=\ball\not= \R^d$. Note that 
$h$ is the only convex function equal to its conjugate $u$ \cite[p.~106]{Roc97}. This example also shows that for a $g$-convex set 
$\fconv$, $\partial_s \fconv$  can be a non-complete $C^1$ hypersurface.

\new{
In this case $\Omega_g$ is the future of the point $(0,\ldots,0,-1/2)$, whereas 
the boundary of the domain $\fconv$ corresponding to $h$ is the graph of $\tilde{u}$, where $\tilde{u}(x)=u(x)=h(x)$  if $x\in\ball$ and 
$\tilde{u}(x)=\|x\|-\frac{1}{2}$ otherwise.
 }
 
The $g$-convex sets that we will consider in the next section will have the property that Cauchy implies contained in the interior of $\Omega_g$.

 \subsection{Some classical results of Monge--Amp\`ere equation}\label{sub: MA regularity}

We quote some results on  Monge--Amp\`ere equation with prescribed continuous data on the boundary of 
the ball that will be used in the next section.
The first one is about uniqueness.

\begin{theorem}[{Comparison Principle, Rauch--Taylor, \cite{RT77}, \cite[1.4.6]{Gu01}}]\label{thm: uniq MA}
Let $O$ be a bounded open convex set of $\R^d$. 
If $u,v \in C(\overline{O})$ are convex functions
such that for any
Borel set $E\subset O$, $\M (u) (E) \leq \M (v) (E)$, then 
$$\operatorname{min}\{u(x)-v(x)|x\in\overline{O}\}=\operatorname{min}\{u(x)-v(x)|x\in\partial O\}~.$$
\end{theorem}

\begin{corollary}
 A $g$-convex set is uniquely determined by its area measure among $g$-convex sets.
\end{corollary}

Now we will quote some regularity results. We will need the following assumption: for a bounded open convex set $O$,
$h$ convex on $O$ and for any Borel set $E$ of $O$,
\begin{equation}\label{eq:hyp mesure bornee}
 m\L(E) \leq \M (h)(E) \le M\L(E), \quad \text{for some fixed } M>m>0~.
\end{equation}

\new{
\begin{remark}{\rm
The condition $m\L(E)\leq \M (h)(E)$ for any Borel set $E$ is equivalent to require that
$\det\hess h(x)\geq m$ $\L$-almost everywhere. Indeed as the determinant of the Hessian 
is the density of the regular part of $\M(h)$ with respect to the Lebesgue measure, if it is bounded from below by $m$, 
then $m\L(E)\leq\M(h)(E)$ by definition of density. Conversely suppose that $m\L(E)\leq\M(h)(E)$ for any Borel set and let
$E_0$ be the support of the singular part of $\M(h)$ with respect to $\L$. Consider $E_1=\{x| \det\hess h(x)<m\}$.
Notice that $\M(h)(E_1\setminus E_0)\leq m \L(E_1\setminus E_0)$, but by the assumption the opposite inequality holds.
As a consequence $\L(E_1\setminus E_0)=0$. As $\L(E_0)=0$, we conclude that $\det \hess h(x)\geq m$ $\L$-almost everywhere.

On the other hand the condition $\M(h)(E)\leq M\L(E)$ is equivalent to require that $\M(h)$ is absolutely continuous with respect to $\L$
with density bounded from above by $M$. 
}\end{remark}
}

\begin{theorem}[{Caffarelli, \cite[Theorem~5.2.1]{Gu01}}]\label{thm:strict cvxe}
Let $O$ be a bounded open convex set of $\R^d$. 
 Let $h$ be a convex function in $O$ satisfying \eqref{eq:hyp mesure bornee}.
 Assume that $h\geq 0$ on $O$ and let $C=\{x\in O|h(x)=0\}$. If 
 $C$ is non-empty and contains more than one point, then $C$ has no extremal point
 in $O$.
\end{theorem}

An extremal point of a convex set $C$ is a point of $\partial C$ which is not a convex combination of other
points of the closure $\overline{C}$ of $C$. A classical result of convex geometry states that if $\overline{C}$ is compact then 
 $\overline{C}$ is the convex hull of the extremal points of $C$.

\begin{theorem}[Multidimensional Alexandrov--Heinz Theorem]\label{thm: alex heinz}
Let $h$ be a convex function on the $d$ dimensional unit ball $\ball^d$ in $\mathbb R^d$ such that the Monge--Amp\`ere
measure of $h$ satisfies $\M(h)\geq c_0\mathcal{L}$ for a positive $c_0$.

Let $W$ be an affine subspace of $\R^d$ passing through $0$, of dimension $d-k$, for $k\leq d/2$.
Suppose that $h\equiv 0$ on the boundary of a $(d-k)$ ball $\ball^{d-k}=W\cap \ball^d$.
Then 
\[
          h(0)<0~.
 \]
 \end{theorem}
 
The theorem says that, if the Monge--Amp\`ere measure of $h$ is bounded from below by a positive constant
(times the Lebesgue measure), then $h$ cannot be affine on a $(d-k)$-plane. In particular, $h$ cannot be affine
on a hyperplane. If $d=2$,  this says that  if the Monge--Amp\`ere measure of $h$ is bounded from below by a positive constant, then
$h$ is strictly convex. This is the usual formulation of  Alexandrov--Heinz Theorem, see \cite[Remark~3.2]{TW08}.
For completeness,
we give a proof of Theorem~\ref{thm: alex heinz} at the end of this section.  
  The proof  closely follows the proof in the $2$-d case given in \cite{TW08}.

\begin{remark}\label{rk:ex pogo}
{\rm
The assumption on the dimension on $k$ in the Theorem is optimal.
In fact a famous example due to Pogorelov  \cite[pp.~81--86]{Pog78} 
shows that for $d\geq 3$ there exists a convex function $f$ on $\ball^d$ with strictly positive
 Monge--Amp\`ere measure (in the sense above) which is constant on a segment.
 
 Actually the example can be easily generalized: let $k>d/2$ and for any vector $x\in\R^d$ let
 $r(x)=\sqrt{x_{d-k+1}^2+\cdots+x_{d}^2}$ and $t(x)=\sqrt{x_{1}^2+\cdots+x_{d-k}^2}$. Consider  the function
 $$f(x)=\beta r(x)^\alpha(1+\beta t(x)^2)$$ with $\alpha=2k/d$ and $\beta>1.$
 This function is $0$ on the ball $\ball^{d-k}$ of dimension $d-k$ defined by the equation $x_{d-k+1}=\cdots=x_d=0$. 
 Moreover $f$ is smooth on $\ball^d\setminus \ball^{d-k}$ and $C^1$ on $\ball^d$ (here the assumption on $k$ is needed).
 A direct computation shows that for a suitable choice of $\beta$ the function $f$ satisfies
 $\det \Hess f\geq c_0$ on $\ball^d\setminus \ball^{d-k}$ and some positive constant $c_0$.
 From this estimates \newnew{and \eqref{eq:ma c3}}\oldold{, as by a general fact $\M(f)\geq \det\Hess f\L$  \cite[Lemma~2.3]{TW08},} one deduces that
 $\M(f)$ is bigger than $c_0\mathcal L$.
 
 As we will not use this result in the paper we omit the computation leaving it to the reader (basically it can be done
 following the same line as in \cite[pp.~81--86]{Pog78}).
  
}
\end{remark}

\begin{theorem}[{\cite[Theorem~3.1]{TW08}}]\label{thm: regul}  Let $h$ be a strictly convex solution of 
$$\M(h)=\varphi \mathcal{L} $$
on a ball $O$.
Suppose $\varphi > 0$ and
$\varphi\in C^{1,1}(O)$. Then $h \in C^{3,\alpha} (O)$ for any $\alpha \in (0, 1)$. 
If furthermore $\varphi\in C^{k,\alpha}(O)$ for
some $k \geq 2$ and $\alpha \in (0, 1)$, then $h \in C^{k+2,\alpha}(O)$.
\end{theorem}

\begin{corollary}\label{cor: MA regularite}
 Let $\fconv$ be a $g$-convex set with support function $h$ and such that $\A(\fconv)= f(\rad^{-1})_*(\mathrm{d}\H^d)$ 
 with $f\in C^{k+1}(\H^d)$, $k\geq 2$ and   $c_2\geq f \geq c_1>0$.
 
  If moreover
  \begin{enumerate}[nolistsep]
   \item $h_g$ is simple, or   \label{condition 1}
   \item $h<h_g$ on $\ball$, \label{condition 2}
  \end{enumerate}
then $\partial_s \fconv$ is a strictly convex $C^{k+2}$ hypersurface. 
\end{corollary}
\begin{proof}
First notice that  \textit{\ref{condition 1}.} implies  \textit{\ref{condition 2}.} 

Indeed suppose that $h=h_g$ at some point $x$. By Proposition~\ref{prop:a=0},
$F:=F_g(x)$ is a convex set with vertices
on the boundary and $h_g$ is affine on $F$. Thus since $h$ coincides with $h_g$ on
$\partial \ball\cap F$ and at some interior point we get that those functions
coincide on the whole $F$. 
Up to compose by an affine transformation we may assume that $h=h_g$ is zero on $F$.

\old{If the dimension of $F$ is $d$,  since $\A(\fconv)$ is zero on $F$  and the interior of $F$ is non empty, this 
contradicts  the assumption on $\A(\fconv)$. Otherwise, } As $h_g$ is simple
 Theorem~\ref{thm: alex heinz} implies that $h$ is strictly negative on $F$, giving a contradiction,
 
 So  we suppose that \textit{\ref{condition 2}.} is fulfilled.

Let us first check that $h$ is strictly convex.
By contradiction suppose that it is not strictly convex. 
\new{
There exists an affine function $l$ such that
its graph is a support plane for the graph of $h$, and the convex region $C=\{x\in\ball |h(x)=l(x)\}$
contains more than two points.

Notice that the convex function $h'(x)=h(x)-l(x)$ is non negative and 
$$C=\{x\in \ball|h'(x)=0\}\,. $$ Moreover  $\M(h)=\M(h')$ by \eqref{eq:MA affine}.
}
\old{There exists a 
convex region $C$ 
on which $h(x)=l(x)$, with $l(x)$ an affine map.
Let us write
$$C=\{x\in \ball|h'(x)=0\} $$
where $h'=h-l$ is a convex function on $\ball$. Note that $\M(h)=\M(h')$ by \eqref{eq:MA affine}.}

Let $\ball_{\epsilon}$ be the open ball centered at the origin of radius $0<\epsilon<1$,
and let $C_{\epsilon}=C\cap \ball_{\epsilon}$. 
Choose $\epsilon$ so that $C_{\epsilon}$ is not empty.
On $\ball_{\epsilon}$, $\lambda$ is bounded from below by a positive constant depending on $\epsilon$,
hence by the previous paragraph, $\M(h')$ satisfies \eqref{eq:hyp mesure bornee}.
So Theorem~\ref{thm:strict cvxe} applies, and the extremal points of $C_{\epsilon}$ are on the boundary of 
$\ball_{\epsilon}$. This is true for any $\epsilon'>\epsilon$, so the extremal points of
$C$ are on $\partial \ball$. Clearly, $C$ has at least two extremal points $x_0,x_1$ in $\partial \ball$,
and $h=l$ on the segment joining $x_0$ to $x_1$. As $h=h_g$ on $\partial \ball$, by definition
of $h_g$, $h=h_g$ on the segment,  that contradicts the hypothesis $h<h_g$ on $\ball$. Hence $h$ is strictly convex.

 For any $0<\epsilon<1$, $f \lambda^{-d-2}$ has the same regularity than $f$ on the ball of radius $\epsilon$,
 so from Theorem~\ref{thm: regul}, $h\in C^{k+2}$ on this ball, hence $h\in C^{k+2}(\ball)$.
The determinant of the Hessian of $h$ is equal to $f \lambda^{-d-2}>0$, so 
the principal radii or curvature
 of $\fconv$ are positive. $\fconv$ is then a $C^2_+$ F-convex set, and in this case
 $\partial_s \fconv$ has the same regularity than $h$ \cite{fv}.
\end{proof}

\begin{remark}{\rm
 Another result of Caffarelli says that if $h$ is the support function of a $g$-convex set,
  \eqref{eq:hyp mesure bornee} is satisfied, and $h(x)<h_g(x)$ at $x\in \ball$, then $h$ is $C^{1,\alpha}$ at $x$, for $0<\alpha<1$,
  see  \cite[5.4.6]{Gu01}. 
 
 For the  particularities of the  $d=2$ case, see also
   \cite{NS85}, \cite{Pog73} for different viewpoints. 
      
   The $g$ affine case has also particularities, see \cite{Gu01,TW08} or \cite[Theorem 7]{CY77}.
}\end{remark}

%We move now to the proof of Theorem  \ref{thm: alex heinz}.
%We first prove a smooth version, and deduce the general case by an approximation argument.}
%\old{
Finally we give a proof of Theorem \ref{thm: alex heinz}.
\begin{proof}[Proof of Theorem~\ref{thm: alex heinz}]

\old{
The theorem will follow from the following result.

\emph{Fact: Let $h$ be a $C^2$ function on $\ball^d$ such that there is $c_0>0$ with
$\det \Hess h \geq c_0 \operatorname{Id}$. Then there is a constant $c>0$ depending on $c_0$, on the Lipschitz
constant of $h$ on $\frac{3}{4}\ball^d$, on $k$, and on $d$, such that,
$$        \sup_{\partial \ball^{d-k}} h- h(0)>c~.$$}
}
%\new{
%\begin{proposition}
%Let $h$ be a $C^2$ function on $\ball^d$ such that there is $c_0>0$ with
%$\det \Hess h \geq c_0 \operatorname{Id}$. Then there is a constant $c>0$ depending on $c_0$, on the Lipschitz
%constant of $h$ on $\frac{3}{4}\ball^d$, on $k$, and on $d$, such that,
%$$        \sup_{\partial \ball^{d-k}} h- h(0)>c~.$$
%\end{proposition}
%}
%\begin{proof}
Without loss of generality, we set  $\ball^{d-k}=\{x\in \ball^d| x_{d-k+1}=\cdots= x_{d}=0\}$.
For $x\in \ball^d$, let us denote $\tilde{x}=(x_1,\ldots,x_{d-k},0,\ldots,0)$ and $\bar{x}=(0,\ldots,0,x_{d-k+1},\ldots,x_d)$, so 
that if $x\in \ball^{d-k}$, then $x=\tilde{x}$.

\old{Up to a translation, let us assume that  $ \sup_{\partial \ball^{d-k}} h=0$.}
\new{Since $h\equiv 0$ on $\partial\ball^{d-k}$},  by convexity, $h\leq 0$   on $\ball^{d-k}$. Using this and again convexity, we have actually
on $\ball^{d-k}$:
\begin{equation}\label{eq:premieres ineq conv}
 0\geq h\left(\tilde{x}\right)\geq -2 |h(0)|~.
\end{equation}

Let $$ \mathbb D=\frac{1}{4}\ball^{d-k}\times\frac{1}{4}\ball^{k}=
 \left\{x \left.\right| \|\tilde{x}\|< \frac{1}{4}, \|\bar{x}\|< \frac{1}{4}\right\}~.$$

First we will prove that \new{for almost} every $x\in \mathbb D$ 
the  derivatives of $h$ along the  first $d-k$ directions are estimated as follows:
\begin{equation}\label{eq:hgrad}
    \left|\partial_i h(x)\right|<8\left(|h(0)|+C\|\bar{x} \|\right)\quad\text{for }i\leq d-k~,
 \end{equation}
where $C$ is the Lipschitz constant of the restriction of $h$ to $\frac{3}{4}\ball^d$.
For sake of simplicity we prove the formula for $i=1$.

As $h$ is $C$-Lipschitz on $\frac{3}{4}\ball^d$, using \eqref{eq:premieres ineq conv}, if $x\in\frac{1}{2}\ball^d$ we have

\[
\begin{array}{l}
      h(x)\geq h(\tilde{x})-C\|\bar{x}\|\geq -2|h(0)|-C\|\bar{x}\|~,\\
      h(x)\leq  h(\tilde{x})+C\|\bar{x}\|\leq C\|\bar{x}\|~.
\end{array}             
\]
\old{
As for a differentiable convex function $f:\R\rightarrow\R$ we have
$f'(x)(y-x)\leq f(y)-f(x) $, for $x\in D\subset \frac{1}{2}\ball^d$,
 we have that 
}
\new{
If $x\in\mathbb D\subset\frac{1}{2}\ball$, then $(1/2, x_2,,\ldots, x_d)\in\frac{3}{4}\ball$.
so we have for almost all $x$}
\[
    \partial_1 h (x) \leq \left(\frac{1}{2}-x_1\right)^{-1} \left[h\left(\frac{1}{2},x_2,\ldots, x_d\right)
    - h\left(x\right)\right] \]
$$    \leq 4 (2|h(0)|+2C\|\bar{x}\|)\leq 8(|h(0)|+C\|\bar{x}\|)~, $$   
\new{where we have applied the general estimate $f'(a)(b-a)\leq f(b)-f(a) $ to the convex function $f(t)=h(t,x_2,\ldots,x_d)$.}

Analogously
\[
    \partial_1 h(x)\geq \left(x_1+\frac{1}{2}\right)^{-1}\left[h\left(x\right)-
    h\left(-\frac{1}{2}, x_2,\ldots, x_d\right)\right]\geq
     -8(|h(0)|+C\|\bar{x}\|)~,
\]
and \eqref{eq:hgrad} follows.

Let us denote
$D_{[d-k,d-k]}(x)$ (resp. $D_{[k,k]}(x)$ ) the determinant of the $(d-k)\times(d-k)$ (resp. $k\times k$) 
submatrix of the Hessian of $h$ at $x$  made by  first $d-k$ lines and columns 
(resp.  the  last  $k$ lines and columns). 
\new{
Notice that$D_{[d-k,d-k]}$ and $D_{[k,k]}$ are function defined almost everywhere on $\ball$.}
As the Hessian is positive definite, by Fischer inequality \cite[7.8.3]{HJ90},

$$\det \hess h \leq  D_{[d-k,d-k]} D_{[k,k]}~.$$

Now by \eqref{eq:hgrad} for any $\bar{x}\in\frac{1}{4}\ball^{k}$ we have that the function 
$$\tilde{x}\mapsto \widetilde{\grad h}(\tilde x+\bar x) =(\partial_1 h,\ldots, \partial_{d-k}h, 0, \ldots,0)(\tilde x+\bar x)$$
sends \new{almost all points} in  $\frac{1}{4}\ball^{d-k}$ into a $(d-k)$-dimensional disk of radius $8(d-k)(|h(0)|+C\|\bar{x}\|)$.
In particular with \new{\eqref{eq:ma c3}}:
\[
   \int_{\frac{1}{4} \ball^{d-k}}  D_{[d-k,d-k]}(x) \d \tilde{x} \leq C(d,k)(|h(0)|+C\|\bar{x}\|)^{d-k}~.
\]
Moreover as  $D_{[d-k,d-k]} D_{[k,k]}\geq c_0$, 
for \new{almost} any $\bar{x}\in \frac{1}{4} \ball^k$,
\[
\int_{\frac{1}{4} \ball^{d-k}} \frac{1}{D_{[k,k]}(x)} \d\tilde{x} \leq c_0^{-1} \int_{\frac{1}{4} \ball^{d-k}} D_{[d-k,d-k]}(x)\d \tilde{x}
\leq C(d,k)c_0^{-1}(|h(0)|+C\|\bar{x}\|)^{d-k}~.
\]
Cauchy--Schwarz inequality gives
\[
\int_{\frac{1}{4} \ball^{d-k}}  D_{[k,k]}(x)\d\tilde{x} \geq \frac{ \operatorname{Vol}(\frac{1}{4}\ball^{d-k})^2}{\int_{\frac{1}{4} \ball^{d-k} }
\frac{1}{D_{[k,k]}(x)} \d\tilde{x} }
\geq C'(d,k) c_0(|h(0)|+C\|\bar{x}\|)^{-d+k}~.
\]

Now as $h$ is $C$-Lipschitz on $\frac{3}{4}\ball^d$, for any $\tilde{x}\in\frac{3}{4}\ball^{d-k}$ the function 
$$\bar{x}\mapsto\overline{\grad h}(\tilde x+\bar x)= (0,\ldots,0,\partial_{d-k+1}h,\ldots, \partial_d h)(\tilde x+\bar x)$$ 
sends \new{almost all points} in $\frac{1}{4}\ball_+^k$ ---the part of $\frac{1}{4}\ball^k$ with vectors of positive entries--- into a 
$k$-dimensional disc of radius $C$, 
and again by \eqref{eq:ma c3}, $\int_{\frac{1}{4} \ball_+^{k}} D_{[k,k]}(x)\d\bar{x}$ is less than the volume 
of this disc, so
\begin{equation}\label{eq:bound above AH}\int_{\frac{1}{4}\ball^{d-k}} \int_{\frac{1}{4} \ball_+^{k}} D_{[k,k]}(x)\d\bar{x} \d\tilde{x}\leq C''(d,k) C^k~.\end{equation}

Putting the inequalities together, by Fubini,

$$\int_{\frac{1}{4}\ball_+^k}(|h(0)|+C\|\bar{x}\|)^{-d+k} \d \bar{x} \leq C'''(d,k)C^k c_0^{-1}~. $$

Now integrating in polar coordinates on $\ball_+^k$ we get that
\[
    \int_0^{1/4} \frac{t^{k-1}}{(|h(0)|+Ct)^{d-k}}\d t\leq C''''(d,k)C^k c_0^{-1}~.
\]
The function $F(y)=\int_0^{1/4}\frac{t^{k-1}}{(y+Ct)^{d-k}}\d t$ is decreasing and by Beppo-Levi Theorem, as $k\leq d/2$,
$\lim_{y\to 0^+} F(y)=+\infty$.

This shows that $|h(0)|>c$ where $c$ depends only on $C$, $c_0$, $k$, and $d$.

\end{proof}
%\qd
\old{
\begin{proof}[Proof of Theorem~\ref{thm: alex heinz}]
Now let $h$ satisfying the hypothesis of the theorem. Let $h_j$ be a $C^2$ convex function on $\ball^d$ such that
$|h_j(x)-h(x)| <1/j$ on $\ball(1-1/j)$,
with $j$ such that $1-1/j>3/4$.

 As 
 $\M(h_j)$ weakly converges to $\M(h)$, one can see that there is positive constant, that we still denote by $c_0$,  
 such that for $j$ sufficiently large,
 $\M(h_j)\geq c_0\mathcal{L}$. Moreover, on $\frac{3}{4}\ball^d$,  $h_j$ is $C$-Lipschitz for the same constant
 $C$ than $h$ \cite[10.6]{Roc97}. Let us introduce the following $C^ 2$ convex function on $\ball$: 
 $h'_j(x)=h_j(r_jx)$ with $r_j=1-1/j$.  Then $h'_j$ is $C/2$ Lipschitz on $\frac{3}{4}\ball$, 
 $\det\hess h'_j \geq r^{2d}c_1'$ and $h'_j(0)=h_j(0)$. 
 As $\sup_{\partial \ball}h'_j-\sup_{\partial \ball}h_j$ is close to $0$ as $j\to+\infty$, the 
 fact then gives that for $j$ big
 $$\sup_{\partial \ball^{d-k}(1-\frac{1}{j})} h_j-h_j(0)>c~, $$
 for some $c$ independent of $j$.
 So 
 $$\sup_{\partial \ball^{d-k}(1-\frac{1}{j})} h-h(0)+\frac{2}{j}>c~. $$
  
 Now $h$ is convex and is $0$ on  $\partial \ball^{d-k}$, so it is non-positive on $ \ball^{d-k}$, so
 
  $$-h(0)+\frac{2}{j}>c~. $$            
  
Letting $j$ go to $\infty$ we get $-h(0)\geq c$.

\end{proof}
}

\section{Equivariant convex sets}\label{sec3}

\subsection{Definition}\label{sub def tau}

We fix a uniform lattice $\Gamma$ in $\operatorname{SO}^+(d,1)$.
Recall that an affine deformation of $\Gamma$ is a subgroup $\Gamma_\tau$ of $\operatorname{Isom}(\mink)$ obtained
by adding a translation part to elements of $\Gamma$. Thus an affine deformation is defined
by the translation function $\tau:\Gamma\to\R^{d+1}$. 
Given $\gamma\in\Gamma$ we denote by $\gamma_\tau$ the corresponding element in $\Gamma_\tau$, i.e. the
affine transformation $\gamma_\tau(x)=\gamma(x)+\tau_\gamma$.

The condition that $\Gamma_\tau$ is a subgroup forces
$\tau$ to verify the cocycle relation
\[
     \tau_{\alpha\beta}=\tau_\alpha+\alpha\tau_\beta~.
\]
The set of cocycles is a vector space of finite dimension denote here by $Z^1(\Gamma, \R^{d+1})$.
\index{Cocycle}
The subgroup of $\operatorname{Isom}(\mink)$ obtained by conjugating $\Gamma$ by a translation in $\R^{d+1}$ of a vector
$t_0$ is an affine deformation $\Gamma_\tau$ of $\Gamma$.\index{$\Gamma_\tau$} We say that such deformation is trivial.
The corresponding cocycle is $\tau_\alpha=\alpha(t_0)-t_0$. 
A cocycle obtained in this way is called a coboundary. The set of coboundary is a subspace of
$Z^1(\Gamma, \R^{d+1})$ of dimension $d+1$. The quotient 
$H^1(\Gamma, \R^{d+1})=Z^1(\Gamma, \R^{d+1})/B^1(\Gamma, \R^{d+1})$\index{$H^1(\Gamma, \R^{d+1})$} is a geometrically meaningful object under
different points of view. 

If $\Gamma$ is torsion free, 
it parameterizes the MGHCF structures on the manifold $\R\times(\H^d/\Gamma)$, but, if $d>2$,
 it is also the Zarinski 
 tangent space at the identity  of the character variety of representations of $\Gamma$ into $\operatorname{SO}(d+1,1)$
 (see \cite{SCAN00}).

Recall from the introduction  that a $\tau$-convex set is a closed future convex set invariant under the action of $\Gamma_{\tau}$. \index{Convex set! $\tau$-convex set}

If $\Gamma$ is not torsion-free, then by Selberg Lemma it contains a torsion free normal subgroup $\Gamma'$ of finite index. 
By \cite{Bon05, Bar05}, there exists a unique maximal
convex subset invariant under $\Gamma'_\tau$. By uniqueness it is also $\Gamma_\tau$-invariant.
That is, $\tau$-convex sets always exist.

\begin{remark}{\rm
 The argument to show the existence of $\tau$-convex set for general uniform lattices is quite general. 
 It turns out that  if  Theorems \ref{thm:main1} and \ref{thm:main2}  hold for  $\Gamma'$, then they also hold for $\Gamma$.
 So there is no loss of generality to assume that $\Gamma$ is torsion-free.

 In the sequel we will always make such assumption, leaving to the reader the details in the general case.
}
\end{remark}
 
First we want to show that the Gauss map of  a $\tau$-convex set is surjective.

\begin{lemma}
 A $\tau$-convex set $\fconv$ is an F-convex set.
\end{lemma}
 \begin{proof}
\old{
 By the invariance of $\fconv$, if $\eta$ is a support vector for $\fconv$ at $p$, then $\gamma(\eta)$ is a support vector
 for $\fconv$ at $\gamma_\tau(p)$.
 Thus the image of the Gauss map is $\Gamma$-invariant.
 
On the other hand,  it is easy to see that if $\eta_1,\eta_2\in\H^d$ are support vectors of  $\fconv$,
  then all the points on the geodesic segment   of $\H^d$ between $\eta_1$ and $\eta_2$ are support vectors. 
That is, the image of the Gauss map is convex.
 
 Since $\Gamma$ is a uniform lattice, the only $\Gamma$-invariant non-empty convex subset is $\H^d$,
 so the Gauss map is surjective.
}
\new{
We need to prove that the support function $\tilde{H}:\fut(0)\to\R\cup\{+\infty\}$ defined by
\[
    \tilde{H}(X)=\sup_{p\in\fconv}\langle X, p\rangle_-
\]
is finite over $\fut(0)$.

First we notice that $\tilde{H}$ is convex and $1$-homogenous.  Thus, the region where $\tilde{H}$ is finite is  a convex cone
with vertex at $0$:  its intersection with $\H^d$ is a hyperbolic convex set.
If we prove that this region contains an orbit of $\Gamma$ in $\H^d$, we conclude that $\tilde{H}$ is finite over $\fut(0)$.
Indeed as $\Gamma$ is a uniform lattice, the convex hull in $\H^d$ of the orbit of any point is the whole $\H^d$.

In order to prove that $\tilde{H}$ is finite on some orbit of $\Gamma$, we remark that $\tilde{H}$ is finite on the image of the Gauss map.
Indeed, if $\eta=\gauss(p)$, for some $p\in \partial_s\fconv$, $\tilde{H}(p)=\langle \eta, p\rangle_-$. 

 By the invariance of $\fconv$, if $\eta$ is a support vector for $\fconv$ at $p$, then $\gamma(\eta)$ is a support vector
 for $\fconv$ at $\gamma_\tau(p)$.
 Thus the image of the Gauss map is $\Gamma$-invariant, i.e. it is union of some orbits of $\Gamma$.
 
To conclude it is sufficient to observe that $\partial_s\fconv$ is not empty, that is  $\fconv$ is not a half-space bounded by any lightlike plane $P$.
Otherwise the lightlike direction orthogonal to the tangent space of $P$ would be fixed by $\Gamma$, but we know that $\Gamma$ does not fix
any point in $\partial\H^d$.}
 \end{proof}

 As stated in the introduction, there is a natural action of $\operatorname{Isom}(\mink)$ on the set of the support functions
 of F-convex sets. Basically $\sigma\cdot H$  is by definition the support function of $\sigma(\fconv)$.
 It turns out that $\fconv$ is a $\tau$-convex set if and only if its support function $H$ is fixed by the action of $\Gamma_\tau$.
 
 Another simple remark is that the action of $\operatorname{Isom}(\mink)$ on the set of support functions preserves the natural
 partial 
 order. Indeed since we have $H\leq H'$ if and only if the corresponding F-convex set satisfy $\fconv\subset \fconv'$, it turns out that 
 \begin{equation}\label{eq:ordpres}
      H\leq H' \Rightarrow \sigma\cdot H\leq \sigma\cdot H'~.
 \end{equation}
 
 The function  $\sigma\cdot H$  can be explicitly described in terms of $\sigma$ and $H$.
 \begin{lemma}
 If $H$ is a support function of an F-convex set, and $\sigma\in \operatorname{Isom}(\mink)$ has linear part
 $\gamma$ and translational part $\tau$, then 
  \begin{equation}\label{eq:act}
   (\sigma\cdot H)(X)=H(\gamma^{-1}X)+\langle \tau, X\rangle_-~.
 \end{equation}
 \end{lemma}
 \begin{proof}
Let $\fconv$ be the F-convex set corresponding to $H$.
We have 
\[
(\sigma \cdot  H)(X)=\sup_{p\in \sigma \fconv}\langle  X, p\rangle_-=
\sup_{q\in \fconv}\langle X, \sigma q\rangle_-=
\sup_{q\in \fconv}\langle  X, \gamma q\rangle_- +\langle 
 X, \tau_\gamma\rangle_-=
H(\gamma^{-1}(X))+\langle  X, \tau_\gamma\rangle_-~.
\]
\end{proof}

In other words the linear part of $\sigma$ acts on $H$ by pull-back, whereas the translation $\tau$ adds a linear
factor to $H$.

 Notice that if $H$ is a $\tau$-support function, then $\gamma_\tau\cdot H=H$ so we have
 $$H(X)=H(\gamma^{-1}X)+\langle \tau_\gamma, X\rangle_-$$ that is, 
 putting $Y=\gamma^{-1} X$
 \begin{equation}\label{eq:inv}
    H(\gamma Y)=H(Y)+\langle \gamma^{-1}\tau_\gamma, Y\rangle_-=
    H(Y)-\langle \tau_{\gamma^{-1}}, Y\rangle_-~.
\end{equation}    

 Notice that equations \eqref{eq:inv} are linear but, if $\tau\neq 0$, not homogeneous. 
 So the set of functions satisfying \eqref{eq:inv}
 is an affine space whose tangent space is the space of $\Gamma$-invariant functions.
  
  It follows that a convex combination of $\tau$-support function is still a $\tau$-support function.
If $\fconv_1$ and $\fconv_2$ are $\tau$-convex set then for any $0\leq \lambda\leq 1$, $\lambda \fconv_1+(1-\lambda)\fconv_2$
is a $\tau$-convex set. In other words, for those operations, the set of $\tau$-convex sets is convex.

Moreover if $\fconv_1$ is a $\tau$-convex set and $\fconv_2$ is a $\Gamma$-invariant set, then $\fconv_1+\fconv_2$ is 
again a $\tau$-convex set.

As explained in Section \ref{sub:back}, a support function $H$ is determined by its restriction on $\H^d$, denoted here
by $\oh$ and by its restriction on the ball $\ball$, say $h$, where more precisely we embed $\ball$ into $\fut(0)$ by the affine
map $x\mapsto\hat{x}=\binom{x}{1}$ and put $h(x)=H(\hat{x})$.

Notice that since $\H^d$ is left invariant by the action of linear isometries of $\mink$, the action of
$\operatorname{Isom}(\mink)$ on the space of hyperbolic support functions can be easily deduced by \eqref{eq:act}.
In fact for $\eta\in\H^d$ 
we have $(\sigma\cdot \oh)(\eta)=\oh(\gamma^{-1}\eta)+\langle \tau, \eta\rangle_-$, so if $\oh$ is the hyperbolic
support function of a $\tau$-convex set we deduce 
$ \oh(\gamma \zeta)=\oh(\zeta)-\langle \tau_{\gamma^{-1}}, \zeta\rangle_-$, for $\zeta\in\H^d$.

The action of $\operatorname{Isom}(\mink)$ on the ball support functions (that are the convex functions on $\ball$) is more
complicated, due to the fact that the immersion of  $\ball$ into $\fut(0)$ is not invariant by the action of $\operatorname{Isom}(\mink)$.
The translation part contributes by adding an affine function, whereas the linear part of $\sigma$ does not act simply by pull-back.
More precisely,
notice that $\ball$ is naturally identified to the projective set of the lines contained in $\fut(0)$, that is a subset of $\RP^{d+1}$.
As $\operatorname{SO}^+(d,1)$ preserves $\fut(0)$ there is a natural projective action of $\operatorname{SO}^+(d,1)$ on $\ball$.
Given $\gamma\in \operatorname{SO}^+(d,1)$ let us denote by $\bar\gamma$ the corresponding projective transformation of $\ball$ 
\[
   \binom{\bar \gamma(x)}{1}=\frac{1}{(\gamma(\hat{x}))_{d+1}}\gamma (\hat{x})~.
\]
Using \eqref{eq:act} and the $1$-homogeneity of $H$ we deduce
how $\operatorname{Isom}(\mink)$ acts on the ball support functions.

\begin{lemma}\label{lem: projective action}
For  $\sigma\in \operatorname{Isom}(\mink)$ and $h$ a convex function on $\ball$,
\[
\begin{split}
  (\sigma \cdot  h)(x)=  (\gamma^{-1}(\hat{x}))_{d+1}  h (\bar \gamma^{-1}(x))+\langle \hat{x}, \tau_\gamma\rangle_-
  \\ =\frac{\lambda(x)}{\lambda(\bar\gamma^{-1}x)} h (\bar \gamma^{-1}(x))+\langle \hat{x}, \tau_\gamma\rangle_-~. 
\end{split}
\]
\end{lemma}
\begin{proof}
We have $(\sigma\cdot h) (x)=(\sigma\cdot H)(\hat{x})=H(\gamma^{-1}\hat{x})+\langle \hat{x}, \tau_\gamma\rangle_-$.
On the other hand, $$H(\gamma^{-1}\hat{x})=(\gamma^{-1}(\hat{x}))_{d+1}H\left(\binom{\bar\gamma^{-1}x}{1}\right)=
(\gamma^{-1}(\hat{x}))_{d+1}h(\bar\gamma^{-1}x)~.$$ So the first part of the formula is proved.

Remark now that 
\[
  \lambda(x)= \|\hat{x}\|_-=\|\gamma^{-1}\hat{x}\|_-
  =(\gamma^{-1}(\hat{x}))_{d+1}\|\widehat{\bar\gamma^{-1}x}\|_{-}
  = (\gamma^{-1}(\hat{x}))_{d+1}\lambda(\bar\gamma^{-1}x)
\]
and so also the second part of the formula holds.
\end{proof}

Since isometries of Minkowski space preserve the volume, the area measure associated to any F-convex set
is preserved by the action: $\A(\sigma \fconv)=\gamma_*\A(\fconv)$.
In particular if $\fconv$ is a $\tau$-convex set, then $\A(\fconv)$ is a $\Gamma$-invariant measure.

Using the correspondence between $\A(\fconv)$ and the Monge--Amp\`ere measure associated to the ball support function
stated in Proposition \ref{prop: area=ma}, 
we deduce that the action of $\operatorname{Isom}(\mink)$ on the set of convex functions of $\ball$ effects in a very 
controlled way the
Monge--Amp\`ere measure.

\begin{lemma}
If $h$ is a convex function of $\ball$ and $\sigma\in \operatorname{Isom}(\mink)$ then
\[
   \M(\sigma\cdot h)=\frac{\sqrt{1-\|\bar\gamma^{-1}(x)\|^2}}{\sqrt{1-\|x\|^2}}\bar\gamma_*\M(h)~.
\]
\end{lemma}
\begin{proof}
If  $\rad:\ball\to\H^d$ is the radial map, we have
$$\M(\sigma\cdot h)=\lambda^{-1}(\rad^{-1})_*(\A(\sigma \fconv))=
\lambda^{-1}(\rad^{-1})_*(\gamma_*\A(\fconv)).$$

Since $\rad^{-1}\gamma=\bar\gamma\rad^{-1}$ we conclude that
$\M(\sigma\cdot h)=\lambda^{-1}(\bar\gamma_*)(\lambda\M(h))$.
\end{proof}

\begin{corollary}\label{cor:mabound}
For any $\sigma\in \operatorname{Isom}(\mink)$, $c_0\in\R$ and $C$ compact subset of $\ball$, there is a constant 
$c=c(\sigma, c_0, C)$ 
such that if $\M(h)\geq c_0\mathcal L$ on $\ball$ then $\M(\sigma\cdot h)\geq c\mathcal L$ \new{on $C$}.

Moreover $c$ is uniformly bounded if the linear part of $\sigma$ lies in a compact subset of $\operatorname{SO}^+(d,1)$.
\end{corollary}

\subsection{Extension of $\tau$-support functions to the boundary}

\begin{proposition}\label{prop:g tau}
Let $\tau\in Z^1(\Gamma,\R^{d+1})$. Then there exists $g_{\tau}\in C^0(\partial \ball)$ such that
$\tau$-convex sets are $g_{\tau}$-convex sets.
\end{proposition}

See also \cite{Bar05}.

\begin{proof}
As there exists a  $C^2_+$ $\tau$-convex set 
with constant mean radii of curvature \cite{fv}, 
from Proposition~\ref{prop: li},  the corresponding support function on the ball,
say $h_0$, extends to the sphere $\partial \ball$. Denote by $g_\tau$ the extension of $h_0$ to $\partial \ball$.

Let us take any other $\tau$-convex set $\fconv$. We want to prove that its support function on the ball, say $h$,
extends $g_\tau$.
By \eqref{eq:inv} we have that the difference $H-H_0$ is a $\Gamma$-invariant function
on $\fut(0)$. 

So,  its restriction to $\H^d$ reaches a  minimum $a$ and a maximum $b$.
Using that on the ball $h(x)=H(x,1)=\lambda(x)H(\rad(x))$, we deduce that
\[
     a\lambda(x) <h(x)-h_0(x)<b\lambda(x)~
\] 
As $\lambda(x)=\sqrt{1-\|x\|^2}$ goes  uniformly  to $0$ as $\|x\|\rightarrow 1$ we get the result.

So the value of the  extension on the sphere depends only on $ \tau$. 
\end{proof}

\begin{remark}\label{rk:ext inv}
{\rm
The argument of the proof of Proposition \ref{prop:g tau} shows that 
if $\overline{f}$ is any bounded function on $\H^d$, then the restriction of its $1$-homogeneous extension
to the ball, say $f$, is continuous on $\overline{\ball}$ and in fact $f|_{\partial \ball}$ is zero.

This remark can be applied for instance if $\overline{f}$ is a $\Gamma$-invariant function and we will
often use it in the sequel. 
}
\end{remark}

We will denote $\Omega_{g_{\tau}}$ by $\Omega_{\tau}$\index{Convex set! $\Omega_\tau$}. If $\tau$ is a 
coboundary, then $\Omega_{\tau}$ is  the closure of the  future cone of a point, and
$g_{\tau}$ is the restriction of an affine map. 
Despite the fact that in the literature $\Omega_\tau$ is often considered as an open
subset, it is more convenient to  consider its closure in the present paper.

\begin{definition}
For any cocycle $\tau$, $h_\tau$ will be  the support function of $\Omega_\tau$ (that is the convex envelope of $g_\tau$).\index{Support function! $h_\tau$ the support function of $\Omega_\tau$}
\end{definition}
By Proposition \ref{prop:a=0} any $\tau$-convex set is contained in $\Omega_\tau$. 

\new{
For $x\in \Omega_{\tau}$, the \emph{cosmological time} \index{Cosmological time} is the Lorentzian distance 
from $x$ to the boundary of $\Omega_\tau$.
By \cite{Bon05} the cosmological time is a $\Gamma_\tau$-invariant 
$C^{1,1}$-function of the interior of $\Omega_\tau$ with value on $[0,+\infty)$.
Moreover, the level set $\Sigma_t$ is the boundary of the set $\Omega_\tau+t\jplus(\H)$.
In particular the support function of the future of $\Sigma_t$ on $\H^d$ 
is $\oh_\tau-t$.

For a given $\Gamma_\tau$-invariant hypersurface $S$, we will denote by 
$T_{\operatorname{min}}(S)$ (resp.  $T_{\operatorname{max}}(S)$) \index{Cosmological time! $T_{\operatorname{min}}(S)$,  $T_{\operatorname{max}}(S)$}
the minimum (resp. the maximum) of the cosmological time on $S$.

\begin{lemma}\label{lem: supp et cos time}
Let $\fconv$ be a $\tau$-convex set with hyperbolic support function $\oh$, and let $S$ be its boundary.
Then $T_{\operatorname{min}}(S)$ (resp.  $T_{\operatorname{max}}(S)$)  is equal to 
the minimum (resp. maximum) of $\oh_{\tau}-\oh$ on $\H^d$. 
\end{lemma}
\begin{proof}
Let $\Sigma_m$ be the level set of $T$ corresponding to the level $m=T_{\operatorname{min}}(S)$. 
As  $\fconv\subset \overline{\fut(\Sigma_m)}$, then  $h\leq h_{\tau}-m$, i.e. $h_{\tau}-h\geq m$.
On the other hand, if $p$ is the point of  $\fconv$ realizing the minimum  of the cosmological time on $\partial \fconv$,
then the tangent plane of $\Sigma_m$ at $p$ is a support plane of $\fconv$ at $p$.
If $\hat{x}$, $x\in \ball$, is the corresponding normal vector, then $\oh(x)=\oh_{\tau}(x)-m$.
\end{proof}
}

Recall from Definition~\ref{def:cauchy} that 
a 
 $\tau$-convex set is \emph{Cauchy} if $h<h_{\tau}$.
 
\begin{lemma}\label{lem: cauchy at int tau}
 Cauchy $\tau$-convex sets are exactly the $\tau$-convex sets contained in the interior of  $\Omega_{\tau}$. For such a $\fconv$,
 $\partial \fconv=\partial_s \fconv$.
\end{lemma}
\begin{proof}
If $\fconv$ is in the interior of $\Omega_{\tau}$, it is clearly Cauchy.
Now suppose that $\fconv$ is Cauchy and $\tau$-convex.
By Lemma~\ref{lem: supp et cos time}, $\fconv$ is in the future of the level set of the cosmological time $\Sigma_m$ \old{(see after Remark~\ref{rem:hausdorff dist})}, for
$m>0$ the minimum of $\oh_{\tau}-\oh$.\old{ But $\Sigma_t$ has not light-like support plane. This is easy to see as  it is a hypersurface at constant distance 
from $\partial_s \Omega_{\tau}$ and with the same light-like support planes at infinity.}
\new{But $\Sigma_m$ is contained in the interior of $\Omega_\tau$ (\cite{Bon05}) so the same holds for $\fconv$.}

 The last assertion comes from Lemma~\ref{lem: cuahcy int}.
\end{proof}

\begin{remark}\label{rk:action}{\rm
The existence of a maximal $\tau$-convex subset was already proved in \cite{Bon05} and in \cite{Bar05}
by different arguments.
We remark here some consequences of the theory developed in \cite{Bon05} that we will use
here.

A first important property is that the action of $\Gamma_\tau$ on the interior of $\Omega_\tau$  is proper.
The quotient $(\operatorname{int}\Omega_\tau)/\Gamma_\tau$ is foliated by space-like convex compact hypersurfaces.
Moreover if $\fconv$ is any Cauchy $\tau$-convex set  then
$\partial \fconv/\Gamma_\tau$ is compact.

Since any point in the interior of  $\Omega_\tau$ is contained in some space-like convex hypersurface we get that the action
of $\Gamma_\tau$ in $\operatorname{int}\Omega_\tau$ is space-like in the sense that the segment joining two points on
the same orbit is always space-like.}
\end{remark}

\begin{remark}{\rm 
Proposition~\ref{prop:g tau} implies  that for any light-like vector $\ell\in\partial \ball$, 
the plane $P_{g_\tau}(\ell)$ of equation $\langle x,\ell\rangle_-=g_\tau(\ell)$ does not meet the interior of $\Omega_\tau$.

Notice however that this does not mean that any light-like direction is orthogonal to a 
 light-like support plane of $\Omega_{\tau}$. Actually if $d=2$ and if $\tau$ is not 
 a coboundary, then the set of directions of light-like support plane has zero measure 
 in $\partial \H^2$, \cite[Proposition~4.15]{BMS13}.
 For most of the light-like directions  $P_{g_\tau}(\ell)$ is a support plane at infinity: it does not meet $\Omega_\tau$ but 
any parallel displacement of the hyperplane in the future direction will meet the interior of $\Omega_\tau$. 
 }\end{remark}

We say that a continuous function $\oh:\H^d\to\R$ is $\tau$-equivariant if it
satisfies the equivariance relation \eqref{eq:inv}. 
Notice that $\tau$-equivariant functions are precisely those functions $\oh$ that  can be written as $\oh=\oh_{\tau}+\overline{f}$, with
$\overline{f}=\oh-\oh_{\tau}$ a continuous  $\Gamma$-invariant function.

For any  $\Gamma$-invariant function $\overline{f}$, we define the following  convex set:
$$\fconv_{\tau}(\overline{f})=\{p\in \R^{d+1} | \langle p,\eta\rangle_- \leq \oh_{\tau}(\eta)+\overline{f}(\eta),\, \forall \eta\in\H^d\}~. $$

It is immediate to see that if $p\in\Omega_\tau$ and $a>\sup_{\H^d}\overline{f}$, then $p+a\eta\in \fconv_\tau(\overline{f})$ for 
any $\eta\in\H^d$. So this set is not empty. Moreover, a simple computation shows that it is a $\tau$-convex set.

If $f$ denotes the restriction of the $1$-homogeneous extension of $\overline{f}$ to the ball $\ball$,
the support function $h$ of  $\fconv_{\tau}(\bar f)$ is the convex envelope of $f+h_{\tau}$ (because the convex envelope
corresponds to the pointwise supremum of all affine functions majored by $f+h_{\tau}$ \cite[12.1.1]{Roc97}, that is exactly the definition of
the support function). 
Notice that the restriction of $f$ to $\partial \ball$ is $0$ (see Remark~\ref{rk:ext inv}), 
so if $f$ is non negative, then $h$ equals $h_\tau$, that is
 $\fconv_{\tau}(f)=\Omega_{\tau}$. 
 
 Many of the functionals we will introduce below on the set of $\tau$-support functions
 can be then extended to the space of $\tau$-equivariant continuous functions, but we will only need the following.
 
\begin{lemma}\label{lem:pte pt reg}
 Let $\overline{f}$ be any $\Gamma$-invariant function on $\H^d$,  
 and $\oh$ be the support function of $\fconv=\fconv_\tau(\overline{f})$, and $G$ be the Gauss map of $\fconv$.
 Then on $\partial_{\mathrm{reg}}\fconv$, 
 $$(\oh_{\tau}+\overline{f})\circ \gauss= \oh\circ \gauss~. $$
In particular, the set of points such that $\oh<\oh_{\tau}+\overline{f}$ (i.e. such that $\oh_{\tau}+\overline{f}\not= \oh$) has zero area measure. 
 \end{lemma}
\begin{proof}

Denote by $f$ the restriction of the $1$-homogeneous extension of $\overline{f}$ to $\ball$. 
Clearly $h<h_\tau+f$ at $x\in \ball$ if and only if $\oh<\oh_\tau+\overline{f}$ at $\rad(x)$.

Now  let $p\in \partial_s \fconv$ and $\hat{x}$, $x\in \ball$, be a support vector at $p$, so that
$\langle \hat{x},p\rangle_-= h(x)$.
 Suppose  that $h(x)<(h_{\tau}+ f)(x)$. 
 
  By Remark~\ref{rk:ext inv} and Proposition~\ref{prop:g tau}, 
 the function $h_{\tau}+f$ is continuous on $\overline \ball$.
 Notice that   $\fconv$ coincides with the set 
 $\{p\in \R^{d+1} | \langle p,\hat{x}\rangle_- \leq h_{\tau}(x)+ f(x),\, \forall x\in\overline \ball\}$.
 Since $p\in\partial_s \fconv$, by compactness of $\overline \ball$, there 
 must exist $y\in \overline \ball$ 
 such that $\langle p,\hat{y}\rangle_-=(h_{\tau}+ f)(y)$ (otherwise,  $p$ would be an interior point of $\fconv$). 
 
 Hence $x$ and $y$ are two different support vectors at $p$, so $p$ is not regular.
 The Gauss image of the set of regular points has full area measure  by Proposition~\ref{prop: area zero not c1}.
\end{proof}

\subsection{Hausdorff distance and convergence}\label{sec:33}

The difference of two $\tau$-equivariant maps is $\Gamma$-invariant. This gives a continuous function on the compact manifold 
$\H^d/\Gamma$. The set of $\tau$-convex sets can be endowed with the following distance.
 
 \begin{definition}[Hausdorff distance]\label{def:haus dist}
The \emph{$\tau$-Hausdorff distance} between two $\tau$-convex sets is the uniform norm of the difference
of their support functions.\index{$\tau$-Hausdorff distance} 
\end{definition}

\begin{remark}\label{rem:hausdorff dist}{\rm 
\oldold{Let $\jplus(\H)$ denote the closure of the future of $\H^d$ in $\R^{d+1}$.
It is an F-convex set with support function on $\H^d$ equal to $-1$.}\newnew{Recall that the support function of $\jplus(\H)$ on $\H^d$ equal to $-1$.}
So if $\fconv$ is any F-convex set with support function $h$, the support function
of $\fconv+\lambda \jplus(\H)$ is $h-\lambda$.

It follows that 
 the distance between two $\tau$-convex sets $\fconv$ and $\fconv'$ is the minimum of the $\lambda \geq 0$
such that $\fconv+\lambda \jplus(\H)\subset \fconv'$ and $\fconv'+\lambda  \jplus(\H)\subset \fconv$. 
}\end{remark}

\old{
For $x\in \Omega_{\tau}$, the \emph{cosmological time} \index{Cosmological time} is the Lorentzian distance 
from $x$ to the boundary of $\Omega_\tau$.
By \cite{Bon05} the cosmological time is a $\Gamma_\tau$-invariant 
$C^{1,1}$-function of the interior of $\Omega_\tau$ with value on $[0,+\infty)$.
Moreover, the level set $\Sigma_t$ is the boundary of the set $\Omega_\tau+t\jplus(\H)$.
In particular the support function of the future of $\Sigma_t$ on $\H^d$ 
is $\oh_\tau-t$.

For a given $\Gamma_\tau$-invariant hypersurface $S$, we will denote by 
$T_{\operatorname{min}}(S)$ (resp.  $T_{\operatorname{max}}(S)$) \index{Cosmological time! $T_{\operatorname{min}}(S)$,  $T_{\operatorname{max}}(S)$}
the minimum (resp. the maximum) of the cosmological time on $S$.

\begin{lemma}\label{lem: supp et cos time}
Let $\fconv$ be a $\tau$-convex set with support function $h$, and let $S$ be its boundary.
Then $T_{\operatorname{min}}(S)$ (resp.  $T_{\operatorname{max}}(S)$)  is equal to 
the minimum (resp. maximum) of $\oh_{\tau}-\oh$ on $\H^d$. 
\end{lemma}
\begin{proof}
Let $\Sigma_m$ be the level set of $T$ corresponding to the level $m=T_{\operatorname{min}}$. 
As  $\fconv\subset \overline{\fut(\Sigma_m)}$, then  $h\leq h_{\tau}-m$, i.e. $h_{\tau}-h\geq m$.
On the other hand, if $p$ is the point of  $\fconv$ realizing the minimum  of the cosmological time on $\partial \fconv$,
then the tangent plane of $\Sigma_m$ at $p$ is a support plane of $\fconv$ at $p$.
If $\hat{x}$, $x\in \ball$, is the corresponding normal vector, then $\oh(x)=\oh_{\tau}(x)-m$.
\end{proof}
}
The main tool will be the following convergence result.

\begin{lemma}\label{lem: conv lip}
 Let $\Sigma_t$ be a level set of the cosmological time, 
 and let $(\fconv_n)$ be a sequence of $\tau$-convex sets which meet  the past of $\Sigma_t$. Then
 --- up to passing to  a subsequence ---
 $(\fconv_n)$ converges to a $\tau$-convex set with respect to the $\tau$-Hausdorff distance.
\end{lemma}

\begin{proof}
  Let $(S_{n})$ be the sequence of boundaries of the $\fconv_n$.
 Notice that $S_n$ is the graph of some $1$-Lipschitz function $\tilde u_n$ defined  on the horizontal plane
$\R^d$. By a standard use of Ascoli--Arzel\`a theorem,  if there is a compact subset in $\R^{d+1}$ which meets
all the hypersurfaces $S_n$, then up to passing to a subsequence the functions $\tilde u_n$ converge to a
$1$-Lipschitz  function $u$ uniformly on compact subsets of $\R^d$.

To check that $S_n$ meets a compact region
of $\R^{d+1}$, 
notice that by the assumption there is a point $x_n\in\Sigma_t$ whose past meets  $S_n$.
By the invariance of $S_n$, any point in the orbit of $x_n$ works as well.
Since the quotient $\Sigma_t/\Gamma_\tau$ is compact \cite{Bon05}, then there is a compact fundamental
region $C$ of $\Sigma_t$ for the action of $\Gamma_\tau$.  So we may assume that $x_n\in C$.
Thus, $S_n$ meets $\overline{\operatorname{I}^-(C)}\cap\Omega_\tau$ which is a compact region.

Up to passing to a subsequence, we can deduce that $\tilde u_n$ converges to a
$1$-Lipschitz function $\tilde u$. The graph of $\tilde u$ is the boundary of a $\tau$-convex set $\fconv$.
It remains to prove that $\fconv_n$ converges to $\fconv$ with respect to the $\tau$-Hausdorff distance.

Suppose by contradiction that $(\fconv_n)$ does not converge to $\fconv$ in the sense of Definition~\ref{def:haus dist}. 
That means that there exists a $\epsilon>0$
such that, for arbitrarily large $n$, there exists $x_n\in \fconv_n +\epsilon \jplus(\H)$ with  $x_n\notin \fconv$, 
or $x_n\in \fconv +\epsilon \jplus(\H)$ and $x_n\notin \fconv_n$. 

Let us suppose that we are in the former case. 

Let $P$ be any fundamental region of $\Sigma_t$ and let $\fund$ the union of
the time-like rays in $\Omega_\tau$ orthogonal to $P$.
By \cite{Bon05} it turns out that $\fund$ is a fundamental region of the action of $\Gamma_\tau$  on $\Omega_\tau$,
and the intersection of $\fund$ with the boundary of any $\tau$-convex set is compact.
It turns out that there is a point in the orbit of $x_n$, say $y_n$, that lies in $\overline{\operatorname{I}^-(\fund\cap \partial \fconv)}\cap\Omega_\tau$ 
that is a compact region.

Notice that $y_n$ lies in $\fconv_n+\epsilon \jplus(\H)$, so $y_n=z_n+\epsilon \eta_n$ for some $\eta_n\in\H^d$ and $z_n\in \fconv_n$.
Since $z_n\in \operatorname{I}^-(y_n)\cap\Omega_\tau\subset\overline{\operatorname{I}^-(\fund\cap\partial \fconv)}\cap\Omega_\tau$, 
we can also suppose that $z_n$ converges to $z$.
By the convergence of $\tilde u_n$ to $\tilde u$ we deduce that $z\in \fconv$. Moreover notice that $\langle y-z, y-z\rangle_-=-\epsilon^2$
so $y$ lies in the future of $\fconv$ contradicting the assumption that it is a limit of a sequence of points
in the past of $\fconv$.

The latter case is analogous.
\end{proof}

We can deduce the analogue of the Blaschke Selection theorem for convex bodies.

\begin{corollary}\label{cor:blaschke}
 From each sequence of $\tau$-convex sets uniformly bounded in the future, one can extract a sequence converging to a $\tau$-convex set.
\end{corollary}

Another simple consequence of Lemma \ref{lem: conv lip} is that the cosmological time is uniformly bounded for
any $\tau$-invariant hypersurface that meets a fixed level of the cosmological time, see also
\cite[Proposition 7.4.14]{belt} and \cite[Proposition 4.1]{bel}.

\begin{corollary}\label{cor:cosm borne}
Let  $\epsilon>0$. There exists $T=T(\epsilon,\tau)$ such that
 if $\fconv$ is any  $\tau$-convex set whose boundary $S$ meets the level set of the cosmological time
$\Sigma_{\epsilon}$ then $T_{\operatorname{max}}(S)\leq T$, where $T_{\operatorname{max}}(S)$ is 
the maximum of the cosmological time of $S$. 
\end{corollary}
\begin{proof}
By contradiction suppose that there is a sequence $(\fconv_n)$ of $\tau$-convex sets such that
the boundary of $\fconv_n$, say $S_n$, meets $\Sigma_\epsilon$ and 
%et $(S_{n})$ be a sequence of boundaries of $\tau$-convex sets such that $S_n\cap S_{\epsilon}\not= \emptyset$ and
$T_{\operatorname{max}}(S_n)\geq n$.
By Lemma~\ref{lem: conv lip}, up to passing to a subsequence, 
$(\fconv_n)$ converges to  an F-convex set $\fconv$ in the $\tau$-Hausforff sense.
This implies that the corresponding hyperbolic support functions  $\oh_\tau-\oh$ are 
uniformly bounded, but this contradicts Lemma \ref{lem: supp et cos time}. 
\end{proof}

\begin{lemma}\label{lem: max/min} There is a constant $\delta=\delta(\tau)>0$ such that if $T_{\operatorname{min}}(S)\geq 1$ 
then $$T_{\operatorname{max}}(S)/T_{\operatorname{min}}(S)<\delta~.$$
\end{lemma}
\begin{proof}
 
Suppose that there is a sequence of 
$\tau$-convex sets $\fconv_n$ with boundaries
 $S_n$ such that $T_{\operatorname{min}}(S_n)\geq 1$ and
$T_{\operatorname{max}}(S_n)/T_{\operatorname{min}}(S_n)>n$. 
From Corollary~\ref{cor:cosm borne},  $ T_{\operatorname{min}}(S_n)\rightarrow +\infty$.
Let us set $t_n=T_{\operatorname{min}}(S_n)$ and consider the sequence  of rescaled hypersurfaces  $S'_n=(1/t_n)S_n$. 
Notice that $S'_n$ is the boundary of the $(1/t_n)\tau$-convex set $(1/t_n)\fconv_n$ and
 satisfies  $T_{\operatorname{min}}(S'_n)=1$.

So for any $n$, $S'_n$ has a common point with the level set $1$ for the cosmological time  $T_n$
of $\frac{1}{t_n}\Omega_\tau=\Omega_{\frac{1}{t_n}\tau}$. 
By \cite[6]{Bon05},  the graph functions of   those level sets  of $T_n$
converge in the compact open topology to the graph function of the cosmological time $T_0$ of
$\Omega_0=\fut(0)$.
Moreover for every $c$ there is a compact region $U_c$ of $\R^{d+1}$  (independent of $n$) that  contains a fundamental region
of the level set $T_n^{-1}(c)$ for the action of $\Gamma_{\frac{1}{t_n}\tau}$ for  every $n$.

 As for each $n$, the intersection point of $S'_n$ with $T_n^{-1}(1)$ can be chosen in $U_1$, 
 we  deduce that $S'_n$ meets $U_1$ for every $n$.
  As $S'_n$ are graphs of  $1$-Lipschitz functions $\tilde v_n$, we deduce that --- up to extracting a subsequence ---
  $\tilde v_n$ converges to a convex $1$-Lipschitz function 
  $\tilde v$ uniformly on compact subsets    as a  consequence of 
 Arzel\`{a}--Ascoli theorem.
The graph of $\tilde v$,  is the boundary of a convex subset $\fconv$
which is invariant by the action of $\Gamma$ by construction.

\old{Since $S'_n$ is contained in the closure of the future of $T_n^{-1}(1)$ then $\partial \fconv$ is contained
in $\overline{\fut(\H^d)}=\jplus(\H)$ in $\R^{d+1}$. In particular, it is a Cauchy surface in $\fut(0)$. }

\new{As $U_1$ contains a fundamental region of the action of $\Gamma_{\frac{1}{t_n}\tau}$ over $T_n^{-1}(1)$,
there is a sequence of points $x_n\in U_1$ such that  $x_n\in S'_n\cap T_n^{-1}(1)$.
Taking the limit (up to a subsequence) of $(x_n)$, we deduce that $\partial\fconv$ meets the level set $T_0^{-1}(1)$.}
% $U_1\cap T_n^{-1}(1)$, we have that $\partial\fconv$

By Corollary \ref{cor:cosm borne}  $\partial \fconv$ is contained in the past of the level set $T_0^{-1}(c)$,  
for some fixed constant $c$.
Since the graph functions of $S'_n$ converge to the graph function of $\partial \fconv$ uniformly on compact subsets
of $\R^d$ and the same holds for the graph functions of $T^{-1}_n(c)$, we deduce that for any compact region $V$ 
of $\R^{d+1}$ there is $n_0$ (a priori depending on $V$)
 such that $S'_n\cap V$ is in the past of $T^{-1}_n(c)\cap V$ for $n\geq n_0$.

Since the compact region $U_c$ contains a fundamental region of $T_n^{-1}(c)$ for very $n$, we deduce that for $n$ big
$S'_n$ is contained in the past of $T^{-1}_n(c)$.

But by our starting assumption,  one also has that $T_{\operatorname{max}}(S'_n)=T_{\operatorname{max}}(S_n)/t_n\geq n$, %that contradicts Corollary~\ref{cor:cosm borne}.
and this gives a contradiction.
\end{proof}

\subsection{Area and total area}

In the sequel, except if explicitly mentioned, the integrals will be on the compact manifold $\H^d/\Gamma$.

 Let us recall notations from Subsection~\ref{sub:back}. If $\fconv$ is a $\tau$-convex set, then
$V_{\epsilon}(\fconv)$ is clearly a $\Gamma$-invariant measure, and so are the area measures $S_i(\fconv)$.
Any $\Gamma$-invariant Radon measure $\mu$ on $\H^d$ defines a Radon measures in $\H^d/\Gamma$: 
it is the only measure $\bar\mu$ on $\H^d/\Gamma$ such that if $U\subset \H^d/\Gamma$ is a Borel set and 
$\psi: U \to\H^d$ is a measurable section of the projection $\pi:\H^d\to\H^d/\Gamma$, then 
$\psi_*(\bar\mu)=\mu$.

In practice, if $E$ is a Borel subset of $\H^d$ which meets every orbit of $\Gamma$ exactly once, then 
the measure of any Borel set $\omega$
of $\H^d/\Gamma$ is $\mu(\pi^{-1}(\omega)\cap E)$.

In our case, we obtain  measures $\overline{V}_{\epsilon}(\fconv)$  and $\overline{S}_i(\fconv)$ on 
$\H^d/\Gamma$.

Using \eqref{eq: mu} we deduce that
\begin{equation}\label{eq: muinv}
\overline{V}_{\epsilon}(\fconv)=\frac{1}{d+1}\sum_{i=0}^d\epsilon^{d+1-i}\binom{d+1}{i}\overline{S}_i(\fconv)~.
\end{equation}

\begin{remark}{\rm
Notice that for any Borel subset $\omega\subset\H^d/\Gamma$,
 $\overline{S}_d(\fconv)(\omega)$ is the derivative of $\overline{V}_\epsilon(\fconv)(\omega)$ for $\epsilon=0$.

As $\overline{V}_{\epsilon+h}(\fconv)=\overline{V}_{\epsilon}(\fconv)+\overline{V}_h(\fconv+\epsilon \jplus(\H))$ we get that
for any Borel subset $\omega\subset\H^d/\Gamma$
\[
\frac{\d\overline{V}_\epsilon(\fconv)(\omega)}{\d\epsilon}(\epsilon)=\overline{S}_d(\fconv+\epsilon \jplus(\H))(\omega)~,
\]
or equivalently
\begin{equation}\label{eq:V int}
  \overline{V}_\epsilon(\fconv)(\omega)=\int_0^\epsilon \overline{S}_d(\fconv+t \jplus(\H))\d t~.
\end{equation}
}
\end{remark}

We also get an invariant formulation of the weak-convergence stated in Lemma~\ref{lem: weak conv are fconv}.
\begin{lemma}\label{lem:weak inv}
 If the support functions of $\tau$-convex sets converge for the sup norm to the 
 support function of a $\tau$-convex set, then the associated measures $\overline{V}_\epsilon$
 and $\overline{S}_i$ on $\H^d/\Gamma$ weakly converge.
\end{lemma}
\begin{proof}
As any continuous  function on $\H^d/\Gamma$ can be expressed as the sum of continuous 
functions with supports in small  open sets by means of a partition of the unity, it is sufficient to prove that
$$  \int f \d \overline{V}(\fconv_n)\rightarrow \int f \d \overline{V}(\fconv)$$
assuming that the support of $f$ is contained in a small chart $U$ where a section $\psi:U\to\H^d$ is defined.
But, by definition we have
$\int f\d\overline{V}(\fconv_n)=\int_U f \d\overline{V}(\fconv_n)=\int_{\psi(U)} f\circ\psi^{-1} \d V(\fconv_n)$ and
so the convergence stated directly follows from Lemma~\ref{lem: weak conv are fconv}.

 \end{proof}
 
 We still denote by $\A(\fconv)$ the measure $\overline S_d(\fconv)$ on $\H^d/\Gamma$, and
 by $\operatorname{Area}(\fconv)$ the \emph{total area} of $\fconv$, i.e.$$\operatorname{Area}(\fconv)=\A(\fconv)(\H^d/\Gamma)~.$$ \index{ $\operatorname{Area}(\fconv)$ the total area}
 
 \begin{corollary}\label{cor: weak conv area}
 If the support functions of $\tau$-convex sets converge for the sup norm to the 
 support functions of a $\tau$-convex set, then the corresponding area measures weakly converge. In particular, the total
 area is continuous for this topology.
 \end{corollary}

  If $\partial_s \fconv$ is $C^1$, then the total area is the total volume of the manifold $\partial_s \fconv/\Gamma_{\tau}$
 for the induced metric by Corollary~\ref{cor: area C1}.

We prove below that the total area is monotonic.  

\begin{lemma}\label{lem:area monotonic} 
Let $\fconv_0,\fconv_1$ be two $\tau$-convex sets, with $\fconv_1$ contained in the interior of $\fconv_0$.
 Then $$\Area(\fconv_0)\leq \Area(\fconv_1)~.$$
 \end{lemma}
\begin{proof}
Let us first suppose that $\fconv_0$ and $\fconv_1$ are $C^2_+$, with respective hyperbolic 
support functions $\oh_0$ and $\oh_1$ on $\H^d$. Note that $\oh_1<\oh_0$.
 For $\lambda\in(0,1)$, let $\fconv_{\lambda}=(1-\lambda) \fconv_0+\lambda \fconv_1$. It has hyperbolic
 support function $(1-\lambda) \oh_0+\lambda \oh_1$,
so $\fconv_{\lambda}$ is  $C^2_+$. 
We can use the inverse of the Gauss map of $\fconv_\lambda$, say $\chi_\lambda$,
 to get a smooth family of parameterizations of $\partial \fconv_\lambda$.
 Recall  by  \eqref{eq:chi on hyp}
that  $\chi_\lambda(\eta)=\grad^{\H}_\eta\oh_\lambda-\oh_\lambda(\eta)\eta$,
so the variational field $X=\partial_\lambda\chi_\lambda$ at a point $\chi_\lambda(\eta)$ is given by
$$X(\chi_\lambda(\eta))=\grad^{\H}_\eta(\oh_1-\oh_0)-(\oh_1-\oh_0)(\eta)\eta~.$$
The variation of the area of $\partial \fconv_\lambda$ is given by
\[
\new{\frac{d\,}{d\lambda}} \Area(\fconv_\lambda)=-\int_{\partial \fconv_\lambda/\Gamma_\tau}s_\lambda\langle X(x), \gauss_\lambda(x)\rangle_- 
\d\mu_\lambda
\]
where $\mu$ is the intrinsic volume form and $s_\lambda$ is the sum of the principal eigenvalues computed with respect the
normal $\gauss$ (so $s_\lambda$ is positive).

Using the parameterization given by $\chi_\lambda$ and the fact that
$\chi_\lambda^*(\mu_\lambda)= \d\A(\fconv_\lambda)$ we conclude that
\[
\new{\frac{d\,}{d\lambda}}  \Area(\fconv_\lambda)=
\int s_\lambda\langle X(\chi_\lambda(\eta)),\eta\rangle_- \d \A(\fconv_\lambda)=
\int s_\lambda (\oh_0-\oh_1)\d \A(\fconv_\lambda)\geq 0~.
\]

The general case follows by a density argument, see
Appendix~\ref{appen: smooth approx} and Corollary~\ref{cor: weak conv area}. 
\end{proof}

\begin{proposition}\label{prop: c tau}
There exists a non-negative number $c(\tau)$ such that for every $\tau$-convex set $\fconv$, 
 if $\Area(\fconv)>c(\tau)$ then $\fconv$ is a Cauchy convex set (that is, it is contained in the interior of  $\Omega_{\tau}$).
 \end{proposition}

 As already noted, if $\tau$ is a coboundary, i.e.~$\Omega_{\tau}$ is the future cone of a point, 
 and $\fconv$ is a $\tau$-convex set with $\Area(\fconv)\not= 0$, then
 $\fconv \not= \Omega_{\tau}$ that implies that $\fconv$ is Cauchy.
 So in this case $c(\tau)=0$.
 On the other hand, the example constructed in Subsection~\ref{sub: pogo}
implies that in general  $c(\tau)\neq 0$. % cannot be chosen  arbitrarily small.
 
 \begin{proof}
For any $n>0$, let $\fconv_n$ be such that $\Area(\fconv_n)>n$ and suppose that $\fconv_n$ touches the boundary of  $\Omega_\tau$.
 By Corollary~\ref{cor:cosm borne},  $\fconv_n$ is in the past of $\Sigma_T$ for some $T$ independent of $n$, 
 hence by Lemma~\ref{lem:area monotonic},
 $\Area(\fconv_n)\leq\Area(\Sigma_T)$, that is again a contradiction if $n$ is sufficiently large. 
\end{proof}

\subsection{The covolume}

 Let $\fconv$ be a $\tau$-convex set.  The \emph{covolume} of $\fconv$, denoted by $\mathrm{covol}(\fconv)$ \index{ $\mathrm{covol}(\fconv)$ the covolume of $\fconv$)}
is the volume of the complementary of $\fconv/\Gamma_\tau$ in the manifold $\Omega_{\tau}/\Gamma_\tau$. 
As this complementary is open, the covolume is well defined. It is also finite. This can be seen for example for level sets 
of the cosmological time, then using the monotonicity of the covolume.
Equivalently, it is the volume (the Lebesgue measure) of $\Omega_{\tau}\setminus \fconv$ in a fundamental
domain for $\Gamma_\tau$ (note that Minkowski isometries preserve the Lebesgue measure).

It can be convenient  to consider the covolume as a function on the set of
$\tau$-equivariant  hyperbolic support functions. For this reason we often will write
$\cov(\oh)$ to denote the covolume of the corresponding convex set. 
In the same way we will denote by $\A(\oh)$ the area measure on $\H^d$ induced by the domain $\fconv$.

\begin{lemma}\label{lem: cov cont}
The function  $\cov$ is continuous on the set of $\tau$-convex sets.
\end{lemma}
\begin{proof}
 Let $\fconv_n$ be at distance $1/n$ from $\fconv$. From Remark~\ref{rem:hausdorff dist}

$$\fconv+\frac{1}{n} \jplus(\H) \subset \fconv_n \mbox{ and } \fconv_n+ \frac{1}{n} \jplus(\H) \subset \fconv~,$$
so
$$\cov\left(\fconv+\frac{1}{n} \jplus(\H) \right) \geq \cov(\fconv_n)   \mbox{ and  }   \cov\left( \fconv_n+ \frac{1}{n} \jplus(\H)\right)  \geq  \cov(\fconv)~.$$

But clearly $\cov\left(\fconv+\frac{1}{n} \jplus(\H)\right)= \cov(\fconv)+\overline{V}_{1/n}(\fconv)(\H^d/\Gamma)$, 
so 
$$\cov(\fconv_n) - \cov(\fconv) \leq \overline{V}_{1/n}(\fconv)(\H^d/\Gamma)~. $$
\new{By \eqref{eq: muinv}} $\overline{V}_{1/n}(\fconv)(\H^d/\Gamma)$ goes to $0$ when $n$ goes to infinity 
\old{(this is clear using an expression analog to \eqref{eq:vol approx int} below and the argument after it)}.
In the same way 
$$\cov(\fconv) - \cov(\fconv_n) \leq \overline{V}_{1/n}(\fconv_n)(\H^d/\Gamma)~, $$
so to conclude it remains to prove that $\overline{V}_{1/n}(\fconv_n)(\H^d/\Gamma) $ goes to $0$ when $n$ goes to infinity. 
But by \eqref{eq:V int} we have
\begin{equation}\label{eq:vol approx int} \overline{V}_{1/n}(\fconv_n)(\H^d/\Gamma)=\int_0^{1/n} \Area(\fconv_n+t\jplus(\H)) \d t~. \end{equation}
By Lemma~\ref{lem:area monotonic},  $\Area(\fconv_n+t\jplus(\H))\leq \Area (\fconv_n+1/n\jplus(\H))\leq \Area(\fconv_n+\jplus(\H))$. As $\fconv_n$
converges to $\fconv$, $\fconv_n+\jplus(\H)$  are in the past of a $ \tau$-convex set $C$, so again by Lemma~\ref{lem:area monotonic},
$\Area(\fconv_n+t\jplus(\H))\leq \Area(C)$ and $\overline{V}_{1/n}(\fconv_n)(\H^d/\Gamma)\leq 1/n\Area(C)$.
\end{proof}

As we noted in  Section \ref{sub def tau} 
the set of $\tau$-convex sets is convex.
So speaking about convexity of the covolume is meaningful. 
We will first prove that it is convex on the convex subset of Cauchy $\tau$-convex sets. 
The convexity on the set of $\tau$-convex sets will follow by 
continuity. The idea is to write the covolume in terms of a volume of a convex body, obtained by cutting some
convex fundamental domain by a hyperplane. Such convex fundamental is described in the following.

\begin{proposition}\label{prop:fund reg}
Let $S$ be any convex $\Gamma_\tau$-invariant $C^1$ space-like hypersurface in $\Omega_{\tau}$.
Given $x,y\in\mathbb R^{d+1}$, let us set $\psi(x,y)=\la x-y, x-y\ra_-$.
Let us fix a point $x\in S$ and consider the region
\[
  \fund(x)=\{y\in \fut(S)| \psi(x,y)\leq\psi(\gamma x+\tau_\gamma, y)~\forall   \gamma\in\Gamma\}~.
\]
Then $\fund(x)$ is a convex fundamental domain of $\fut(S)$ for the action of $\Gamma_\tau$.

Moreover the interior of $\fund(x)$ is the set
\begin{equation}\label{eq:fundint}
\operatorname{int}\fund=\{y\in \fut(S)|\psi(x,y)<\psi(\gamma x+\tau_\gamma, y)~\forall  \gamma\in\Gamma\setminus\{1\}\}~.
\end{equation}
\end{proposition}

\new{
\begin{definition}
The domain $\fund(x)$ is the \emph{Lorentzian Dirichlet polyhedron} centered at $x$.
\end{definition}
}
 
Note that  in the Fuchsian case it is possible to construct a convex 
fundamental domain  for the whole domain $\Omega$ (the future cone
of the origin), as it suffices to consider the cone in $\R^{d+1}$ of a Dirichlet polyhedron  for $\Gamma$ in
$\H^d$. Moreover,  if $x\in\H^d$, the polyhedron $\fund(x)$ defined in Proposition~\ref{prop:fund reg} 
coincides with the cone on the  Dirichlet polyhedron centered at $x$.

To prove the Proposition we need the following lemma:
\begin{lemma}\label{lem:fund}
For any divergent sequence $\gamma_n\in\Gamma$ and any compact set $C$ in $\overline{\fut(S)}$
\begin{equation}\label{eq:div}
           \inf_{y\in C}\psi(\gamma_nx+\tau_{\gamma_n}, y)\rightarrow+\infty~.
\end{equation}
\end{lemma}
\begin{proof}
This property is easily checked if $C\subset S$. 
Indeed for $y\in S$,  $\sqrt{\psi(\gamma_n(x)+\tau_{\gamma_n},y)}$ 
is bigger than the intrinsic distance $d_S(\gamma_n(x)+\tau_{\gamma_n},y)$.

On the other hand, since $S/\Gamma_\tau$ is compact, the intrinsic distance on $S$
is complete and since the action of $\Gamma_\tau$ is proper we conclude that 
$d_S(\gamma_n(x)+\tau_{\gamma_n},C)$ diverges as
$n\rightarrow+\infty$.

Now if $C\subset  \fut(S)$, we can consider its projection onto $S$. 
For $y\in \fut(S)$, there is a unique point $y'$ on $S$ such that $y=y'+e$
with $e$ time-like supporting vector of $S$ at $y'$ (\cite{Bon05}).
Then for any $z\in S$ we have that
\[
  \psi(z,y)=\psi(z,y')+\la e, e\ra_- +2\la y'-z, e\ra_-~.
\]
Notice that since the plane through $y'$ and orthogonal to $e$ is a support plane for $S$, we have
that $\la y'-z, e\ra_-\geq 0$.

Then putting $c=\sup_{y'\in \fconv'}|\la e, e\ra_-|$, then $\psi(z,y)\geq \psi(z, y')-c$ for any $z\in S$ and $y\in C$.
Then (\ref{eq:div})  follows. 
\end{proof}

\begin{proof}[Proof of Proposition~\ref{prop:fund reg}]
If $x_1, x_2\in\mathbb R^{d+1}$ are two space-like related points,
then the set of points $y$ such that $\psi(x_1, y)\leq \psi(x_2, y)$ is the closed half-space in
$\mathbb R^{d+1}$ bounded by the plane orthogonal to the segment $[x_1, x_2]$ and passing through
its middle point $(x_1+x_2)/2$.
Since the orbit of any point $x\in\Omega_\tau$ is space-like (see Remark~\ref{rk:action}), 
 $\fund(x)$ is convex.

In order to show that the $\Gamma_\tau$-orbit of any point $y\in \fut(S)$ meets $\fund(x)$, 
notice that by Lemma~\ref{lem:fund},  for any $y\in \fut(S)$, there
is a point $x_1$ on the orbit of $x$ where the function $\psi(\bullet, y)$ attains the minimum.
If $x_1=\gamma_1 x+\tau_{\gamma_1}$, then putting $y_1=\gamma_1^{-1} y+\tau_{\gamma_1^{-1}}$
we have that on the orbit of $x$ the minima of the function  $\psi(\bullet, y_1)$ is attained at $x$ so
$y_1\in \fund(x)$.

Finally proving \eqref{eq:fundint} we will have 
the interior of $\fund(x)$ contains at most $1$ point for each orbit, that concludes the proof that $\fund(x)$ is a fundamental region.

So, it remains to prove that
\[
\operatorname{int}\fund=\{y\in \fut(S)|\psi(x,y)<\psi(\gamma x+\tau_\gamma, y)~\forall  \gamma\in\Gamma\setminus\{1\}\}~.
\]
The inclusion $\subset$ is immediate.
Suppose by contradiction that there is a point $y$ on the boundary of $\fund(x)$ such that
$\psi(y,x)<\psi(y, \gamma x+\tau_\gamma)$  $\forall \gamma\in\Gamma\setminus\{1\}$. 

There exists a sequence $y_n\in \overline{\fut(S)}$ converging to $y$ and $\gamma_n$ such that
\[
    \psi(\gamma_n(x)+\tau_{\gamma_n},y_n)<\psi(x,y_n)~. 
\]
As $y_n$ converges to $y$ we may apply \eqref{eq:div} and deduce that $\gamma_n$ is not diverging.
Since $\Gamma$ is discrete, up to passing to a subsequence we may assume that $\gamma_n=\gamma$ is constant different 
from $1$.

But then passing to the limit in the above inequality we get
\[
  \psi(\gamma(x)+\tau_\gamma, y)\leq \psi(x,y)~,
\]
that contradicts the assumption on $y$.
\end{proof}

%%%%%%%%%%%%%%%%%%%%%%%%%%%%%%%%%%%%%%%%%%%%%%%%%%%%%%%%%%%%%%%%%%%%%%%%%%%%%%%%%%%%%%%%%%%%%%%%%%%%%%%%%%%%%%%%%%
%%%%%%%%%%%%%%%%%%%%%%%%%%%%%%%%%%%%%%%%%%%%%%%%%%%%%%%%%%%%%%%%%%%%%%%%%%%%%%%%%%%%%%%%%%%%%%%%%%%%%%%%%%%%%%%

The following lemma in the Euclidean setting will be the main ingredient of the convexity
of the covolume. 
Although it is classical, see e.g.~\cite[50.]{BF87}, we reproduce the proof for completeness.

\begin{lemma}\label{lem:vol concave}
In the Euclidean space $\R^{d+1}$ let $R$ be a hyperplane orthogonal (with respect to the Euclidean structure)
to a vector $u$.
Let $D$ be a convex body with non-empty interior in $R$, and 
 $C_0$ and $C_1$ be two convex bodies in $\R^{d+1}$ contained in the same side of  $R$ as $u$, 
such that their orthogonal projection 
onto $R$ is $D$.
Then, for $\lambda\in[0,1]$,
$$V((1-\lambda)C_0+\lambda C_1)\geq (1-\lambda) V(C_0)+\lambda V(C_1)~,$$
where $V$ is the volume.

Equality holds if and only if either
$C_0=C_1+U$ or $C_1=C_0+U$, where $U$  is some segment directed by $u$.
\end{lemma}

\oldold{Such convex bodies are sometimes called \emph{canal class} \cite{sch93}.}\newnew{Convex bodies with  the same image for the orthogonal projection onto a given hyperplane are said to form a \emph{canal class} \cite{sch93}.}

\begin{proof}
For any $x\in D$, let $L(x)$ be the line orthogonal to $R $ and passing through $x$. 
Take any $k_i\in C_i\cap L(x)$, $i=1,2$. The convex combination of $k_0$ and $k_1$
 belongs to $ ((1-\lambda)C_0+\lambda C_1)=:C_{\lambda}$ and also to $L(x)$.
 So the convex combination of  $C_i\cap L(x)$,  for $i=0,1$, is a subset of  $C_{\lambda}\cap L(x)$.
 Denoting %by $l_i(x)$ and
 by  $l_{\lambda}(x)$ the  length of the segment $C_\lambda \cap L(x)$, 
 %$i=0,1$ and $C_{\lambda}\cap L(x)$, 
 and using the linearity of the length, we get
 \begin{equation}\label{eq:long comb conv}l_{\lambda}(x)\geq (1-\lambda)l_1(x)+\lambda l_2(x)~. \end{equation}
 Of course $C_{\lambda}$ projects orthogonally onto $D$, and by Fubini Theorem,
\begin{equation}\label{eq: fub vol}V(C_{\lambda})=\int_{D} l_{\lambda}(x)d x\geq (1-\lambda) \int_D l_0(x)d x+\lambda \int_D l_1(x)d x=
 (1-\lambda)V(C_0)+\lambda V(C_1)~.\end{equation}

 Suppose that $C_1=C_0+U$ where $U$ is a segment parallel to $u$ of length $c$. Then 
 $C_{\lambda}=C_0+\lambda U$, so
 $l_{\lambda}(x)=l_0(x)+\lambda c$, and equality holds in \eqref{eq:long comb conv} and hence in \eqref{eq: fub vol}.
 
 Now suppose that  equality holds in \eqref{eq: fub vol}.  Then equality holds in \eqref{eq:long comb conv} and so
 for every $x\in D$ and every $\lambda\in[0,1]$
 \begin{equation}\label{eq:ccc}
 (1-\lambda)\left( C_0\cap L(x)\right) + \lambda \left(C_1 \cap L(x)\right)=C_{\lambda}\cap L(x)~.
 \end{equation} 
 
For $x\in D$, let $y_{1/2}$ be an end-point of the vertical segment  $L(x)\cap C_{1/2}$, and let
$v$ be a (Euclidean) support vector at $y_{1/2}$.

By \eqref{eq:ccc} there are end-points $y_0$, $y_1$ on $L(x)\cap C_0$
and $L(x)\cap C_1$ respectively such that $y_{1/2}=\frac{1}{2}y_{0}+\frac{1}{2}y_1$.

Denoting  by $H^e_{\lambda}$  the (Euclidean) support functions of 
 $C_{\lambda}$,
we have
  $$H^e_{1/2}(v)=\langle y_{1/2},v\rangle=\frac{1}{2}\langle y_0,v\rangle +\frac{1}{2}
 \langle y_1,v\rangle~,$$
 where  $\langle \cdot,\cdot\rangle$ is the positive  product of $\R^{d+1}$.

 On the other hand, 
 $$H^e_{1/2}(v)=\frac{1}{2}H^e_0(v)+\frac{1}{2} H^e_1(v)~, $$
 that, together with the previous equation and with  $\langle y_i,v\rangle \leq H^e_i(v)$, leads to
 $H^e_i(v)=\langle y_i,v\rangle$. 
 
 In other words,  $C_0$ and $C_1$ have parallel support hyperplanes 
 along the lines directed by $u$. 
 On any segment in the interior of $D$, part of the boundaries of $C_0$ and $C_1$ are graph of convex or concave functions
 with parallel tangents, so they have the same derivative almost everywhere, hence they differ by a constant $c_1$ for the convex part
 and $c_2$ for the concave part.
 \old{
So for example $H^e_0(v)=\langle y_0,v\rangle = \langle y_1+c_1u,v\rangle=
 H^e_1(v)+c_1\langle v,u\rangle$
for suitable $v\in\R^{d+1}$. 
Note that as
$C_0$ and $C_1$ are contained in the cylinder which projects onto $D$, 
$H_0^e(v)=H_1^e(v)$  for any 
$v\in \R^{d+1}$ with $\langle u,v\rangle=0$. 
Finally,
$H_0-H_1=\operatorname{max}(\langle c_1 u,\cdot\rangle,\langle c_2 u,\cdot\rangle)$ and
 this last term is the support function of the segment between $c_1u$ and $c_2u$.

}
\new{
Let $C_0'=C_0+c_1 u$. Notice that the convex part of the boundary of $C_0'$ coincides with with the convex part of the boundary of $C_1$,
whereas the concave part differ by $c_2-c_1$. If $c_2-c_1>0$ then $C_1=C_0'+U'$, where $U$ is the segment joining the origin to $|c_2-c_1|u$.
On the other hand if $c_2-c_1<0$, then $C_0'=C_1+U'$. In the former case $C_1=C_0+[c_1 u, c_2u]$, in the latter case $C_0=C_1+[-c_1u, -c_2u]$
}
 \end{proof}

From \ref{lem:vol concave} we can simply  deduce an analogous statement in 
the Lorentzian setting.

\begin{corollary}\label{cor:vol concave lor}
In Minkowski space $\mink$ let $R$ be a space-like
hyperplane orthogonal (for the Minkowski product)
to a vector $u$.
Let $D$ be a convex body with non-empty interior in $R$, and %orthogonal, and 
 $C_0$ and $C_1$ be two convex bodies in $\R^{d+1}$ contained in the same side of  $R$ as $u$, 
such that their orthogonal projection (in Minkowski sense) 
onto $R$ is $D$.
Then, for $\lambda\in[0,1]$,
$$V((1-\lambda)C_0+\lambda C_1)\geq (1-\lambda) V(C_0)+\lambda V(C_1)~,$$
where $V$ is the volume.

Equality holds if and only if either
$C_0=C_1+U$ or $C_1=C_0+U$, where $U$ is some segment directed by $u$.
\end{corollary}
\begin{proof}
If $R$ is the horizontal plane $\R^d\subset\R^{d+1}$, then the Lorentzian
orthogonal projection coincides with the Euclidean projection, so the statement is 
an immediate consequence of Lemma~\ref{lem:vol concave}. 

If $R$ is not horizontal, there is a Lorentzian isometry $f$ sending $R$ to the horizontal 
plane $\R^d\subset\R^{d+1}$. 
Notice that the orthogonal projections of $f(C_i)$ to $\R^{d}$ coincide with $f(D)$. So we deduce
$$V((1-\lambda)f(C_0)+\lambda f(C_1))\geq (1-\lambda) V(f(C_0))+\lambda V(f(C_1))~.$$
Since $f$ is volume preserving and $(1-\lambda)f(C_0)+\lambda f(C_1)=f((1-\lambda)C_0+\lambda C_1)$
we obtain the result.
\end{proof}

\begin{figure}[h!]
\centering
\includegraphics[scale=0.15]{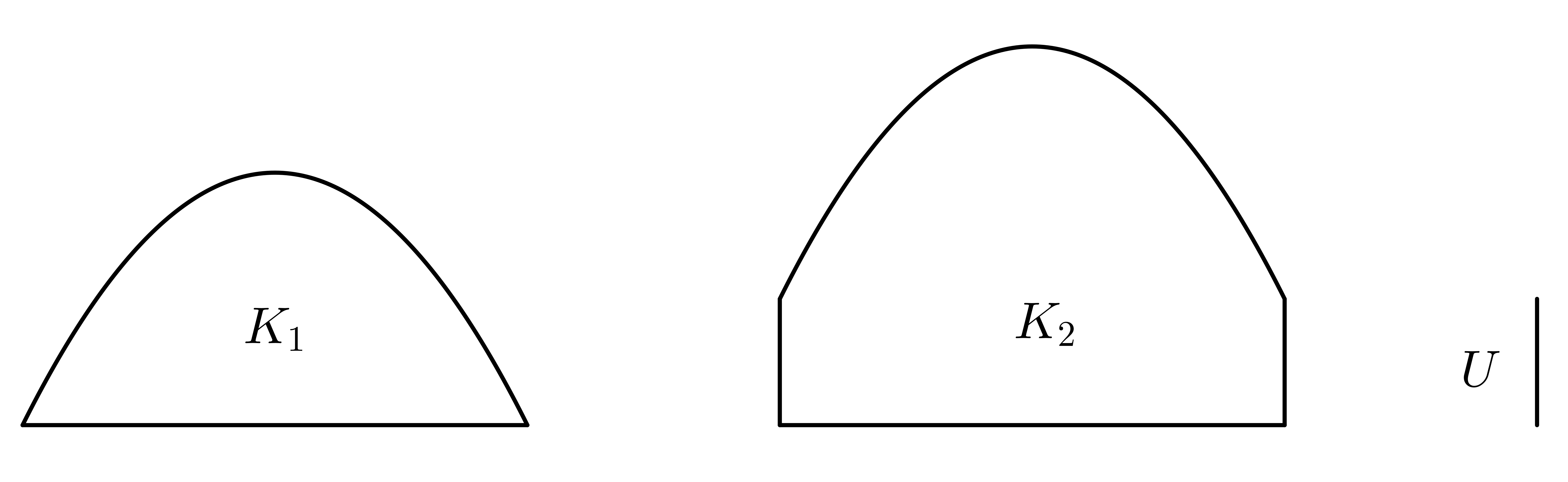}
\caption{$\fconv_1$ is a convex cap, and $\fconv_2=\fconv_1+U$. For $\lambda\in[0,1]$,  the volume is linear along
$(1-\lambda)\fconv_1+\lambda \fconv_2$, but for $\lambda\not= 0$, this is not a convex cap.}
\label{fig:caps}
 \end{figure}

Let $D$ be a convex body contained in a space-like hyperplane $R$ with non-empty relative interior.
A (Lorentzian)  \emph{convex cap} based on $D$ is a convex body in $\R^{d+1}$ with no empty interior such that
\begin{itemize}
\item $R$ is a support plane of $\fconv$ and $\fconv\cap R=D$
\item For each $x\in\partial D$ the line $L(x)$ through $x$ orthogonal in the Minkowski sense to $R$ meets $\fconv$ only at $x$.
\end{itemize}

Notice that if $\fconv$ is a convex cap based on $D$ then the second condition implies that
$D$ is the orthogonal projection of $\fconv$ to $R$, but the condition is not equivalent (see Figure \ref{fig:caps}).

Since two caps based on $D$ cannot differ by a segment, Corollary \ref{cor:vol concave lor}
immediately gives the following consequence.
 
\begin{corollary}\label{cor:vol cap}
 The volume is strictly concave on the set of convex caps (of same base).
\end{corollary}
Notice that the same result holds as well for caps in Euclidean setting.

We can prove now the convexity of the covolume.

\begin{proposition}\label{prop: vol conv}
 The covolume is strictly convex on the set of Cauchy $\tau$-convex hypersurfaces.
\end{proposition}

\new{
First we prove  the following technical lemma that is needed in the proof of Proposition \ref{prop: vol conv}.
\begin{lemma}\label{lm: fund}
Let $S$ be a $\tau$-invariant strictly convex $C^1$ hypersurface, and fix $s\in S$.
Let $\fund=\fund(s)$ be the Lorentzian Dirichlet polyhedron centered at $s$ of $\fut(S)$.
For any support vector \newnew{$u$} of $S$ at $s$ and for any $y\in\fund$, the point
   $y+u$ lies in the interior of $\fund$.
\end{lemma}
\begin{proof}
  By a direct computation we have 
  \begin{equation}  \label{eq:condom}
  \psi(y+u,\gamma_\tau(s))-\psi(y+u, s)=\psi(y, \gamma_\tau(s))-\psi(y,s)+2\langle u, s-\gamma_\tau(s)\rangle_-~.
  \end{equation}
  Notice that the last term is strictly positive since $u$ is a support vector of $S$ at $s$, whereas
  $\psi(y, \gamma_\tau(s))-\psi(y,s)\geq 0$ by the assumption on $y$.
  By \eqref{eq:fundint} it follows that $y+u$ is in the interior of $\fund$.
\end{proof}
}

\begin{proof}[Proof of Proposition \ref{prop: vol conv}]
 Let $\fconv_1$ and $\fconv_2$ be two Cauchy $\tau$-convex hypersurfaces,   and for $\lambda\in[0,1]$
 let $$\fconv_{\lambda}=(1-\lambda)\fconv_1+\lambda \fconv_2~.$$ Let us choose
 a strictly convex space-like $C^1$ $\tau$-invariant convex hypersurface $S$ such that the $\fconv_{\lambda}$ are in the
 future of $S$.
   Let $s\in S$ and $\fund:=\fund(s)$ be the Lorentzian Dirichlet polyhedron of $\fut(S)$ centered at $s$.
\old{
   given by Proposition~\ref{prop:fund reg}.
   \emph{Fact: Let $u$ be any support vector of $S$ at $s$. Then
   $y+u$ lies in the interior of $\fund$ for any  $y\in \fund$.}
  
     By a direct computation we have 
  \begin{equation}  \label{eq:condom}
  \psi(y+u,\gamma_\tau(s))-\psi(y+u, s)=\psi(y, \gamma_\tau(s))-\psi(y,s)+2\langle u, s-\gamma_\tau(s)\rangle_-~.
  \end{equation}
  Notice that the last term is strictly positive since $u$ is a support vector of $S$ at $s$, whereas
  $\psi(y, \gamma_\tau(s))-\psi(y,s)\geq 0$ by the assumption on $y$.
  By \eqref{eq:fundint} it follows that $y+u$ is in the interior of $\fund$.
\qd
}
Let us fix a support vector $u$ of $S$ at $s$.
   Let $R$ be an affine  hyperplane orthogonal to $u$
     chosen so that  $\fconv_{\lambda}\cap \fund$ is in the past side $R^-$ of $R$ 
   for every $\lambda\in [0,1]$.
Note that $D=R\cap \fund$ is compact, as it is part of the boundary of the compact convex set $R^-\cap\Omega_\tau$.

\begin{figure}[h]
\centering
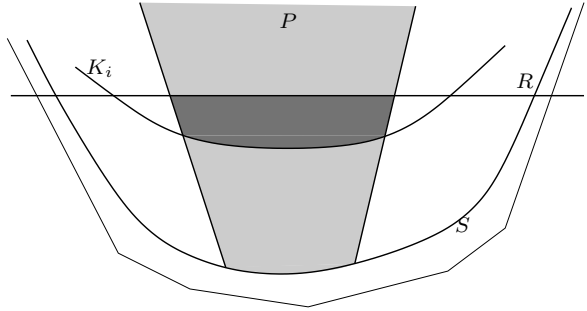
\caption{The light-gray  region represents $\fund$ whereas the dark-grey region represents $C_i$.}\label{fig:ci}
\label{fig:covol}
 \end{figure}

Finally let us denote $C_i=\fconv_i\cap \fund\cap R^-, i=1,2$.
By the choice of $R$, we have that $C_i\cap R=\fund\cap R =D$.
\new{Lemma~\ref{lm: fund} shows that  if $x-\epsilon u\in\fund$ then $x\in\operatorname{int}\fund$.
So for  $x\in R\cap\partial\fund$, there is no $\epsilon>0$ such that
the point $x-\epsilon u$ is  contained
in $\fund$.} It follows that $C_i$ is a convex cap (see Figure~\ref{fig:ci}).

 Define   $C_{\lambda}=(1-\lambda)C_1+\lambda C_2$. 
 For $x\in C_\lambda$, there are $x_i\in C_i$ such that $x$ is a convex combination of
 $x_1$ and $x_2$. By convexity, $x$ belongs to $\fund$ and $R^-$, and by definition $x$ belongs to
 $\fconv_{\lambda}$. So $C_{\lambda}\subset \fconv_{\lambda}\cap \fund\cap R^-$.
 
 If $\fconv_1\not= \fconv_2$, then $C_1$ and   $C_2$, are different caps so 
 Corollary~\ref{cor:vol cap} gives
 \begin{equation}\label{eq: vol comp}
 \begin{array}{ll}&
 (1-\lambda)V(\fconv_1\cap \fund\cap R^-)+\lambda V(\fconv_2\cap \fund\cap R^-)= \\
& (1-\lambda)V(C_1)+\lambda V(C_2)< V(C_{\lambda}) \leq V(\fconv_{\lambda}\cap \fund\cap R^-)~.
\end{array}
 \end{equation}
  By definition, for any  $\lambda\in[0,1]$,
 \begin{equation}\label{eq: vol covol}
 \cov(\fconv_{\lambda})=-V(\fconv_{\lambda}\cap \fund\cap R^-)+c \end{equation}
 where $c$ is a constant depending only on $S, \fund $ and $H$ (actually $c=V(\fund\cap R^-)+\cov(S)$).
  \end{proof}
  
 Since each $\tau$-convex set is limit of Cauchy $\tau$-convex sets, 
and the covolume is continuous,  Proposition \ref{prop: vol conv} implies that the covolume
is convex on the set of $\tau$-convex sets.
Then,  Theorem~\ref{thm: cov con} is proved.

\begin{remark}{\rm
One could also consider the covolume  as a function on the set of all $\tau$-convex sets, for all $\tau\in Z^1(\Gamma,\R^{d+1})$.
This set is a convex cone. But the covolume is certainly not convex on this bigger set.
Otherwise one should have
$$\cov((1-\lambda)\Omega_{\tau}+\lambda \Omega_{\tau'})\leq (1-\lambda)\cov(\Omega_{\tau})+\lambda \cov(\Omega_{\tau'}) $$
that implies that the covolume of $(1-\lambda)\Omega_{\tau}+\lambda \Omega_{\tau'}$ is zero.
But  $(1-\lambda)\Omega_{\tau}+\lambda \Omega_{\tau'}\subset \Omega_{(1-\lambda)\tau + \lambda\tau'}$, and in general the 
inclusion is strict. Then contradiction follows because for any $\tau$, $\Omega_{\tau}$
is the only $\tau$-convex set with zero covolume. 

}\end{remark}

For $\Gamma$-invariant hypersurfaces, there is a representation formula for the covolume, namely
\begin{equation}\label{eq:for cov fuch}\cov(\oh)=-\frac{1}{d+1}\int \oh \d \A(\oh)  \end{equation}
where $\oh$ is the hyperbolic support function of $\fconv$.
We cannot hope such a simple formula holds in  general, but we have the following relation.

\begin{lemma}
If $\fconv_1\subset \fconv_0$ are $\tau$-convex sets with hyperbolic support functions 
$\oh_1 \leq \oh_0$,   then
   \begin{equation}\label{eq: enc diff vol}\int (\oh_0-\oh_1) \d \A(\oh_0) \leq \cov(\oh_1)-\cov(\oh_0)\leq 
   \int (\oh_0-\oh_1) \d \A(\oh_1)~.\end{equation}
\end{lemma}
\begin{proof}
First assume that  $\oh_0$ and $\oh_1$ are $C^2_+$ and  $\oh_1<\oh_0$.
Notice that the region $\fconv_0\setminus \fconv_1$ is foliated by the convex combinations $\fconv_t$ of $\fconv_0$ and $\fconv_1$.
Thus   $\cov(\oh_1)-\cov(\oh_0)$ is the volume of the quotient of the region
that is image of the equivariant map
from $[0,1]\times \H^d$ to $\Omega_\tau$    defined by $\chi(t,\eta)= \grad^{\H}_\eta \oh_t-\oh_t(\eta)\eta$
where $\oh_t=(1-t)\oh_0+t\oh_1$ (compare with \eqref{eq:chi on hyp}).

Notice that $\cov(\oh_1)-\cov(\oh_0)=\int_{[0,1]\times\H^d/\Gamma}\chi^*(\d V)$ where $\d V$ is the $(d+1)$-volume form on
$\R^{d+1}$. Since for each $t$ the map $\chi(t,\bullet)$ is inverse of the Gauss map of  the boundary of $\fconv_t=(1-t)\fconv_0+t\fconv_1$, 
and since the area form of a $C^1$-space-like hypersurface is defined by contracting the volume form by the normal vector, we deduce that

\begin{align*}
\chi^*(\d V)(\partial_t, e_1,\ldots, e_d)=\d V(\chi_*(\partial_t),\chi_*(e_1),\ldots, \chi_*(e_d))=\\
-\langle \chi_*(\partial_t), \eta\rangle_-\d V(\eta, \chi_*(e_1), \ldots, \chi_*(e_d))=\\
-\langle \chi_*(\partial_t), \eta\rangle_-\d\A(\fconv_t)(e_1,\ldots e_d)~.
\end{align*}

In other words,  $\chi^*(\d V)=-\langle \chi_*(\partial_t), \eta\rangle_-\d t\wedge \d\A(\oh_t)$.
Now $$D\chi(\partial_t)=\grad^{\H}(\oh_1-\oh_0)-(\oh_1-\oh_0)\eta$$ so we deduce that  
$\chi^*(\d V)=(\oh_0-\oh_1)\d t\wedge \d\A(\oh_t)$.
Finally integrating,  we obtain
$$\cov(\oh_1)-\cov(\oh_0)=\int_0^1 \int (\oh_0-\oh_1) \d \A(\oh_t)\d t~.$$
As for a convex function $f:[0,1]\rightarrow \R$ 
       $$f'(0)\leq f(1)-f(0)\leq f'(1)~,$$
by Theorem~\ref{thm: cov con} we deduce the formula above.

Now for the general case, take  sequences of $C^2_+$ 
$\tau$-equivariant 
support functions $(\oh_0(n))_n$ and $(\oh_1(n))_n$ converging respectively to $\oh_0$ and $\oh_1$, 
see Appendix~\ref{appen: smooth approx}.
 For large $n$, up to add a suitable constant, we can suppose that  $\oh_1(n)<\oh_0(n)$. 
By continuity of covolume we have that 
$$    \cov(\oh_1(n))-\cov(\oh_0(n)) \underset{n\rightarrow +\infty}{\longrightarrow} \cov(\oh_1)-\cov(\oh_0)~.$$
So in order to prove \eqref{eq: enc diff vol} it is sufficient to prove that
$$    \int  (\oh_0(n)-\oh_1(n)) \d \A(\oh_0(n))\rightarrow \int (\oh_0-\oh_1) \d \A(\oh_0)~,  $$
as well as  the convergence for the analog right hand terms.
As $\oh_0(n)-\oh_1(n)$ uniformly converges to $\oh_0-\oh_1$, by 
Corollary~\ref{cor: weak conv area}, 
$\A(\oh_0(n))$ weakly converge to  $\A(\oh_0)$.
Putting $ f_n=\oh_0(n)-\oh_1(n)$  and $f=\oh_0-\oh_1$ we have
$$\left| \int  f_n \d\A(\oh_0(n) ) -  \int  f \d\A(\oh_0)\right|\leq \|f_n-f\|_\infty \Area(\oh_0(n)) + \left|\int f \d\A(\oh_0(n))-\int f\d \A(\oh_0))\right|~.$$
As $\Area(\oh_0(n))$ are uniformly bounded we get the result.
\end{proof}

\begin{remark}{\rm
With $\oh_1\leq \oh_0$, putting $\oh_t=\oh_0+t(\oh_1-\oh_0)$, 
the previous lemma and the weak convergence of the area measure implies that
 $$\underset{t\rightarrow 0}{\mathrm{lim}}\frac{\cov_{\tau}(\oh_t)-\cov_{\tau}(\oh_0)}{t}=\int(\oh_1-\oh_0)\d\A(\oh_0)~.$$
More generally, adapting the Euclidean argument \cite{car04}, one can prove that if $\oh<\oh_\tau$, the area measure is
the G\^ateaux gradient of the covolume.
}\end{remark}

\begin{remark}{\rm
In the $C^2_+$ case, the area measure is the determinant of a symmetric matrix \eqref{eqref:det hess}.
The determinant of $d\times d$ symmetric matrix can be polarized as a 
$d$-linear form. If moreover we are in the Fuchsian case,
it follows from the representation formula  \eqref{eq:for cov fuch} that
the covolume can be polarized as a $(d+1)$-linear form. Using a density argument,
the covolume can be polarized  on the set of all $\Gamma$-invariant
convex sets. This leads to a theory similar to the mixed-volume theory in the convex bodies case
\cite{Fil12}. There is no such similar theory in the general $\tau$-convex case.
The point is that the set of $\Gamma$ invariant convex set is a (convex) cone, and we can take a convex cone as 
fundamental domain in $\fut(0)$. Actually, 
a mixed-covolume theory can be done
for suitable convex sets in any convex cone --- without mention of an ambient metric, see \cite{KT13}.

Although some questions remain, for example to find an
isoperimetric inequality in the general $\tau$-convex case, see \cite{Fil12} for the
Fuchsian case.
}\end{remark}

\subsection{Proof of the theorems}

We fix a $\Gamma$-invariant positive Radon measure $\mu$ on $\H^d$ and a cocycle $\tau$. The aim of this section is to construct
a $\tau$-convex set with area measure $\mu$.

Consider the functional on the space of  hyperbolic $\tau$-equivariant support functions

$$ L_{\mu}(\oh)=\cov(\oh)-\int(\oh_{\tau}-\oh)\d\mu~.$$

Note that $L_{\mu}$ is continuous (and even convex) by Lemma~\ref{lem: cov cont}.

\begin{lemma}\label{lem:coercitif}
 $L_{\mu}$ is coercive: if $\|\oh_n-\oh_\tau\|_{\infty}\rightarrow +\infty$, then $L_{\mu}(\oh_n)\rightarrow +\infty$.
 \end{lemma}

 \begin{proof}
 By Lemma~\ref{lem: supp et cos time}, 
the maximum of the cosmological time on the corresponding convex sets goes to infinity. 
By Lemma~\ref{lem: max/min}, the minimum of the cosmological time also goes to infinity. 
Let $a_n$ be the minimum of the cosmological time on the convex set with support function $h_n$. Again by \old{Lemma~\ref{lem: supp et cos time} and }Lemma~\ref{lem: max/min},
there is a constant $\delta$ such that
$(\oh_{\tau}-\oh_n)\leq \delta a_n$. So the linear part of $L_{\mu}$ grows (in absolute value) linearly in $a_n$.
On the other hand $\cov(\oh_n)\geq \cov(\oh_{\tau}-a_n)$. 

But  the volume of the past of the level set of the cosmological time coincides with $\overline{V}_{a_n}(\oh_\tau)$ that is
a polynomial of degree $d+1$ by \eqref{eq: muinv}.
It follows that  $\cov(\oh_{\tau}-a_n) \geq c' a_n^{d+1}$.
 \end{proof}

\begin{lemma}
 $L_{\mu}$ attains its  minimum.
\end{lemma}
\begin{proof}
 Let $\alpha\in [-\infty,+\infty)$ be the infimum of the range of $L_{\mu}$.
 Let  $(\oh_n)$ be a minimizing sequence for $L_\mu$ of $\tau$-invariant hyperbolic support functions.
 By Lemma~\ref{lem:coercitif}, the $L^\infty$-norm of $(\oh_n-\oh_\tau)$ are uniformly bounded from above, 
 so from Corollary~\ref{cor:blaschke}  one can extract a  subsequence
 converging to some support function $\oh$. By continuity, $L_{\mu}(\oh)=\alpha$. 
\end{proof}

\begin{lemma}\label{lem: premier estimee aire}
Let $\oh_0$ be a support function on which $L_{\mu}$ attains its minimum. Then
for any $\Gamma$-invariant support function $f$,
$$\int(-f)\d\A(\oh_0)\geq  \int(-f) \d\mu~.$$
\end{lemma}
\begin{proof}
If $f=0$, the result is obvious. If $f\not=0$, we know that $f<0$.
Let $\oh_a=\oh_0+af$  for $a\in[0,1]$. As 
$L_{\mu}$ attains its minimum at $\oh_0$, 
$$0\leq L_{\mu}(\oh_a)-L_{\mu}(\oh_0)=\cov(\oh_a)-\cov(\oh_0)+\int af \d\mu~. $$
As $h_a<h_0$, we can use \eqref{eq: enc diff vol}, which gives
$$0\leq \int (-af) \d\A(\oh_a) + \int a f\d\mu$$
so we get
$$\int (-f) \d\A(\oh_a)\geq \int (-f)\d\mu$$
for any $a$. The result follows by letting $a\rightarrow 0$ and using the weak convergence of the area measures (Corollary~\ref{cor: weak conv area}).
\end{proof}

\begin{lemma}\label{lem: egalite cst} With the notations of Lemma~\ref{lem: premier estimee aire},
 $$ \int \d \A(\oh_0)= \int \d\mu~. $$
\end{lemma}
\begin{proof}
From Lemma~\ref{lem: premier estimee aire}, with $f=-1$, we get   $$ \Area(h_0)\geq \int \d\mu~.$$ 
We have to prove the other inequality.

For a positive $a$ consider the $\tau$-equivariant function $\oh_0+a$. Notice that in general
it is not a support function. However we consider  the domain
  $$\fconv_a=\{p\in \R^{d+1}|  \langle p, \eta\rangle_- \leq \oh_0(\eta)+a, \forall \eta\in \H^d\}~,$$
and let $\oh_a$ be the support function of $\fconv_a$. 
\new{As $\fconv\subset\fconv_a$  we have  $\oh_0\leq \oh_a$.
On the other hand by the definition of support function  $\oh_a \leq \oh_0+a$.}
From \eqref{eq: enc diff vol} we get 

$$\int (\oh_a-\oh_0)\d\A(\oh_a) \leq  \cov(\oh_0)-\cov(\oh_a)\leq \int (\oh_a-\oh_0) \d\A(\oh_0)$$
or in a more convenient way

\begin{equation}\label{eq: dble ineq}
- \int (\oh_a-\oh_0) \d\A(\oh_a) \geq \cov(\oh_a) -\cov(\oh_0) \geq   - \int(\oh_a-\oh_0) \d\A(\oh_0)~. 
\end{equation}

As $L_{\mu}(\oh_a)\geq L_{\mu}(\oh_0)$,

 $$   \cov(\oh_a)-\cov(\oh_0) +\int (\oh_a-\oh_0)\d\mu\geq 0~,$$

and using \eqref{eq: dble ineq} 

$$-\int (\oh_a-\oh_0)\d\A(\oh_a) +\int(\oh_a-\oh_0)\d\mu \geq 0$$

that is

$$\int(\oh_a-\oh_0)\d\mu \geq \int (\oh_a-\oh_0)\d\A(\oh_a)~.$$

Now using that $ a \geq \oh_a-\oh_0$ we get
   
$$    a\int \d\mu\geq \int (\oh_a-\oh_0) \d\A(\oh_a)~.$$

From Lemma~\ref{lem:pte pt reg}, $\oh_a$ coincides with $\oh_0+a$  almost everywhere for $\A(\oh_a)$, hence
$$     \int \d\mu\geq \Area(\oh_a)~.$$
Letting $a \rightarrow 0$ and using the  convergence of the total area, we get the estimate we need.

\end{proof}

\begin{remark}\label{rem: bakelman}{\rm
 In the Fuchsian case, using the representation formula \eqref{eq:for cov fuch} for the covolume, the functional 
 $L_{\mu}$ has a form very similar to the functional used in the variational approach for the Monge--Amp\`ere equation, see
 \cite[4.5]{TW08} and the references therein.
}\end{remark}

\begin{proof}[Proof of Theorem~\ref{thm:main1}]

Let $f$ be a $C^2$ $\Gamma$-invariant function on $\H^d$. 
There is a constant $c>0$ such that $\nabla^2f -(f-c)g_{\H}\geq 0$.
In particular $f-c$ is a $\Gamma$-invariant support function, see \eqref{eq:twdff}.
By Lemma~\ref{lem: premier estimee aire}:
$$   \int (f-c)\d\mu\geq \int (f-c)\d\A(h_0)~.$$
From Lemma~\ref{lem: egalite cst}, we know that the equality holds for constants. Then we can conclude that
$$\int f\d\mu\geq\int f\d\A(h_0)$$
for any $C^2$-function. Considering $-f$,  for any $C^2$-function 
we get 
$$\int f\d\mu=\int f\d\A(h_0)~.$$
By approximation we obtain that this equality is true for any $C^0$ function, hence $\mu=\A(h_0)$, and the theorem is proved.
\end{proof}

Using the correspondence between the Area measure of $\fconv$ and the Monge--Amp\`ere measure
$\M(h)$ described in Proposition~\ref{prop: area=ma}, we can  prove Theorem~\ref{thm:main2} stated in the introduction.

\begin{proof}[Proof of Theorem~\ref{thm:main2}]
Notice that for each $\gamma$ there is a vector $s_\gamma\in \R^d$ and $a_\gamma\in \R$ such that
\[
(\gamma\cdot h_0)(x)-h_0(x)=\langle s_\gamma, x\rangle-a_{\gamma}~=\left\langle \hat{x}, \binom{s_\gamma}{a_\gamma}\right\rangle_-.
\]
In particular letting $H_0$ be the $1$-homogeneous extension of $h$ to $\fut(0)$ we conclude that
\[
   H_0(\gamma^{-1}\hat{x})-H_0(\hat{x})=\left\langle \hat{x}, \binom{s_\gamma}{a_\gamma}\right\rangle_-~.
\]
Putting $\tau_\gamma:=-\binom{s_{\gamma}}{a_{\gamma}}$ the formula above easily shows that
$\tau$ is a cocycle and the equivariance
\[
    H_0(\gamma Y)-H_0(Y)=-\langle \tau_{\gamma^{-1}}, Y\rangle_- 
\]
holds. By \eqref{eq:act}  $H_0$ is a $\tau$-equivariant convex function.

So by Theorem~\ref{thm:main1}, there is a $\tau$-equivariant convex function $h$ on $\ball$ such that 
$$\A(h)=\sqrt{1-\|x\|^2}\mu~.$$
By Proposition~\ref{prop: area=ma} 
we have $\M(h)=\mu$. As $h$ is $\tau$-equivariant as  $h_0$, then on the boundary $h=h_0$.

\end{proof}

Finally let us consider the regularity problem.

 From Proposition~\ref{prop: c tau} there is a constant $c=c(\tau)\geq 0$ such that if $\mu>c \d\H^d$, then 
 the $\tau$-convex set $\fconv$ with $\A(\fconv)=\mu$ is contained in the interior of $\Omega_\tau$.

 By Corollary~\ref{cor: MA regularite} we get:

\begin{theorem}\label{thm:reg equi}
Let $f$ be  a $\Gamma$-invariant $C^{k+1}$-function, $k\geq 2$, with $f(x)>c(\tau)$ for every $x\in\H^d$.
Then the boundary of the $\tau$-convex set $\fconv$ with $\A(\fconv)=f\d\H^d$  is
a strictly convex space-like hypersurface of class $C^{k+2}$ with Gauss--Kronecker curvature $f^{-1}$.
\end{theorem}
\begin{proof}
Corollary~\ref{cor: MA regularite} immediately gives that $\partial_s\fconv$ is strictly convex $C^{k+2}$ hypersurface.
Notice that its Gauss--Kronecker curvature is $f^{-1}>0$ so the Gauss map is a $C^{k+1}$-homeomorphism with inverse map of class $C^{k+1}$.
Moreover, $\partial \fconv=\partial_s \fconv$ by Lemma~\ref{lem: cauchy at int tau}.
 \end{proof}

 Theorem~\ref{thm GH1} and Theorem~\ref{thm GH2} are just quotient versions of preceding results.

%%%%%%%%%%%%%%%%%%%%%%%%%%%%%%%%%%%%%%%%%%%%%%%%%%%%%%%%%%%%%%%%%%%%%%%

\subsection{Equivariant Pogorelov example}\label{sub: pogo}

\new{
In this section we fix an uniform lattice $\Gamma$ and  a cocycle $\tau\in Z^1(\Gamma, \mink)$ 
such that $\partial_s\Omega_{\tau}$
 %is simplicial and  
contains a $(d-1)$ Euclidean polyhedron $\poly$. }%(such examples are constructed 
%for instance in \cite{BMS13}). 

\old{The goal of this section is to prove that} 
\new{We will prove that }there exists a $\tau$-convex set with positive area measure
(in the sense that $\A(\fconv)\geq c_0\d\H^d$ for some $c_0$)
which meets the boundary of $\Omega_\tau$.

\begin{remark}\rm{
\old{
Recall by \cite{Bon05} that $\partial_s\Omega_\tau$ is simplicial if the corresponding stratification
in $\H^d$ is locally finite. In dimension $d=3$ if $\partial_s\Omega_\tau$ is simplicial, then either
it is a tree (that is the space is simple), or it must contain a $2$-polygon.
Examples of simplicial space-times which are not simple are for instance constructed in \cite{BMS13}.
This shows the existence of cocyles satisfying the condition we consider in this section.
}
\new{
In dimension $d=3$ examples of pairs $(\Gamma,\tau)$ such that $\partial_s\Omega_\tau$ contains a 
$2$-polygon were constructed in \cite{BMS13}.
}
%The geometry of $\partial_s\Omega$ in the non-simplicial case is still not understood if $d>2$.
%The only example the authors know in dimension $d=3$ is described in \cite{BMS13}.
%We remark that in this example the boundary of $\partial_s\Omega$ contains a $2$-polygon so the construction
%described in this section works even in this case.

It remains open to understand whether there is a lattice $\Gamma<\operatorname{SO}^+(3,1)$ and a cocycle $\tau$ such that
the corresponding domain is not simple but $c(\tau)=0$.}
\end{remark}

In practice we will construct a $\tau$-support function $h$ on $\ball$ with $\M(h)>c_0\mathcal L$, \new{ and $h=h_\tau$ on a segment}.
\old{Considering some $r<1$ 
such that a fundamental domain for the projective action of $\Gamma$ on $\ball$ is contained in the ball of radius $r$, 
 from Corollary~\ref{cor: relation area ma} there will be a constant $c'$ such that   $\A(\oh)\geq c'\d\H^d$.}
\new{ The latter condition implies that $\partial_s\Omega_\tau$ meets the $\tau$-convex set supported by $h$.
 On the other hand, the former condition and Corollary~\ref{cor: relation area ma} show that on a compact fundamental domain  
 $\A(\oh)\geq c'\d\H^d$, for some constant $c'$. As the area measure is $\Gamma$ invariant  we will conclude that 
 $\A(\oh)\geq c'\d\H^d$ everywhere.
 }
 
 Let $f_0$ be a function  given in Remark~\ref{rk:ex pogo} for $k=d-1$:
  $$f_0(x)=\beta^2(1+\beta^2 x_1^2)r^{\alpha}(x)~.$$
where $r(x)=\sqrt{x_2^2+\cdots +x_d^2}$, $\alpha=2-2/d$ and $\beta$ is a suitable constant $\geq 1$.
We have
 $f_0\equiv 0$ on  $[\ell_-,\ell_+]$, with $\ell_\pm=(\pm 1,0,\ldots,0).$

Up to compose by a global isometry, we can suppose that the polyhedron
 $\poly\subset\partial_s\Omega_\tau$ contains $0$ in its interior and is contained in the $(d-1)$-linear space
$$R=\{x\in\R^{d+1}| x_{d+1}=x_1=0\}~.$$

Recall that $H_\tau$ is the support function of $\Omega_\tau$ on $\fut(0)$ and 
$h_\tau$ is  its restriction to the ball $\ball$. 
We notice that $h_{\tau}\geq 0$ and $h_\tau=0$ exactly on $[\ell_-,\ell_+]$.

\begin{lemma}
There is $\epsilon>0$ such that $\epsilon f_0<  h_\tau$ on $\overline \ball\setminus [\ell_-,\ell_+]$.
\end{lemma}
\begin{proof}
By the definition of $ f_0$ we have that $f_0(x)\leq cr(x)$ for any $x\in\overline \ball$ and some $c$.
So it is sufficient to show that there is $M$ such that $ h_\tau(x)> M r(x)$ on $\overline \ball\setminus [\ell_-,\ell_+]$. 

Let $\Omega=\fut(\poly)$. It is an F-convex set contained in $\Omega_\tau$
with support function  $h$ on $\ball$. Note that $h=0$ on $[\ell_-,\ell_+]$. So to prove the claim it is sufficient to check that
for $x\in \ball\setminus [\ell_-,\ell_+]$ we have $h(x)> c r(x)$ for some $c$.

But in $R$, $0$ is contained in the interior of the convex body $\poly$,  so $\poly$ contains in its interior a small ball
centered at $0$, whose support function is $c\|\cdot\|_R$  some $c$, where $\|\cdot\|_R$
is the norm associated to $\langle\cdot,\cdot\rangle_R$, the scalar
product on $R$ induced by the Minkowski product. For any $x'\in R\setminus\{0\}$
$$\operatorname{max}_{y\in \poly}\langle x',y\rangle_R > c\|x'\|_R~.$$
But for any $y\in \poly$, $\langle y,e_1\rangle_-=\langle y,e_{d+1}\rangle_-=0$, so for any $x=(x_1,x')\in \ball$ 
$$h(x)=\operatorname{max}_{y\in \poly}\left\langle \hat{x},y\right\rangle_-= \operatorname{max}_{y\in \poly}\langle x',y\rangle_R  $$
and $\|x'\|=r(x)$. 
\end{proof}

Now we take $ f=\epsilon f_0$. Notice that
\begin{itemize}
\item $\M( f)\stackrel{\eqref{eq:MA homothetie}}{=} \epsilon^{d+1}\M(f_0)>\epsilon^{d+1}c\L$;
\item $ f=0$ on the segment $[\ell_-,\ell_+]$;
\item $ f<  h_\tau$ on $\overline \ball\setminus [\ell_-,\ell_+]$.
\end{itemize}

We then define
\[
  f_\infty(x)=\sup_{\gamma\in\Gamma_\tau}(\gamma\cdot f)(x)
\]
where $\Gamma_\tau$ acts on the set of convex functions of $\ball$ as explained in Section \ref{sub def tau}.

By definition $f_\infty $ is a convex function on $\ball$. Actually it is the support function of the convex hull
of  the union of all the $\gamma\cdot \fconv(f)$, so it is a fixed point for the action of $\Gamma_\tau$.
Moreover, since $ f\leq  h_\tau$, we have that $ f_\infty(x)\leq  h_\tau(x)$ for any $x\in\overline{\ball}$,
 and $ f_\infty= h_\tau=0$ on the segment $[\ell_-,\ell_+]$

The rest of this section is devoted to prove the following proposition.

\begin{proposition} \label{pr:cex}
For any compact subset $C$ of $\ball$, there is $c_C>0$ such that
\[
M(f_\infty)\geq c_C\L
\]
on $C$.
\end{proposition}

The first step is to compute the stabilizer of $\poly$ in $\Gamma_\tau$.

\begin{lemma}\label{lm:stab}
 The stabilizer of $\poly$ in $\Gamma_\tau$ is
a $\mathbb Z$-group generated by an element, say $\mu_\tau$.
\end{lemma}
\begin{proof}
First we prove that the stabilizer of $\poly$ is not trivial.
Notice that for each $\gamma_\tau\in\Gamma_\tau$ either 
$\gamma_\tau(\poly)=\poly$ or $\gamma_\tau(\operatorname{int} \poly)$ is disjoint from
$\operatorname{int} \poly$.

If $\poly_1$ denotes the subset of the level set $\Sigma_1$ of the cosmological time that retracts in $\operatorname{int} \poly$,
by the equivariance of the retraction  we have that either $\gamma_\tau(\poly_1)=\poly_1$ or they are disjoint.

By \cite{Bon05} $\poly_1$ is isometric to the Cartesian product $(\operatorname{int} \poly)\times\R$, so its volume is infinite.
If the stabilizer of $\poly$ were trivial the projection $\poly_1\to \Sigma_1/\Gamma_\tau$ would be an injective
isometric immersion of a Riemann manifold of infinite volume into a compact manifold, and this is impossible.

Now if $\gamma_\tau$ is in the stabilizer of $\poly$, then $\gamma$ stabilizes the image of $\poly$ through the Gauss map
in $\H^d$. By \cite{Bon05} this is the geodesic in $\H^d$ orthogonal to the hyperplane $R$ containing $\poly$.
Notice that the stabilizer of a geodesic in a uniform lattice can be either trivial or a cyclic group. Since we have proved that
it is not trivial, the Lemma follows.
\end{proof}

By the proof of Lemma \ref{lm:stab}, the projective transformation 
 $\bar \mu$ stabilizes the segment $[\ell_-, \ell_+]$.  
Moreover  $\mu_\tau$ acts as an isometry on $\poly$,  so it fixes a point of $\poly$.
Up to a translation we may suppose that $\mu_\tau$ fixes $0$, that is, $\tau_\mu=0$ and $\mu_\tau=\mu$.

Consider the function
\[
    f_1(x)=\sup_{n\in\mathbb Z} (\mu^n\cdot  f)(x)~.
\]
It is $<\mu>$-invariant by construction.

\begin{lemma}\label{lm:delta}
For any $\delta>0$, let $\left. C=\left\{z\in \overline \ball\right| |x_1(z)|<1-\delta\right\}$.
Then  there are a finite number of $n_1\ldots n_k\in\mathbb Z$ such that
\[
    f_1(z)=\max_{i}(\mu^{n_i}\cdot f)(z)
\]
for any $z\in C$.
\end{lemma}
\begin{proof}
The transformation $\mu$ leaves  $R$ and $\poly$ invariant (in particular
$\mu|_R$ is an orthogonal transformation of finite order).
On the other hand we have  

\[\begin{split}
\mu^n(1,0,\ldots, 0)=(\ch na,  0,\ldots,0,\sh na)\,,\\
\mu^n_{d+1}(0,0\ldots,0,1)=(\sh na,0,\ldots, 0, \ch na)
\end{split}\]
 for any $n\in\mathbb Z$.

Since $\tau_{\mu}=0$,  by Lemma~\ref{lem: projective action},
\[
  (\mu^n\cdot  f)(x)=(\mu^{-n}\hat{x})_{d+1} f\left (\frac{\mu^{-n}(\hat{x})}{(\mu^{-n}\hat{x})_{d+1}}\right)
\]
so a simple computation shows that
\[
  (\mu^n\cdot  f)(x)=\frac{\epsilon \beta^2 r^\alpha(x)}{(\ch an- x_1\sh an)^{\alpha-1}}\left(1+\beta^2\left(\frac{x_1-\tanh an}
  {1-x_1\tanh an}\right)^2\right).
\]

Since $x_1\in[-1,1]$  we have that $\frac{x_1-\tanh an}{1-x_1\tanh an}\in[-1,1]$. 

By the definition of  $C$,  we have $|x_1|\leq 1-\delta$ for every $x\in C$.
So we easily see that 
\[
\frac{1}{(\ch an- x_1\sh an)^{\alpha-1}}\left(1+\beta^2\left(\frac{x_1-\tanh an}{1-x_1\tanh an}\right)^2\right)\rightarrow 0
\]
uniformly on $C$ both for $n\rightarrow+\infty$ and for $n\rightarrow -\infty$.

It follows that, putting $a_0=\min f|_{\new{C}}$,
 there is a finite list $I=\{n_1\ldots, n_k\}$ such that for $n\notin I$
\[
\frac{1}{(\ch an- x_1\sh an)^{\alpha-1}}\left(1+\beta^2\left(\frac{x_1-\tanh an}{1-x_1\tanh an}\right)^2\right)\leq a_0
\]
for every $x\in C$.

So for $x\in C$
\[
   f_1(x)=\max_{i}(\mu^{n_i} f)(x)~.
\]
\end{proof}

The previous lemma, Corollary \ref{cor:mabound} and formula \eqref{eq: MA max}  imply the following:

\begin{corollary}\label{co ma bound fini}
For any compact subset $C$ in  $\ball$ there is a constant $c_0>0$ such that
$\M( f_1)>c_0\L$ on $C$.
\end{corollary}

Another simple consequence of Lemma \ref{lm:delta} is the following

\begin{corollary} \label{cor:less}
For any $\xi\in \partial \ball\setminus\{\ell_\pm\}$ we have $ f_1(\xi)< h_\tau(\xi)$ .
\end{corollary}
\begin{proof}
By Lemma \ref{lm:delta} there exists $n\in\mathbb Z$ such that
$ f_1(\xi)=(\mu^n\cdot  f)(\xi)$. Now, as $\bar \mu^{-n}(\xi)\notin\{\ell_\pm\}$ (because 
$\bar \mu$ stabilizes $[\ell_-,\ell_+]$) we have that $(\mu^n\cdot  f)(\xi)<(\mu^n\cdot h_{\tau}(\xi))$,
 and then  
 $$ f_1(\xi)<(\mu^n\cdot  h_\tau)(\xi)= h_\tau(\xi)~,$$
 where the last equality holds since $ h_\tau$ is $\Gamma_\tau$-invariant. 
\end{proof}

Notice now that the function $ f_\infty$ can be  obtained as
\[
  f_\infty(x)=\sup_{\gamma_\tau\in\Gamma_\tau}(\gamma_\tau\cdot  f_1)~.
\]
Since $f_1$ is $\mu$-invariant then 
$\gamma_\tau\cdot f_1=(\gamma_\tau\mu^n)\cdot f_1$ so the function
$\gamma_\tau\cdot  f_1$ depends only on the coset of $\gamma$ in the set of cosets $X=\Gamma/<\mu>$.
So putting $f_{[\gamma]}=\gamma_\tau\cdot f$ we have that
\[
   f_\infty(x)=\sup_{[\gamma]\in X}( f_{[\gamma]})~.
\]

Proposition \ref{pr:cex} follows then  from the following lemma, together with 
Corollary~\ref{co ma bound fini}, Corollary \ref{cor:mabound}, and  \eqref{eq: MA max} imply the following:

\begin{lemma}
For any compact subset $C$ in  $\ball$ there exist a finite number of cosets $[\gamma_1],\ldots,[\gamma_k]$ such that
on $C$
\[
     f_\infty(x)=\max_{i} f_{[\gamma_i]}~.
    \]
\end{lemma}
\begin{proof}
We will prove that for any infinite sequence $[\gamma_n]$ of distinct elements in $X$ we have
that  $f_{[\gamma_n]}$ converges uniformly to $-\infty$ on $C$.

Suppose by contradiction that there exists a sequence $[\gamma_n]$ such that 
$\min f_{[\gamma_n]}$ is bounded on $C$.
As $\Gamma$ is cocompact, its elements are hyperbolic. So,
up to pass to a subsequence  there is a point $\ell\in \partial \ball$ such that
$\bar \gamma_n^{-1}(x)$ converges uniformly to $\ell$ on $C$.

\emph{Fact: we may choose $\gamma_n$ in their cosets so that $\ell\notin\{\ell_-, \ell_+\}$.}

Take the disk $\Delta=\{x\in \ball|x_1=0\}$. Its $<\bar\mu>$-orbit  is a union of disjoint disks $\Delta_i=\bar\mu^i(\Delta)$
which converge to $\ell_+$ (for $n\rightarrow+\infty$) and to $\ell_-$ (for $n\rightarrow -\infty$) and whose
hyperbolic distance is $a$.

In particular, for each $n$ we have that $\bar \gamma_n^{-1}(C)$ is contained between $\Delta_{k_n}$ and $\Delta_{k_n+h}$ where
$h$ is any integer number less than the hyperbolic diameter of $C$ divided by $a$.
It follows that $\bar\mu^{-k_n}\bar\gamma_n^{-1}$ sends $C$ into the region bounded by $\Delta_1$ and $\Delta_h$.

In particular changing $\gamma_n$ with $\gamma_n\mu^{k_n}$ we have that the point $\ell$ is contained in the region
annulus on $\partial \ball$ bounded $\Delta_1$ and $\Delta_h$, so it is neither $\ell_-$ nor $\ell_+$. The fact is proved.
\qd

Now by Lemma \ref{lem: projective action} we have
\[
  (\gamma_n\cdot  f_1)(x)=(\gamma_n^{-1}(\hat{x}))_{d+1}( f_1(\bar \gamma_n^{-1}x))+\langle \hat{x},  \tau_{\gamma_n}\rangle_-~.
\]
Since $ h_\tau$ is invariant by the action of $\Gamma_\tau$ we deduce that 
\[
  -\langle \hat{x}, \tau_{\gamma_n}\rangle_-= (\gamma_n^{-1}(\hat{x}))_{d+1}( h_\tau( \bar \gamma_n^{-1}\hat{x}))-  h_\tau(x)
\]
so
\[
  (\gamma_n\cdot  f_1)(x)=(\gamma_n^{-1}(\hat{x}))_{d+1}( f_1(\bar \gamma_n^{-1}x)- h_\tau(\bar \gamma_n^{-1}x))+ h_\tau(x)~.
\]
On the other hand since $\bar \gamma_n^{-1}x\rightarrow\ell$ uniformly on $C$, and 
$ f_1(\new{\ell})< h_{\tau}(\ell)$ (see Corollary~\ref{cor:less})
there is $\epsilon$ such that $ f_1(\bar \gamma_n^{-1}(x))- h_{\tau}(\bar \gamma_n^{-1}(x))\leq -\epsilon$ for $x\in C$ and for $n$ big.

Notice that $(\gamma_n^{-1}(\hat{x}))_{d+1}\rightarrow+\infty$ on $C$ so 
\[
  (\gamma_n\cdot  f_1)(x)\rightarrow-\infty
\]
uniformly on $C$ contradicts the assumption.
\end{proof}

Now let $C$ be a compact set of $\H^d$ containing a fundamental domain for the action of $\Gamma$. 
From Proposition~\ref{pr:cex} and Proposition~\ref{prop: area=ma}, on $C$ the area measure of 
$\fconv(f_{\infty})$ is bounded from below by a positive constant times the hyperbolic volume measure. 
As the area measure is invariant under the action of $\Gamma$, it follows that on $\H^d$  the area measure of 
$\fconv(f_{\infty})$ is bounded from below by a positive constant times the hyperbolic volume measure. But
$\fconv(f_{\infty})$ touches the boundary of $\Omega_{\tau}$, because $f_{\infty}=h_{\tau}$ on $[\ell_-,\ell_+]$.

\begin{corollary}\label{cor:count}
Let $\fconv_\infty$ be the $\tau$-convex set corresponding to $f_\infty$. 
There is a constant $c_0$ such that $\A(\fconv_\infty)\geq c_0\d\H^d$, but $\fconv_\infty$ meets $\partial_s\Omega_{\tau}$ and in particular it
is not strictly convex.
\end{corollary}

As a corollary we get

\begin{corollary}
\new{If $\partial_s\Omega_\tau$ contains a $(d-1)$-dimensional polyhedron $\poly$, then}
there is no $\tau$-convex set with  smooth boundary and with  constant curvature function $c_0$,
\new{where $c_0$ is the constant in Corollary \ref{cor:count}} .
\end{corollary}
\begin{proof}
Suppose by contradiction  that such set, say $\fconv$, exists and let $h$ its support function on $\ball$.
By Proposition \ref{prop: area=ma} and \eqref{eq:density vol ball} we should have
\[
   \M(h)=c_0(1-\|x\|^2)^{(-d-2)/2}\mathcal L~.
\]
On the other hand by Corollary \ref{cor:count}  we have
\[
   \M(f_\infty)\geq c_0(1-\|x\|^2)^{(-d-2)/2}\mathcal L~.
\]

By the comparison principle  (Theorem~\ref{thm: uniq MA})
\[
   h_\tau \geq h\geq f_\infty\geq 0~.
\]

This implies that $h$ coincides with $h_\tau$ on $[\ell_-, \ell_+]$, so 
$\partial \fconv$ contains the points of $\partial\Omega_\tau$ which satisfy either
 $\langle p, \new{\ell_-}\rangle_-=0$ or $\langle p, \new{\ell_+}\rangle_-=0$.
In particular for each point $p$ in the interior of $\poly$, the two light-like
vectors through $p$ directed as $\new{\ell_-}$ and $\new{\ell_+}$ are contained in
$\partial \fconv$.

Thus the intersection of $\partial \fconv$ with the time-like $2$-plane $U$ passing through $x$ 
and orthogonal to the plane $R$ containing $\poly$
is the union of two geodesic rays with different directions. Since this intersection is transverse, $\partial \fconv$ 
cannot be smooth.
 
\end{proof}

%%%%%%%%%%%%%%%%%%%%%%%%%%%%%%%%%%%%%%%%%%%%%%%%%%%%%%%%%%%%%%%%%%%%%%%%%%%%%%%%%%%%%%%%%%%%%%%%%%%%%%%%%%%%%%%%%%%%%%%%%%%%%%%%%%%%%%%%%%%%%%%%%%
%%%%%%%%%%%%%%%%%%%%%%%%%%%%%%%%%%%%%%%%%%%%%%%%%%%%%%%%%%%%%%%%%%%%%%%%%%%%%%%%%%%%%%%%%%%%%%%%%%%%%%%%%%%%%%%%%%%%%%%%%%%%%%%%%%%%%%%%%%%%%%%%%%%

\appendix

\section{Smooth approximation}\label{appen: smooth approx}

In this section we will prove that any $\tau$-convex set $\fconv$ can be approximated by a sequence
$(\fconv_n)$ of $C^2_+$ $\tau$-convex sets.  We limit ourself to consider $C^2$-approximation only for
simplicity. In fact the same techniques can be implemented to produce a smooth approximation.
In Remark~\ref{rk:highorder} some details will be given in this direction.

The main idea is to use an average procedure to construct a sequence $\oh_n$ of $C^2$-
support functions converging to $\oh$. 
Basically, given any $r>0$ we define by the hyperbolic average of any continuous function 
$\oh$ of radius $r$ the following function
\[
 \hat h_r(x)=\displaystyle \frac{d}{\omega_{d-1}(\sh r)^d} \int_{B(x,r)}\oh(y)\d\mathbb H^d(y)~,
\]
where $B(x,r)$ is the hyperbolic ball centered at $x$ of a radius $r$ and $\omega_{d-1}$ is the volume
of the standard sphere $S^{d-1}$ of constant curvature $1$.
For brevity we denote by  $\const(r)=\displaystyle \frac{d}{\omega_{d-1}(\sh r)^d}$.
Note that  $\const(r)V(r)\rightarrow 1$ when $r\rightarrow 0$, where $V(r)$ is the volume of the hyperbolic ball
of radius $r$ \cite[3.4]{Rat06}.

We will show the following facts
\begin{itemize}
\item $\hat h_r\rightarrow \oh$  as $r\rightarrow 0$ uniformly on compact subsets of $\H^d$. 
\item  if $\oh$ is $\tau$-equivariant, so is $\hat h_r$.
\item  $\hat h_r$ is of class $C^1$. More generally, 
if $\oh$ is of class $C^k$, then $\hat h_r$ is of class $C^{k+1}$.
\end{itemize}

A further difficulty is that the average procedure in general does not produce support functions.
However  if $\oh$ is $\tau$-equivariant, we will prove that there exists a constant $C>0$ such that $\hat h_r-Cr$
is a support function for any $r$. 
To prove this result  it will suffice the following property about $\hat h_r$.

\begin{itemize}
\item If $\oh$ is a support function,
for any convex compact domain $K$ of $\mathbb H^d$, there is a constant
$C=C(\|\oh\|_{\operatorname{Lip}(K)})$ such that $\hat h_r-Cr$ is a support function on $K$.
\end{itemize}

(Here being a support function on $K\subset\H^d$ means that its $1$-homogeneous extension on the cone over
$K$ is convex.)

First let us show how the enlisted properties of the hyperbolic 
average  imply the density statement we want to prove in this section.

\begin{theorem}
 For a given  $\tau$-support function $\oh$ and for any $\delta>0$, there is 
 a $C^2_+$ $\tau$-support function $\oh_\delta$ such that $\|\oh-\oh_\delta\|_{L^{\infty}}\leq \delta$.
 \end{theorem}

\begin{proof}
Let $K$ be a convex domain such that the $\Gamma$-translates  of the interior of $K$ cover $\H^d$, and $C$ be the constant
appearing in the last property.
If $r$ is sufficiently small, $\|\hat h_r-Cr-\oh\|_K<\delta/3$.
Notice that the $1$-homogeneous extension of  $\hat h_r-Cr$ is convex in the cone over $K$. Since $\hat h_r-Cr$ is $\tau$-equivariant,
it turns out that $\oh'=\hat h_r-Cr$ is a $C^1$ $\tau$-support function over $\H^d$. Moreover, since 
$\oh'-\oh=\hat h_r-Cr-\oh$ is $\Gamma$-invariant we get that $\|\oh'-\oh\|_{\H^d}\leq\delta/3$.

Repeating the average procedure on $\oh'$ we obtain  a $C^2$ $\tau$-support function \newnew{$\oh''$} whose distance from $\oh'$ is less than $\delta/3$.
Notice that $\nabla^2 (\oh'')-\oh'' g_{\H}$ is semidefinite positive by Equation \eqref{eq:hessh}, and $\|\oh-\oh''\|<2\delta/3$, 
so  $\oh_\delta:=\oh''-\delta/3$ is a $C^2_+$ $\tau$-support function whose distance from $\oh$ is less than $\delta$.
\end{proof}

Let us now prove the properties of this average process.

The first property is a direct consequence of the next proposition and the fact
that the function $\const(r)$ behaves as the inverse of the volume $V(r)$ of the hyperbolic ball
of radius $r$ for small $r$.

\begin{proposition}
Let $\oh$ be a continuous function of $\mathbb H^d$.
Given a compact subset $K$, denote by $$\delta(r,K)=\sup\{|\oh(x)-\oh(y)|, x\in K, y\in B(x,r)\}~.$$
Then $\|\hat h_r- \oh\|_{K}\leq \left|\const(r)V(r)-1\right|\|\oh\|_{K_r}+\delta(r,K)$.%, where $\|\cdot\|_A$ is the sup norm on $A$.
\end{proposition}
\begin{proof}
The estimate is obtained by a direct computation. For $x\in K$:
$$
\left|\hat h_r(x)-\oh(x)\right|=$$ $$\left|\const(r)\int_{B(x,r)} \oh(y)\d\mathbb H^d(y)-\frac{1}{V(r)}\int_{B(x,r)} \oh(y)\d\mathbb H^d(y)+
\frac{1}{V(r)}\int_{B(x,r)} (\oh(y)-\oh(x))\d\mathbb H^d(y)\right|$$
$$\leq
\left|\const(r)V(r)-1\right|\|\oh\|_{K_r}+\delta(r,K)~.
$$
\end{proof}

We now show that if $\oh$ is $\tau$-equivariant so is $\hat h_r$.
Notice that if $\gamma\in\Gamma$ then 
\[
  \hat h_r(\gamma(x))-\hat h_r(x)=\const(r)\left(\int_{B(\gamma(x),r)}\oh(y) \d\mathbb H^d(y)-
  \int_{B(x,r)}\oh(y)\d\mathbb H^d(y)\right)\]�\[=\const(r)\left(\int_{B(x,r)}(\oh(\gamma y)-\oh(y))\d\mathbb H^d(y)\right)~.
 \]
 By \eqref{eq:inv} we have that $\bar f(y)=\oh(\gamma y)-\oh(y)$ is the restriction of the linear function
 $$\bar f(x)=\langle x, \gamma^{-1}\tau_{\gamma}\rangle_-~,$$ and  we have that 
 $\hat h_r(\gamma x)-\hat h_r(x)=\hat f_r(x)$.

Now the $\tau$-equivariance of $\hat h_r$ is ensured by this remark and the following
Lemma.

\begin{lemma}
If $\bar f$ is the restriction of a linear function on $\mathbb H^d$, then $\hat f_r=\bar f$ for every $r$.
\end{lemma}
\begin{proof}
Since the $1$-homogeneous extension, $F$, of  $\bar f$ is linear, we have that its Minkowski gradient is constant.
By the computation in  Lemma~\ref{lm:c2+}  the derivative of $F$ at a point $\H^d$ coincides with
$\nabla \grad^{\H}(\bar f)-\bar f\operatorname{Id}$, so we conclude that
$\nabla^2\bar f-\bar f g_{\H}=0$ on $\H^d$.  

In particular we have $\Delta \bar f=d\cdot \bar f$. So we deduce from the divergence theorem that
\[
\hat f_r(x)=\frac{d}{\omega_{d-1}\sh(r)^{d}}\int_{B(x,r)}\frac{1}{d}\Delta \bar f(y) \d\mathbb H^d(y)=
\frac{1}{\omega_{d-1}\sh(r)^{d}}\int_{S(x,r)} D\bar f(y)[\nu] \d S_r(y)
\]
where $S(x,r)$ is the boundary of $B(x,r)$, $\nu$ is the unit exterior normal and 
$\d S_r$ is the area form of the sphere $S(x,r)$.

In order to compute the last integral consider the unit tangent sphere at $x$,
$$S^{d-1}=\{v\in\mathbb R^{d,1}|\langle v,v\rangle_-=1, \langle v,x\rangle_-=0\}~,$$ and the parameterization
given by the exponential map:
\[
    \sigma:S^{d-1}\rightarrow S(x,r)~,\qquad \sigma(v)=\ch(r) x+\sh(r) v~.
\]
We have $\sigma^*(\d S_r)=\sh (r)^{d-1} \d S^{d-1}$. From
$\nu(\sigma(v))=\sh(r)x+\ch(r)v$  we deduce that
\[
  D\bar f(x)[\nu]=DF(x)[\nu]=F(\nu)=\sh(r) F(x)+\ch(r)F(v)=\sh(r)\bar f(x)+\ch(r)F(v),
\]
so we get
\[
\hat f_r(x)=\frac{1}{\omega_{d-1}\sh(r)^d}\int_{S^{d-1}}(\sh(r)^d \bar f(x)+\sh (r)^{d-1}\ch(r)F(v))\d S^{d-1}~.
\]
Since $F$ is linear $\int_{S^{d-1}} F\d S^{d-1}=0$, so the previous formula shows that
$\hat f_r(x)=\bar f(x)$.
\end{proof}

Let us consider now the regularizing properties of the average process. 
Until  the end of  the section we use the 
 notation  $\langle\cdot,\cdot\rangle$ to denote the metric on $\H^d$.

\begin{lemma}\label{lem:c1app}
If $\oh$ is a continuous function, then $\hat h_r$ is $C^1$ and given a tangent 
vector $v$ at $x$,  we have
\[
  D\hat h_r(x)[v]=\const(r)\int_{S(x,r)}\oh(y) \langle \nu(y), \tilde V(y)\rangle \d S_r(y)
\]
where $\tilde V$ is any Killing vector field such that $\tilde V(x)=v$.
\end{lemma}
\begin{proof}
First we assume that $\oh$ is $C^1$. 
Let $\gamma_t$ the isometric flow generated by $\tilde V$.
We have to compute
\[
\lim_{t\rightarrow 0} \frac{\hat h_r(\gamma_t(x))-\hat h_r(x)}{t}=
\lim_{t\rightarrow 0}\frac{1}{t} \const(r)\int_{B(x,r)} (\oh(\gamma_t y)-\oh(y)) \d\mathbb H^d(y)~.
\]
This limit clearly exists and we have
\[
  D\hat h_r(x)[v]=\const(r)\int_{B(x,r)} \langle\grad^{\H} \oh(y), \tilde V(y)\rangle \d\mathbb H^d(y)~. 
\]
Since $\tilde V$ is a Killing vector field, it is divergence free, so 
$\langle\grad^{\H} \oh, \tilde V\rangle=\operatorname{div}(\oh V)$. Applying the divergence theorem
we obtain a proof of the formula. Notice that $D\hat h_r$ continuously
depends on the point, so $\hat h_r$ is $C^1$.

Consider now the case where $\oh$ is only $C^0$.
Take a sequence of $C^1$-functions $\oh_n$ converging to $\oh$ uniformly on compact subsets of $\H^d$. %in the $C^0$-toplogy.
We have that $ ((\hat{h}_n)_r)$ converges to $\hat h_r$ as $n\rightarrow+\infty$ uniformly on compact subsets.
So in order to conclude it is sufficient to notice that 
$D(\hat{h}_n)_r(x)$ converges to the $1$-form
\[
 \alpha_r(x)[v]=\int_{S(x,r)}u(y)\langle \nu(y), \tilde V(y)\rangle \d S_r(y)~.
\]
It follows that $\hat h_r$ is differentiable and $D\hat h_r=\alpha$.
\end{proof}

\begin{remark}\label{rk:highorder}
{ \rm
In the proof we have showed in particular that if $\oh$ is $C^1$ then
\[
  D\hat h_r(x)[v]=\const(r)\int_{B(x,r)}\langle\grad^{\H} \oh(y), \tilde V(y)\rangle \d\mathbb H^d(y)~.
\]
Applying  Lemma~\ref{lem:c1app} to the continuous function $\oh'(x)=\langle\grad^{\H} \oh(x), \tilde V(x)\rangle$,
we see that for any Killing vector field $\tilde V$ the function $x\mapsto D\hat h_r(x)(\tilde V(x))$ is $C^1$. 
Thus   $\hat h_r$ is  $C^2$. 
More generally by an induction argument, one shows that if
$\oh$ is of class $C^k$, then $\hat h_r$ is of class $C^{k+1}$.
}\end{remark}

Finally we have to prove the last property. Let $\mathrm{Lip}_K(\oh)$ be the best Lipschitz constant of $\oh$ on $K$,
and let $\|\oh \|_{\operatorname{Lip}(K)}=\|\oh\|_K+\mathrm{Lip}_K(\oh)$.

\begin{proposition}\label{prop: app apro}
 Let $\oh$ be a support function.
For any convex compact subset $K$ of $\mathbb H^d$, there is a constant
$C=C(\|\oh\|_{\operatorname{Lip}(K)})$ such that $\hat h_r-Cr$ is a support function on $K$.
\end{proposition}

The proof of this proposition is first done assuming that $\oh$ is $C^2$ and then
using an approximation argument.
The proof in the regular case follows by a computation of the Hessian of $\hat h_r$.

\begin{lemma}\label{lem appen approx}
Assume that $\oh$ is $C^2$. Let $v$ be a tangent vector at $x\in\mathbb H^d$.
Then
\[
   \nabla^2(\hat h_r)_x (v,v)=\const(r)\left(\int_{B(x,r)} (\nabla^2 \oh)_y(\tilde
   V_y, \tilde V_y)\d\mathbb H^d(y) \ +\ \int_{B(x,r)}  \langle
   \grad^{\H} \oh_y, \nabla_{\tilde V}\tilde V\rangle \d\mathbb H^d(y)\right)
\]
where $\tilde V$ is the Killing vector field  extending $v$ and  generating a hyperbolic transformation with
axis through $x$ (i.e.,  $\tilde V(x)=v$ and $\nabla_{\tilde V}\tilde V=0$ at x). 
\end{lemma}
\begin{proof}[Proof of Lemma~\ref{lem appen approx}]
Let $\gamma_t$ be the isometry flow generated by $\tilde V$.
The path $c(t)=\gamma_t(x)$ is a geodesic passing through $x$ such that $\dot c=v$.
In particular, we have
\[
     \nabla^2(\hat h_r)_x(v,v)\rangle=\frac{d^2 \hat h_r(c(t))}{dt^2}|_{t=0}~.
\]
Differentiating the formula
\[
    \hat h_r(c(t))=\const(r)\int_{B(x,r)} \oh(\gamma_t y)\d\mathbb H^d(y)
\]
we get
\[
  \frac{d \hat h_r\circ c}{dt}(t)=\const(r)\int_{B(x,r)}\langle \grad^{\H} \oh(\gamma_t y),
  \tilde V(\gamma_t(y))\rangle \d\mathbb H^d(y)~.
\]
The statement is then proved by differentiating again this formula at $t=0$.
\end{proof}

We will need the following lemma. 

\begin{lemma}\label{lem:limsup lipschitz}
Let $(u_n)_n$ be a sequence of convex functions on an open set $O$ of the unit ball $B\subset \R^d$
converging uniformly on any compact set of $O$ to $u$. 
Let $K\subset O$ be a compact set, and $K'\subset O$ a compact set containing $K$ in its interior.
Then
$$\operatorname{Limsup} \|u_n \|_{\operatorname{Lip}(K)}\leq \|u \|_{\operatorname{Lip}(K')}~. $$
\end{lemma}
\begin{proof}
By uniform convergence, $\|u_n \|_{K}$ converge to $\|u \|_K$, which is less than $\|u \|_{K'}$.
So we have to prove that the limit superior of $\mathrm{Lip}_K(u_n)$ is less than $\mathrm{Lip}_{K'}(u)$.

Let $x,y$ be two points   in $K$, and $z$ be the intersection point of the ray from $x$ towards $y$ and $\partial K'$.
So $y=(1-\lambda)x+\lambda z$ with $\lambda=\|y-x\|/\|z-x\|$, 
 then by convexity we have
 \begin{equation}\label{eq: base convex lipschitz}
   \frac{u_n(y)-u_n(x)}{\| y-x\|} < \frac{u_n(z)- u_n(x)}{\|z-x\|}~.
 \end{equation}

Let $\epsilon>0$ and suppose by contradiction that 
$u_n$ are not $(\mathrm{Lip}_{K'}(u)+\epsilon)$-Lipschitz on $K$ for $n$ big. This implies that there is two sequences
 $x_n$ and $y_n$ in $K$ such that 

$$        u_n(y_n)-u_n(x_n)> (\mathrm{Lip}_{K'}(u)+\epsilon)\|y_n-x_n\|~.$$

Let $z_n$ be the intersection point of the ray from $x_n$ towards $y_n$ and $\partial K'$.
By \eqref{eq: base convex lipschitz} we have that
       \begin{equation}\label{eq: base convex lipschitz2} \frac{u_n(z_n)-u_n(x_n)}{\|z_n-x_n\|}>
       \mathrm{Lip}_{K'}(u)+\epsilon~. \end{equation}

Now, up to take a subsequence, $z_n\rightarrow z$ in $\partial K'$ and $x_n\rightarrow x$ in $K$.
Notice in particular that $\|z-x\|>0$.
Since $u_n\rightarrow u$ uniformly on $K'$ we have that
   $$u_n(z_n)-u_n(x_n)\rightarrow u(z)-u(x)$$
 
and by \eqref{eq: base convex lipschitz2}

  $$ \frac{u(z)-u(x)}{\|z-x\|}\geq \mathrm{Lip}_{K'}(u)+\epsilon $$

that contradicts that $u$ is $(\mathrm{Lip}_{K'}(u))$-Lipschitz on $K'$.
\end{proof}

Given a hyperbolic support function $\oh$ let us denote by $h$ the corresponding support function on the ball.
Recall that $h(x)=\lambda(x)\oh(\rad (x))$ where $\lambda(x)=\sqrt{1-\|x\|^2}$ and $\rad:B\to\H^d$ is the radial map.

Using that for each compact subset $K$ of $\H^d$ the map $\rad: \rad^{-1}(K)\to K$ is bi-Lipschitz 
we simply check that there are two constants $\alpha<\beta$ depending on $K$ such that
for any support function $\oh$
\[
    \alpha \|h\|_{\mathrm{Lip(\rad^{-1}}(K))}\leq \|\oh\|_{\mathrm{Lip}(K)}\leq \beta \|h\|_{\mathrm{Lip(\rad^{-1}}(K))}
\]
where the Lipschitz constant of $h$ is computed with respect to the  Euclidean metric of the ball while
the Lipschitz constant of $\oh$ is computed with respect to the hyperbolic metric.

 So Lemma \ref{lem:limsup lipschitz} implies the following corollary.
 
 \begin{corollary}\label{cor:supp}
 Let $(\oh_n)_n$ be a sequence of hyperbolic support functions on an open subset $U\subset\H^d$
converging uniformly on any compact set of $U$ to $\oh$. 
Let $K\subset U$ be a compact set, and $K'\subset U$ a compact set containing $K$ in its interior.
Then there is a constant $C_1=C_1(K')$ such that
$$\operatorname{Limsup} \|\oh_n \|_{\operatorname{Lip}(K)}\leq C_1 \|\oh \|_{\operatorname{Lip}(K')}~. $$
 \end{corollary}

\begin{proof}[Proof of Proposition~\ref{prop: app apro}]
There is a constant $C= C(d)$ such that if $\tilde V$ is any Killing vector field and $x$ a point
in $\mathbb H^d$ such that $\nabla_{\tilde V} \tilde V (x)=0$ then for any $y\in B(x,1)$

$$  \left|\|\tilde V_y\|^2- \|\tilde V_x\|^2\right|<C r,~\qquad\|\nabla_{\tilde{V}} \tilde V(y)\|\leq C\|\tilde V_x\| r$$
where $r=d_{\mathbb H}(x,y)$. This can be seen because as the hyperbolic space has negative curvature, 
the map $f(y)=\|\tilde V_y\|$ is convex, hence locally
Lipschitz, and $f\operatorname{grad}^{\mathbb H}f=-\nabla_{\tilde V} \tilde V$ \cite{BO69}. By applying an isometry, one sees
that the constant  does not depend on $x$.

Let us assume that $\oh$ is $C^2$.
By Cauchy--Schwarz and one of the inequalities above, we get that for every $y\in B(x,r)$
$$\langle \grad^{\H} \oh,\nabla_{\tilde V}\tilde V(y)\rangle \geq - \|\grad^{\H} \oh\|_{B(x,r)} \|\tilde V_x\| C r~.$$

Moreover, $\nabla^2 \oh -\oh g_{\H} \geq 0$, that implies 
$$\nabla^2(\oh)(\tilde V_y,\tilde V_y)-\oh(y) \|\tilde V_x\|^2 \geq -|\oh(y)|\left| \|\tilde V_y\|^2 - \|\tilde V_x\|^2\right|\geq -\|\oh\|_{B(x,r)} C r~. $$

Lemma~\ref{lem appen approx} implies that
\[
  \nabla^2(\hat h_r)-\hat h_r g_{\H}( v, v) \geq
  -\const(r)V(r) C  (\|\oh\|_{B(x,r)}+\|\grad^{\H} \oh\|_{B(x,r)} \|v\|)r~.
 \]
So in particular for any  $x\in K$, if $\|v\|=1$,  then
\[
(\nabla^2(\hat h_r)-\hat h_r g_{\H}) (v, v)\geq
- \const(r)V(r)C(\|\oh\|_{K_r }+\|\grad^{\H} \oh\|_{K_r})r
\]
and as $K_r$ is convex, $\|\grad \oh\|_{K_r}=\operatorname{Lip}_{K_r}(\oh)$, and 
as $ V(r)\rightarrow 1$ when $r\rightarrow 0$, there exists $C'$ such that 
\[
(\nabla^2(\hat h_r)-\hat h_r g_{\H}) (v, v) \geq
-C' \|\oh\|_{\operatorname{Lip}(K_r)}r\,,\quad ||v||=1\,,
\]
so $ \hat h_r -C' \|\oh\|_{\operatorname{Lip}(K_r)}r$ is a support function.

Now consider the case where $\oh$ is only continuous.
Let $(H_n)_n$ be a sequence of $C^2_+$ support functions 
converging to the extension $H$ of $\oh$
uniformly on compact subset of $I^+(0)$ (Lemma~\ref{lem: conv unif sur compt}).
From Corollary~\ref{cor:supp} %Lemma~\ref{lem:limsup lipschitz},
 for any $\epsilon>0$ the functions
\[
   f_n= (\hat{h}_n)_r-C'C_1(\|\oh\|_{\operatorname{Lip}(K_{2r})}+\epsilon)r
\]
are support functions for $n$ big and $r<1$, 
where $C_1=C_1(K_{1})$ is the constant of Corollary~\ref{cor:supp}.

Taking the $1$-homogeneous extension $F_n$ of $f_n$, we have that $F_n$ is convex over the cone on $K$  for $n$ big.
Letting $n$ go to $+\infty$, $F_n$ converges to the $1$-homogeneous extension of 
\[
  f=\hat h_r-C'C_1(\|\oh\|_{\operatorname{Lip}(K_r)}+\epsilon)r~.
\]

So $f$ is a support function on $K$. Since $\epsilon$ can be chose arbitrarily we deduce that
$\hat h_r-C'C_1\|\oh\|_{\operatorname{Lip}(K_r)}r$ is a support function over $K$.
\end{proof}

%%%%%%%%%%%%%%%%%%%%%%%%%%%%%%%%%%%%%%%%%%%%%%%%%%%%%%%%%%%%%%%%%%%%%%%%%%%%%%%%%%%%%%%%%%%%%%%%%%%%%%%%%%%
%%%%%%%%%%%%%%%%%%%%%%%%%%%%%%%%%%%%%%%%%%%%%%%%%%%%%%%%%%%%%%%%%%%%%%%%%%%%%%%%%%%%%%%%%%%%%%%%%%%%%%%%%%%

%\printindex

\begin{spacing}{0.9}
\begin{footnotesize}
\bibliography{minkowski}

\def\cprime{$'$} \def\cprime{$'$}
\begin{thebibliography}{{Bel}15}

\bibitem[Ale37]{Ale37}
A.~D. Alexandrov.
\newblock On the theory of mixed volumes {II}.
\newblock {\em Mat. Sbornik}, 44:1205--1238, 1937.
\newblock (Russian. Translated in \cite{Ale96}).

\bibitem[Bar05]{Bar05}
T.~Barbot.
\newblock Globally hyperbolic flat space-times.
\newblock {\em J. Geom. Phys.}, 53(2):123--165, 2005.

\bibitem[Bar10]{Bar10}
T.~Barbot.
\newblock Domaines globalement hyperboliques de l'espace de {M}inkowski et de
  l'espace anti-de {S}itter.
\newblock In {\em Alg{\`e}bre, dynamique et analyse pour la g{\'e}om{\'e}trie :
  aspects r{\'e}cents}, pages 101--138. Ellipses, 2010.
\newblock Proceedings des {\'E}coles de {G}{\'e}om{\'e}trie et {S}yst{\`e}mes
  dynamiques, {A}lg{\'e}rie, 2004-2007.

\bibitem[BBZ11]{BBZ}
T.~Barbot, F.~B{\'e}guin, and A.~Zeghib.
\newblock Prescribing {G}auss curvature of surfaces in 3-dimensional
  spacetimes: application to the {M}inkowski problem in the {M}inkowski space.
\newblock {\em Ann. Inst. Fourier (Grenoble)}, 61(2):511--591, 2011.

\bibitem[Bel13]{belt}
Mehdi Belraouti.
\newblock {\em {Asymptomatic convergence of level sets of quasi-concave times
  in a space-time of constant curvature}}.
\newblock Theses, {Universit{\'e} d'Avignon}, June 2013.

\bibitem[{Bel}15]{bel}
M.~{Belraouti}.
\newblock {Asymptotic behavior of Cauchy hypersurfaces in constant curvature
  space-times}.
\newblock {\em ArXiv e-prints}, March 2015.

\bibitem[Ber14]{Ber14}
J.~Bertrand.
\newblock Prescription of {G}auss curvature on compact hyperbolic orbifolds.
\newblock {\em Discrete Contin. Dyn. Syst.}, 34(4):1269--1284, 2014.

\bibitem[BF87]{BF87}
T.~Bonnesen and W.~Fenchel.
\newblock {\em Theory of convex bodies}.
\newblock BCS Associates, Moscow, ID, 1987.
\newblock Translated from the German and edited by L. Boron, C. Christenson and
  B. Smith.

\bibitem[BMS14]{BMS13}
Francesco Bonsante, Catherine Meusburger, and Jean-Marc Schlenker.
\newblock Recovering the geometry of a flat spacetime from background
  radiation.
\newblock {\em Ann. Henri Poincar\'e}, 15(9):1733--1799, 2014.

\bibitem[BO69]{BO69}
R.~L. Bishop and B.~O'Neill.
\newblock Manifolds of negative curvature.
\newblock {\em Trans. Amer. Math. Soc.}, 145:1--49, 1969.

\bibitem[Bon05]{Bon05}
F.~Bonsante.
\newblock Flat spacetimes with compact hyperbolic {C}auchy surfaces.
\newblock {\em J. Differential Geom.}, 69(3):441--521, 2005.

\bibitem[Car04]{car04}
G.~Carlier.
\newblock On a theorem of {A}lexandrov.
\newblock {\em J. Nonlinear Convex Anal.}, 5(1):49--58, 2004.

\bibitem[CBG69]{GH69}
Y.~Choquet-Bruhat and R.~Geroch.
\newblock Global aspects of the {C}auchy problem in general relativity.
\newblock {\em Comm. Math. Phys.}, 14:329--335, 1969.

\bibitem[CY77]{CY77}
S.~Y. Cheng and S.~T. Yau.
\newblock On the regularity of the {M}onge-{A}mp\`ere equation {${\rm
  det}(\partial ^{2}u/\partial x_{i}\partial sx_{j})=F(x,u)$}.
\newblock {\em Comm. Pure Appl. Math.}, 30(1):41--68, 1977.

\bibitem[Fil13]{Fil12}
F.~Fillastre.
\newblock Fuchsian convex bodies: basics of {B}runn-{M}inkowski theory.
\newblock {\em Geom. Funct. Anal.}, 23(1):295--333, 2013.

\bibitem[FV16]{fv}
F.~{Fillastre} and G.~{Veronelli}.
\newblock {Lorentzian area measures and the Christoffel problem}.
\newblock {\em Ann. Sc. Norm. Super. Pisa Cl. Sci.}, XVI(5):383--467, 2016.

\bibitem[GJS06]{GJS06}
B.~Guan, H.-Y. Jian, and R.~Schoen.
\newblock Entire spacelike hypersurfaces of prescribed {G}auss curvature in
  {M}inkowski space.
\newblock {\em J. Reine Angew. Math.}, 595:167--188, 2006.

\bibitem[Gut01]{Gu01}
C.~Guti{\'e}rrez.
\newblock {\em The {M}onge-{A}mp\`ere equation}.
\newblock Progress in Nonlinear Differential Equations and their Applications,
  44. Birkh\"auser Boston Inc., Boston, MA, 2001.

\bibitem[HJ90]{HJ90}
R.~Horn and C.~Johnson.
\newblock {\em Matrix analysis}.
\newblock Cambridge University Press, Cambridge, 1990.
\newblock Corrected reprint of the 1985 original.

\bibitem[HUL93]{HUL93}
J.-B. Hiriart-Urruty and C.~Lemar{\'e}chal.
\newblock {\em Convex analysis and minimization algorithms. {I}}, volume 305 of
  {\em Grundlehren der Mathematischen Wissenschaften [Fundamental Principles of
  Mathematical Sciences]}.
\newblock Springer-Verlag, Berlin, 1993.
\newblock Fundamentals.

\bibitem[KT14]{KT13}
Askold Khovanski{\u\i} and Vladlen Timorin.
\newblock On the theory of coconvex bodies.
\newblock {\em Discrete Comput. Geom.}, 52(4):806--823, 2014.

\bibitem[Li95]{li95}
A.~M. Li.
\newblock Spacelike hypersurfaces with constant {G}auss-{K}ronecker curvature
  in the {M}inkowski space.
\newblock {\em Arch. Math. (Basel)}, 64(6):534--551, 1995.

\bibitem[Mes07]{Mes07}
G.~Mess.
\newblock Lorentz spacetimes of constant curvature.
\newblock {\em Geom. Dedicata}, 126:3--45, 2007.

\bibitem[NS85]{NS85}
I.~G. Nikolaev and S.~Z. Shefel{\cprime}.
\newblock Convex surfaces with positive bounded specific curvature, and a
  priori estimates for {M}onge-{A}mp\`ere equations.
\newblock {\em Sibirsk. Mat. Zh.}, 26(4):120--136, 205, 1985.

\bibitem[OS83]{OS83}
V.~I. Oliker and U.~Simon.
\newblock Codazzi tensors and equations of {M}onge-{A}mp\`ere type on compact
  manifolds of constant sectional curvature.
\newblock {\em J. Reine Angew. Math.}, 342:35--65, 1983.

\bibitem[Pog73]{Pog73}
A.~V. Pogorelov.
\newblock {\em Extrinsic geometry of convex surfaces}.
\newblock American Mathematical Society, Providence, R.I., 1973.
\newblock Translated from the Russian by Israel Program for Scientific
  Translations, Translations of Mathematical Monographs, Vol. 35.

\bibitem[Pog78]{Pog78}
A.~V. Pogorelov.
\newblock {\em The {M}inkowski multidimensional problem}.
\newblock V. H. Winston \& Sons, Washington, D.C., 1978.
\newblock Translated from the Russian by Vladimir Oliker, Introduction by Louis
  Nirenberg, Scripta Series in Mathematics.

\bibitem[Rat06]{Rat06}
J.~Ratcliffe.
\newblock {\em Foundations of hyperbolic manifolds}, volume 149 of {\em
  Graduate Texts in Mathematics}.
\newblock Springer, New York, second edition, 2006.

\bibitem[Roc97]{Roc97}
T.~Rockafellar.
\newblock {\em Convex analysis}.
\newblock Princeton Landmarks in Mathematics. Princeton University Press,
  Princeton, NJ, 1997.
\newblock Reprint of the 1970 original, Princeton Paperbacks.

\bibitem[RT77]{RT77}
J.~Rauch and B.~A. Taylor.
\newblock The {D}irichlet problem for the multidimensional {M}onge-{A}mp\`ere
  equation.
\newblock {\em Rocky Mountain J. Math.}, 7(2):345--364, 1977.

\bibitem[RW98]{RW98}
T.~Rockafellar and R.~Wets.
\newblock {\em Variational analysis}, volume 317 of {\em Grundlehren der
  Mathematischen Wissenschaften [Fundamental Principles of Mathematical
  Sciences]}.
\newblock Springer-Verlag, Berlin, 1998.

\bibitem[Sca00]{SCAN00}
K.~Scannell.
\newblock Infinitesimal deformations of some {${\rm SO}(3,1)$} lattices.
\newblock {\em Pacific J. Math.}, 194(2):455--464, 2000.

\bibitem[Sch93]{sch93}
R.~Schneider.
\newblock {\em Convex bodies: the {B}runn-{M}inkowski theory}, volume~44 of
  {\em Encyclopedia of Mathematics and its Applications}.
\newblock Cambridge University Press, Cambridge, 1993.

\bibitem[TW08]{TW08}
N.~Trudinger and X.-J. Wang.
\newblock The {M}onge-{A}mp\`ere equation and its geometric applications.
\newblock In {\em Handbook of geometric analysis. {N}o. 1}, volume~7 of {\em
  Adv. Lect. Math. (ALM)}, pages 467--524. Int. Press, Somerville, MA, 2008.

\end{thebibliography}
\bibliographystyle{alpha}
\end{footnotesize}
\end{spacing}

\end{document}